\documentclass[12pt]{amsart}
\usepackage{txfonts}
\usepackage{mathrsfs}
\usepackage{amsmath}
\usepackage{amssymb, amsmath}
\usepackage{graphicx}
\allowdisplaybreaks[3]

\newcommand{\newcom}{\newcommand}

\newcom{\al}{\alpha}
\newcom{\be}{\beta}
\newcom{\eps}{\epsilon}
\newcom{\veps}{\varepsilon}
\newcom{\e}{\varepsilon}
\newcom{\de}{\delta}
\newcom{\ga}{\gamma}
\newcom{\Ga}{\Gamma}
\newcom{\ka}{\kappa}
\newcom{\Lam}{\Lambda}
\newcom{\lam}{\lambda}
\newcom{\Om}{\Omega}
\newcom{\om}{\omega}
\newcom{\Si}{\Sigma}
\newcom{\si}{\sigma}
\newcom{\tht}{\theta}
\newcom{\rth}{\mathring\theta}
\newcom{\bth}{\bar\theta}
\newcom{\dtri}{\nabla}
\newcom{\tri}{\triangle}
\newcom{\oo}{\infty}
\newcom{\vphi}{\varphi}
\newcom{\nna}{\big\langle \nabla \big\rangle}
\newcom{\cB}{{\mathcal B}}
\newcom{\cC}{{\mathcal C}}
\newcom{\cD}{{\mathcal D}}
\newcom{\cF}{{\mathcal F}}
\newcom{\cH}{{\mathcal H}}
\newcom{\cL}{{ S}}
\newcom{\cM}{{\mathcal M}}
\newcom{\cN}{{\mathcal N}}
\newcom{\cP}{{\mathcal P}}
\newcom{\cS}{{\mathcal S}}
\newcom{\cQ}{{\mathcal Q}}
\newcom{\cT}{{\mathcal T}}
\newcom{\cY}{{\mathcal Y}}
\newcom{\cZ}{{\mathcal Z}}
\newcom{\R}{\mathbb R}
\newcom{\bbT}{{\mathbb{T}}}
\newcom{\BT}{{\mathbb{T}}^2}
\newcom{\Z}{\mathbb Z}
\newcom{\C}{\mathbb C}
\newcom{\E}{\mathbb E}

\newcom{\hha}{\hat{\mathbf h}}
\newcom{\ha}{\hat{h}}
\newcom{\ul}{\underline}
\newcommand{\vc}[1]{{\mathbf #1}}
\newcom{\ve}{\vc{e}}
\newcom{\vN}{\vc{N}}
\newcom{\vn}{\vc{n}}
\newcom{\vt}{\vc{t}}
\newcom{\ut}{\underline\vt}
\newcom{\vG}{\vc{G}}
\newcom{\vF}{\vc{F}}
\newcom{\vf}{\vc{f}}
\newcom{\vg}{\vc{g}}
\newcom{\vq}{\vc{q}}
\newcom{\vu}{\vc{u}}
\newcom{\vv}{\vc{v}}
\newcom{\vw}{\vc{w}}
\newcom{\vW}{\vc{W}}
\newcom{\vX}{\vc{X}}
\newcom{\vb}{\vc{b}}
\newcom{\vh}{\vc{h}}
\newcom{\vz}{\vc{z}}
\newcom{\fs}{\mathfrak{s}}
\newcom{\vup}{\vu^{+}}
\newcom{\vum}{\vu^{-}}
\newcom{\vvp}{\vv^{+}}
\newcom{\vvm}{\vv^{-}}
\newcom{\vbp}{\vb^{+}}
\newcom{\vbm}{\vb^{-}}
\newcom{\vhp}{\vh^{+}}
\newcom{\vhm}{\vh^{-}}
\newcom{\Omp}{{\Om^+}}
\newcom{\Omm}{{\Om^-}}
\newcom{\vupm}{{\vu^{\pm}}}
\newcom{\vvpm}{{\vv^{\pm}}}
\newcom{\vbpm}{{\vb^{\pm}}}
\newcom{\vhpm}{{\vh^{\pm}}}
\newcom{\vwp}{{\vc{w}^+}}
\newcom{\vwm}{{\vc{w}^-}}
\newcom{\vwpm}{{\vc{w}^{\pm}}}
\newcom{\Ompm}{{\Omega^{\pm}}}
\newcom{\vom}{\boldsymbol{\omega}}
\newcom{\vvap}{\vc{\varpi}}
\newcom{\vop}{\vom^{+}}
\newcom{\vnu}{\vc{\nu}}
\newcom{\vopm}{\vom^{\pm}}
\newcom{\vjp}{\vj^+}
\newcom{\vjm}{\vj^-}
\newcom{\vjpm}{\vj^{\pm}}
\newcom{\vj}{\boldsymbol{\xi}}
\newcom{\Ds} {\langle\nabla\rangle^{s-\f12}}

\newcom{\ds}{{\rm d} s}

\newcom{\f}{\frac}
\newcom{\di}{\displaystyle\int}
\newcom{\dl}{\displaystyle\lim}
\newcom{\ov}{\overline}
\newcom{\sset}{\subset}
\newcom{\wt}{\widetilde}
\newcom{\pa}{\partial}
\newcom{\p}{\partial}
\newcom\na{\nabla}
\newcom{\suml}{\sum\limits}
\newcom{\supl}{\sup\limits}
\newcom{\intl}{\int\limits}
\newcom{\infl}{\inf\limits}
\newcom{\disp}{\displaystyle}
\newcom{\non}{\nonumber}
\newcom{\no}{\noindent}
\newcom{\QED}{$\square$}

\def\div{\mathop{\rm div}\nolimits}

\def\eqdefa{\buildrel\hbox{\footnotesize def}\over =}

\newtheorem{athm}{\bf \t}[section]
\newenvironment{thm} [1] {\def\t{#1}\begin{athm} \bf \rm} {\end{athm}}
\newcom{\bthm}{\begin{thm}}\newcom{\ethm}{\end{thm}}
\newtheorem{theorem}{Theorem}[section]
\newtheorem{lemma}{Lemma}[section]
\newtheorem{remark}{Remark}[section]

\newtheorem{proposition}{Proposition}[section]

\numberwithin{equation}{section}

\begin{document}
\title[Immersed Boundary Problem with Bending and Stretching Energy]{Stability of the Stokes Immersed Boundary Problem with Bending and Stretching Energy}
\author{Hui Li}
\address{Department of Mathematics, Zhejiang University, Hangzhou 310027, China}
\email{lihui92@zju.edu.cn}
\begin{abstract}
	We study the motion of a 1-D closed elastic string with bending and stretching energy immersed in a 2-D Stokes flow. In this paper we introduce the curve’s tangent angle function and the stretching function to describe the deferent deformations of the elastic string. These two functions are defined on the arc-length coordinate and the material coordinate respectively. With the help of the fundamental solution of the Stokes equation, we reformulate the problem into a parabolic system which is called the contour dynamic system. Under the non-self-intersecting and well-stretched assumptions on initial configurations, we establish the local well-posedness of the free boundary problem in Sobolev space. When the initial configurations are sufficiently close to the equilibrium state (i.e. an evenly parametrized circle), we prove that the solutions can be extended globally and the global solutions will converge to the equilibrium state exponentially as t $\to +\infty$.
\end{abstract}
\maketitle
\section{Introduction}
\subsection{Presentation of the Problem}
This paper is concerned with the hydrodynamics on the moving surface of a bilayer membrane immersed in a 2-D Stokes flow. Bilayer membranes are the outer layer of living cells whose thickness is much smaller than the length scale of the cell. The membranes undergo two different elastic deformations: bending and stretching. In general, we ignore the inertia of the membrane, and regard the membrane as a mathematical surface. 

The membranes we considered equip with the following free energy:
\begin{align*}
	E=&\underbrace{\int_\Gamma\left(\frac{\mathfrak c_{1}}{2}(H-B)^{2}+\mathfrak c_{2} K\right) d \sigma}_{e_{H}}+\underbrace{\int_{\Gamma_{op}}\frac{\mathfrak c_3}{2}(\frac{\Delta a}{a})^2 d \sigma}_{e_s}+\int_\Gamma\lambda \mathrm{d} \sigma.
\end{align*}
Here $e_H$ is bending energy or Hefrich energy, $\Gamma$ is the surface representing the membrane, $H$ and $K$ are the mean curvature and the Gaussian curvature respectively, $B$ is the spontaneous curvature that reflects the initial or intrinsic curvature of the membrane, $\mathfrak c_1$ and $\mathfrak c_2$ are the elastic coefficients, and $d\sigma$ is the area form of the surface. $e_s$ is the stretching energy, $\Gamma_{op}$ is the optimal surface which has no stretching deformation, $\mathfrak c_3$ is the elastic modulus of stretching, $\frac{\Delta a}{a}$ is the relative change per unit area. $\lambda$ denotes the surface tension.

The purpose of this paper is to study the motion of a membrane in 2-D Stokes flow. We regard the bilayer membrane as a 1-D elastic string $\Gamma_t$ which is a Jordan curve parametrized by $\vX(s,t)$, where $s\in\bbT$ is the material coordinate (or the Lagrangian coordinate), $\bbT\eqdefa \mathbb R/2\pi\mathbb Z$ is the 1-D torus, and $t\ge0$ is the time variable.  Thus the free energy can be rewritten as
\begin{align*}
	E=\frac{\mathfrak c_1}{2}\int_\bbT (\kappa-B)^2|\vX_s|ds+\frac{\mathfrak c_3}{2}\int_\bbT(|\vX_s|-\fs_{op})^2ds+\lambda \int_\bbT|\vX_s|ds.
\end{align*}
Here we use $\kappa$ to denote the curvature of the string, and $\fs_{op}$ to denote $\frac{1}{2\pi}$ of the perimeter of the optimal string $\Gamma_{op}$.

In this paper, we only consider the case with $B\equiv0$ and $\fs_{op}=0$. We choose $\mathfrak c_{1}=\mathfrak c_3=1$ for simplicity, and assume $\lambda\ge0$ to be a constant. Therefore, the elastic force applied on the string has the following formulation:
\begin{align}\label{eq:force}
	\widetilde\vF(s,t)=&\big(\lambda \kappa-\Delta_{\Gamma_t} \kappa-\frac{1}{2}\kappa^3\big)\vn+\frac{1}{|\vX_s|}\pa_s^2\vX\eqdefa\frac{1}{|\vX_s|}\vF(s,t),
\end{align}
where $\Delta_{\Gamma_t}$ is the Beltrami-Laplace operator on the boundary and $\vn$ is the inward unit normal vector.

We now introduce the precise mathematical statement of the problem we are interested in. Let $\Omega_t=\mathbb R^2/\Gamma_t$, the velocity field $\vu$ and pressure $p$ satisfy the following system:
\begin{equation}\label{eq1}
  \left\{
  	\begin{array}{ll}
  		\Delta\vu(x,t)=\nabla p(x,t),&(x,t)\in \Omega_t\times \mathbb R_+,\\
  		\nabla\cdot\vu(x,t)=0,&(x,t)\in \Omega_t\times \mathbb R_+,\\
  		\lbrack-p(\vX(s,t),t)\vc{I}+\tau(\vX(s,t),t)\rbrack\cdot\vn=\widetilde\vF(s,t),&(s,t)\in \bbT\times \mathbb R_+,\\
  		\lbrack\vu\rbrack(\vX(s,t),t)=0,&(s,t)\in \bbT\times \mathbb R_+,\\
  		\vu(\vX(s,t),t)=\vX_t(s,t),&(s,t)\in\bbT\times \mathbb R_+ ,\\
  		|\vu|,|p|\to0,&\text{as}\ |x|\to\infty,
  	\end{array}
  \right.
\end{equation}
where $\tau=\nabla\vu+{\nabla\vu}^\top$ is the stress of the bulk fluid, $\widetilde\vF$ is given in \eqref{eq:force}, $\lbrack\cdot\rbrack$ denotes the jump of function across the free boundary. The kinematic equation of the string $\vu(\vX(s,t),t)=\vX_t(s,t)$ means that the string moves along the flow. This system can be rewritten in the immersed boundary formulation:
\begin{equation}\label{eq2}
  \left\{
  	\begin{array}{ll}
  		-\Delta\vu(x,t)+\nabla p(x,t)=\vf(x,t),&(x,t)\in \mathbb R^2\times \mathbb R_+,\\
  		\div \vu(x,t)=0,&(x,t)\in \mathbb R^2\times \mathbb R_+,\\
  		\vu(\vX(s,t),t)=\vX_t(s,t),&(s,t)\in\bbT\times \mathbb R_+ ,\\
  		|\vu|,|p|\to0,&\text{as}\ |x|\to\infty,
  	\end{array}
  \right.
\end{equation}
where
\begin{align*}
	\vf(x,t)=\int_{\bbT} \vF(s,t)\delta(x-\vX(s,t))ds,
\end{align*}
and $\delta$ is the 2-D Dirac measure. The first equation of \eqref{eq2} holds in the sense of distribution, and the expression of $\vf$ shows that the force is only applied on the string. The immersed boundary formulation was initially introduced by Peskin \cite{Pe}. It is easy to verify that \eqref{eq2} is equivalent to \eqref{eq1} if both $\vz$ and $\vX$ are sufficiently smooth \cite{LL}.

In 2-D Stokes flow, the velocity field $\vu$ and pressure $p$ can be solved from the force $\vf$ by using boundary integral. It holds that
\begin{align*}
	\vu(x,t)=\int_{\mathbb R^2}G(x-y)\vf(y,t)dy,\ p(x,t)=\int_{\mathbb R^2}Q(x-y)\vf(y,t)dy,
\end{align*}
where 
	\begin{align*}
		G(x)=\frac{1}{4\pi}(-\ln|x|Id+\frac{x\otimes x}{|x|^2}),\quad Q(x)=\frac{x}{2\pi|x|^2}
	\end{align*}
are the fundamental solutions \cite{Po}, and $Id$ is the $2\times2$ identity matrix. The above formula shows that $\vu$ and $p$ are determined by the configuration of the string. Since $G(x)$ is a single-layer potential, $\vu$ is continuous on $\mathbb R^2$. It follows that
\begin{align}\label{eq:vu}
	\vu(\vX(s,t),t)=\int_\bbT G(\vX(s,t)-\vX(s',t))\cdot \vF(s',t) ds'.
\end{align}
On the other hand, it holds that $\vX_t(s,t)=\vu(\vX(s,t),t)$, and the fluid velocity on the free boundary also determines the evolution of the membrane's configuration. As a result, system \eqref{eq2} is equivalent to the following equation:
\begin{align}\label{eq3}
 	\vX_t(s,t)=&\int_\bbT G(\vX(s,t)-\vX(s',t))\cdot \vF(s',t) ds'.
\end{align}
\subsection{Related Results}
During the past several decades, several models have been developed to research the elastic membranes with or without surrounding fluid \cite{He,OH,Ko,Lip,CL,CG,HZE,Po2,St,Wa}. There are also many analytic studies on the membrane dynamic problems. Without surrounding fluid, Hu-Song-Zhang \cite{HSZ} and Wang-Zhang-Zhang \cite{WZZ} analyze the dynamics of a membrane in 2-D and 3-D space respectively. They regard the membrane as a coupled system comprising a moving elastic surface and an incompressible membrane fluid. For the fluid-structure interaction problems, Cheng-Coutand-Shkoller \cite{CCS,CS2} obtain the local well-posedness of moving boundary problems which model the motion of a viscous incompressible fluid inside of a nonlinear elastic fluid shell. In their works, both bending, stretching, and inertial force is considered. However, there is no fluid outside the shells, which is different to the immersed boundary problem. More results about fluid-structure interaction problems can be found in \cite{CGM,GH,LR,MC}.

For the immersed boundary problems, Lin-Tong \cite{LT} study the coupled motion of a 1-D closed elastic string immersed in a 2-D Stokes flow. The string they considered equips the following stretching energy
\begin{align*}
	\mathcal E=\frac{1}{2}\int_\bbT|\vX_s|^2ds.
\end{align*}
They prove the local-wellposedness of this model with an arbitrary initial configuration in $H^{5/2}(\bbT)$. Moreover, when the initial string configuration is sufficiently close to an evenly parametrized circular configuration, they also prove that a global-in-time solution uniquely exists, and will converge to the equilibrium configuration exponentially as $t\to+\infty$. The framework they developed is useful in treating immersed boundary problems. The method we used in this paper is inspired from their work. In a parallel work, Mori-Rodenberg-Spirn \cite{MRS} study the same model and establish well-posedness results in low-regularity H$\mathrm{\ddot{o}}$lder spaces. They prove nonlinear stability of equilibrium states with explicit exponential decay estimates, and verify the optimality of which numerically. In \cite{To}, Tong studies the regularized problem of this model, and derives error estimates under various norms. In a very recent work \cite{MP}, Matioc-Prokert study the two-phase Stokes flow by capillarity in full 2-D space. It is a Stokes immersed boundary problem with external force generated from surface tension. The authors establish the local will-posedness and give a criterion for global existence of solutions.

When the elastic membranes interact with inviscid fluids, the model become to the one called the hydroelastic wave. Ambrose-Siegel \cite{AS} get the local well-posedness of 2-D hydroelastic waves. In that model, the external force applied on the surface is generated from the Helfrich energy. Liu-Ambrose \cite{LA} get similar results for the case with inertial force. 
\subsection{Main Results}
Before presenting the results, we first introduce some notations used throughout this paper. 

Recall that we use $\vX(s,t)$ to parametrize the elastic string $\Gamma_t$ which is defined on the material coordinate $s$. We introduce another vector-valued function $\vz(\al,t)$ to parametrize $\Gamma_t$, where $\al\in\bbT$ is the arc-length coordinate. As $\vz(\al,t)$ and $\vX(s,t)$ denote the same string, it holds that 
\begin{align*}
	\vz(\al(s,t),t)=\vX(s,t),\ \al(s,t)=s+y(s,t),\ \forall s\in\bbT;\\
	 |\vz_\al(\al,t)|\eqdefa\mathfrak s(t),\ \forall \al\in\bbT.
\end{align*}
Here $\al(s,t)=s+y(s,t)$ is the transfer function between these two coordinates, $\mathfrak s$ stands for $\frac{1}{2\pi}$ of the perimeter of $\Gamma_t$. We call $y_s(s,t)=\pa_sy(s,t)$ the stretching function which quantifies the stretching deformation of the elastic string.

 We define the following tangent angle function
\begin{align*}
	\theta\eqdefa\arctan(\frac{z^{(2)}_\al}{z^{(1)}_\al}),
\end{align*}
which is the angle between the string’s tangent direction and the horizontal axis. Using $\theta$, we can describe the shape of the string with the perimeter function $\fs$. This idea goes back at least as far as \cite{HLS}. Though $\arctan$ is a multivalued function, we require $\theta$ to be continuous on $[-\pi,\pi]$. We emphasize that $\theta$ is not continuous on $\bbT$ due to $\theta(\pi)-\theta(-\pi)=2\pi$. We denote by $\theta_0$ the tangent angle function from a initial configuration $\vX_0$. In this paper, we always assume $(\theta,y_s,\fs)$ to be tangent angle function, stretching function, and perimeter function corresponding to $\vX$. In the next section, we will show that one can reconstruct $\vX$ from $(\theta,y_s,\fs)$.

Given $\beta_1,\beta_2>0$, we introduce the non-self-intersecting assumption
	\begin{align}\label{con:betaal}
			\frac{1}{|\al_1-\al_2|}\big|\int^{\al_1}_{\al_2}\big(\cos(\theta),\sin(\theta)\big)d\al'\big|\ge \beta_1,\qquad \forall \al_1,\al_2\in\bbT,
	\end{align}
and the well-stretched assumption
	\begin{align}\label{con:betas}
		1+y_s(s,t)\ge \beta_2,\qquad \forall s\in\bbT,
	\end{align}
where $|\al_1-\al_2|$ is the distance between $\al_1$ and $\al_2$ on $\bbT$. Assumption \eqref{con:betaal} ensures the string is a Jordan curve. If \eqref{con:betas} holds, $\al(s,t)$ is a invertible function, and we use $s(\al,t)$ to denote its inverse.

We denote by $||\cdot||_{L^{p}(\bbT)}$, $||\cdot||_{H^{s}(\bbT)}$, $||\cdot||_{\dot H^{s}(\bbT)}$ the Lebesgue norm, the ordinary Sobolev norm, and the homogeneous Sobolev norm on $\bbT$ for the arc-length coordinate $\al$, and $||\cdot||_{l^{p}(\bbT)}$, $||\cdot||_{h^{s}(\bbT)}$, $||\cdot||_{\dot h^{s}(\bbT)}$ for the material coordinate $s$.  When no confusion can arise, we will write 
\begin{align*}
	&||\cdot||_{L^p([0,t];H^{s}(\bbT))}\to||\cdot||_{L^p_t H^{s}},\quad||\cdot||_{H^{s}(\bbT)}\to||\cdot||_{H^{s}},\\
	&||\cdot||_{L^p([0,t];h^{s}(\bbT))}\to||\cdot||_{L^p_t h^{s}},\quad||\cdot||_{h^{s}(\bbT)}\to||\cdot||_{h^{s}}
\end{align*}
for simplicity of notation.

For membrane dynamic problems involving only bending deformation or stretching deformation, one can study the evolution equations of the free surface in the arc-length coordinate or in the material coordinate respectively. However, such method is no longer suitable for membranes in which both bending and stretching deformation occur. If only the arc-length coordinate is used, the information of stretching deformation will be lost, while if only the material coordinated is used, the stabilizing effect of bending deformation is hard to reflect. To overcome this difficulty, we introduce two variables - the tangent angle function $\theta$ and the stretching function $y_s$, describing bending and stretching deformation of the membrane respectively. $\theta$ is defined in the arc-length coordinate, and $y_s$ is defined in the material coordinate, and we observe that the evolution equations of these two functions have favorable structures in their respective coordinates. Based on this idea, we got the following results.
\begin{theorem}\label{thm:locex}(Existence and uniqueness of local-in-time solution)
	Suppose $\vX_0$ is a closed string which satisfies 
	\begin{align*}
		\theta_0-\al\in H^{5/2}(\bbT),\ y_{0s}\in h^{3/2}(\bbT),\ \fs_0\ge c>0.
	\end{align*}
	Furthermore, we assume that \eqref{con:betaal}-\eqref{con:betas} hold for some constants $\beta_1,\beta_2>0$. Then there exists $T>0$ such that the immersed boundary problem \eqref{eq3} admits a unique solution $\vX(s,t)$ satisfying
	\begin{align}
		||\theta||_{L^\infty_T \dot{H}^{5/2}\cap L^2_T\dot{H}^4}+||y_s||_{L^\infty_T \dot h^{3/2}\cap L^2_T\dot h^2}\le& C{||\theta_0||_{\dot{H}^{5/2}}+||y_{0s}||_{\dot h^{3/2}}},\label{eq:regthy}\\
		||\pa_t \theta||_{L^2_T\dot{H}^1}+||\pa_t y_s||_{L^2_T\dot{h}^1}\le& C\big(||\theta_0||_{\dot{H}^{5/2}}+||y_{0s}||_{\dot h^{3/2}}\big),
	\end{align}
	and
	\begin{align}
			\frac{1}{|\al_1-\al_2|}\big|\int^{\al_1}_{\al_2}\big(\cos(\theta(\al',t)),\sin(\theta(\al',t))\big)d\al'\big|\ge \frac{1}{2}\beta_1,\qquad \forall \al_1,\al_2\in\bbT,\ t\in[0,T],\\
		1+y_s(s,t)\ge \frac{1}{2}\beta_2,\qquad\qquad\qquad\qquad\qquad \forall s\in\bbT,\ t\in[0,T],\label{eq:ybeta}
	\end{align}
	where $C$ is a constant depends only on $\fs_0$.
\end{theorem}
\begin{remark}
	In Theorem \ref{thm:locex}, we only consider the function spaces related to variables $\theta$ and $y_s$. In fact, \eqref{eq:regthy} implies that 
	\begin{align*}
		\vz(\al,t)\in L^\infty([0,T];H^{7/2}(\bbT))\cap L^2([0,T];H^{5}(\bbT)),\\
		 \vX(s,t)\in L^\infty([0,T];h^{5/2}(\bbT))\cap L^2([0,T];h^{3}(\bbT)).
	\end{align*}
	One can see that the string has different regularities in different coordinates. This is the reason we introduce both the arc-length and the material coordinates.
\end{remark}

\begin{theorem}\label{thm:global}
	(Existence and uniqueness of global-in-time solution near equilibrium.) There exists a constant $\e>0$ such that, if $\vX_0$ is a closed string and satisfies
	\begin{gather}
		||\theta_0-\al||_{\dot H^{5/2}}+||y_{0s}||_{\dot h^{3/2}}\le\varepsilon,\label{con:1}
	\end{gather}
	then there is a unique solution $\vX\in C([0,+\infty);h^{5/2}(\bbT))\cap L^2_{loc}([0,+\infty);h^3(\bbT))$ of the system \eqref{eq3} with initial data $\vX_0$. The solution satisfies 
	\begin{gather*}
		\theta_t(\al,t)\in L^2_{loc}([0,+\infty);H^1(\bbT)),\ y_{st}(s,t)\in L^2_{loc}([0,+\infty);h^1(\bbT)),\\
		||\theta-\al||_{L^\infty([0,+\infty);\dot H^{5/2})}+||y_s||_{L^\infty([0,+\infty);\dot h^{3/2})}\le C\varepsilon,\\
		\frac{1}{|\al_1-\al_2|}\big|\int^{\al_1}_{\al_2}\big(\cos(\theta(\al',t)),\sin(\theta(\al',t))\big)d\al'\big|\ge \frac{1}{\pi},\qquad \forall \al_1,\al_2\in\bbT,\ t\in[0,+\infty),\\
		1+y_s(s,t)\ge \frac{3}{4},\qquad \forall s\in\bbT,\ t\in[0,+\infty),
	\end{gather*}
	where $C$ is a constant.
\end{theorem}
\begin{theorem}\label{thm:conver}
	(Exponential convergence to the equilibriums). Let $\vX_0$ be a closed string satisfying all the assumptions in Theorem \ref{thm:global}, and let $\vX$ be the global solution obtained in Theorem \ref{thm:global} starting from $\vX_0$. There exist universal constants $\varepsilon,\gamma_*,C>0$ such that if in addition
	\begin{align*}
		||\theta_0-\al||_{\dot H^{5/2}}+||y_{0s}||_{\dot h^{3/2}}\le\varepsilon,
	\end{align*}
	then it holds that 
	\begin{align*}
		||\theta-\al||_{\dot H^{5/2}}(t)+||y_{s}||_{\dot h^{3/2}}(t)\le C e^{-\gamma_*t}\varepsilon.
	\end{align*}
	Furthermore, $\vX$ converges to an equilibrium configuration
	\begin{align*}
	 	\vX_\infty(s)\eqdefa\sqrt{\frac{\mathfrak a}{\pi}}\big(\sin(s+\theta_\infty),\cos(s+\theta_\infty)\big)+x_\infty,\quad s\in\bbT,
	\end{align*}
	and satisfies
	\begin{align*}
		||\vX-\vX_\infty||_{h^{5/2}}\le C\sqrt{\frac{\mathfrak a}{\pi}}e^{-\gamma_*t}\varepsilon,
	\end{align*}
	where $\mathfrak a$ is the area enclosed by $\vX_0$.
\end{theorem}
\begin{remark}
	For other kinds of elastic membranes, such as one or two of $(\mathfrak c_1,\mathfrak c_3,\lambda)$ equals zero, $B$ and $\fs_{op}$ are nonzero positive constants, similar results can be obtained by using the method developed herein.
\end{remark}
The rest of this paper is organized as follows. In Section 2, we reformulate the problem to the contour dynamic system, and give an energy identity of this problem. In Section 3, we introduce an modified contour dynamic system and give a priori estimates of this system. Section 4 provides the local well-posedness of the modified system, and shows that the modified system is equivalent to the original system when the string is closed initially. In Section 5, by using the energy identity, we get the global-in-time existence of solutions to \eqref{eq3} provided that the initial data is sufficiently close to an equilibrium configuration. In Section 6, we observe that when $\vz$ is closed to a circle, the first Fourier modes of $\theta-\al$ are extremely small compared to $||\theta-\al||_{L^2}$. Based on this observation, we prove the global solution gotten in Section 5 converges to an equilibrium configuration exponentially as $t\to+\infty$. Section 7 shows that the method developed in this paper can be applied to other kinds of immersed boundary problems. In the appendices, we state some auxiliary results and give a new proof to the weak version of 2-dimensional Fuglede's isoperimetric inequality.
\section{Reformulation of The Problem}
In this section, we reformulate the immersed boundary problem to the contour dynamic system, and  give an energy identity of this problem. The contour dynamic system is the combination of evolution equations for the tangent angle function, the stretching function, and the perimeter function, in which we can see the stabilization mechanism of the elastic force explicitly.
\subsection{The arc-length coordinate and the material coordinate}
From the definition of arc-length, it holds that
\begin{align}\label{eq:length}
 	|\vz_\alpha(\al,t)|=\mathfrak s(t),
\end{align}
where $\mathfrak s$ is $\frac{1}{2\pi}$ of the perimeter of the string as we defined in the previous section. Let $\vn$ and $\vt$ denote the inward unit normal vector and the unit tangent vector of the free boundary, it holds that
\begin{align*}
	\vn=\frac{\vz_\al^\perp}{\mathfrak s },\quad\vt=\frac{\vz_\al}{\mathfrak s }.
\end{align*}
Here $\vv^\perp=(-v^{(2)},v^{(1)})$ for each two dimensional vector $\vv=(v^{(1)},v^{(2)})$. Applying $\al$-derivation on $\vn$ and $\vt$, it follows that
\begin{align*}
	\vt_\al=\kappa \mathfrak s \vn,\qquad \vn_\al=-\kappa \mathfrak s \vt.
\end{align*}
Here we use the fact that
\begin{align*}
	\kappa\vn=\frac{\vz_{\al\al}}{\mathfrak s ^2}=\frac{\vz_{\al\al}\cdot\vz_\al^\perp}{\mathfrak s ^3}\vn.
\end{align*}
Recalling the definition of the tangent angle function 
\begin{align}\label{eq:arc}
	\theta\eqdefa\arctan(\frac{z^{(2)}_\al}{z^{(1)}_\al}),
\end{align}
one see immediately that
\begin{align}\label{eq:direc}
	\kappa=\frac{\theta_\al}{\mathfrak s },\quad\vn=(-\sin(\theta),\cos(\theta)),\quad \vt=(\cos(\theta),\sin(\theta)).
\end{align}

As what we have mentioned above, the relation between $\vX$ and $\vz$ is
\begin{align}\label{eq:tranf}
	\vX(s,t)=\vz(\al(s,t),t),\qquad \al(s,t)=s+y(s,t),
\end{align}
and it follows that
\begin{align}\label{eq:xtzt}
	\vX_t(s,t)&=\vz_t(s+y(s,t),t)+y_t(s,t)\vz_\al(s+y(s,t),t).
\end{align}
As $\vz_\al=\mathfrak s \vt$, we only have $\vz_t(\al,t)\cdot\vn=\vu(\vz(\al,t),t)\cdot\vn$. That is to say, $\vz_t$ is not the real velocity of the string, and $\vz$ is an abstract curve. We decompose $\vz_t$ into the normal and tangent direction
\begin{align*}
	\vz_t=(\vz_t\cdot\vt) \vt+(\vz_t\cdot\vn)\vn\eqdefa\cT\vt+U\vn.
\end{align*}
Here $U(\al,t)=\vu(\vz(\al,t),t)\cdot\vn(\al,t)$. Differentiating \eqref{eq:length} and \eqref{eq:arc} in time, we get the evolution equations for $\theta$ and $\fs$:
\begin{align}
	\mathfrak s_t&=\frac{\vz_{\al t}\cdot\vz_\al}{\mathfrak s }=\cT_\al-\theta_\alpha U,\label{eq:sa}\\
	\theta_t&=\frac{\vz_{\al t}\cdot\vz_\al^\perp}{\mathfrak s ^2}=\frac{(\vz_{ t}\cdot\vn)_\al}{\mathfrak s }-\frac{\vz_t\cdot\vz_{\al\al}^\perp}{\mathfrak s ^2}=\frac{U_\al}{\mathfrak s }+\frac{\cT}{\mathfrak s }\theta_\al\label{eq:thetaoriginal}.
\end{align}
As $\cT$ is continuous on $\bbT$, it holds that
\begin{align}\label{eq:sa2}
	2\pi\mathfrak s_t=\int^\pi_{-\pi}\cT_\al d\al-\int^\pi_{-\pi}\theta_\alpha Ud\al=-\int^\pi_{-\pi}\theta_\alpha Ud\al.
\end{align}
Integrating \eqref{eq:sa} from $-\pi$ to $\al$, we have
\begin{align}\label{eq:T}
	\cT(\al,t)=\int^\al_{-\pi}\theta_\al(\al') U(\al')d\al'-\frac{\al+\pi}{2\pi}\int^\pi_{-\pi}\theta_\alpha Ud\al+\overline \cT(t).
\end{align}
Here $\overline \cT(t)$ is a scalar function to be determined later. In the material coordinate, we can freely choose the starting point of the arc-length coordinate $\al(-\pi,t)$. It follows from \eqref{eq:xtzt} that
\begin{align}\label{eq:yt}
	y_t(s,t)=\frac{1}{\mathfrak s}\big(\vX_t(s,t)\cdot\vt-\vz_t(s+y(s,t),t)\cdot\vt\big).
\end{align}
We choose
\begin{align}\label{eq:barct}
	\overline \cT(t)=\vX_t(-\pi,t)\cdot\vt(\al(-\pi,t),t),
\end{align}
then we have
\begin{align}\label{eq:ytpi}
	y_t(-\pi,t)=\frac{1}{\mathfrak s}\big(\vX_t(-\pi,t)\cdot\vt-\overline \cT(t)\big)=0.
\end{align}
In this paper, we always assume that $y(-\pi,0)=0$, so that $y(-\pi,t)\equiv0$ and $y(s,t)=\int_{-\pi}^s y_s(s',t)ds'$. We also have $\al(-\pi,t)\equiv-\pi$, which means that at each time, the arc-length coordinate $\al$ started at the same point of the string.
\begin{remark}
	To simplify notation, when no confusion can arise, we will write
	\begin{align*}
		y_s\big(s(\al,t),t\big)\to y_s(\al,t),\qquad \theta\big(\al(s,t),t\big)\to\theta(s,t),
	\end{align*}
	and
	\begin{align*}
		\int_\bbT y_s\big(s(\al,t),t\big)+\theta(\al,t)d\al\to\int_\bbT y_s+\theta d\al,\\
		\int_\bbT y_s(s,t)+\theta\big(\al(s,t),t\big)ds\to\int_\bbT y_s+\theta ds.
	\end{align*}
\end{remark}
\subsection{Velocity of the string}
With the help of the fundamental solution $G$, velocity fields on the string have the following expression
\begin{align}\label{eq:u}
	\vu(\vz(\al,t),t)=\int_\bbT G(\vz(\al,t)-\vX(s',t))\cdot \vF(s',t) ds'.
\end{align}
In the rest of paper, we use $\vu(\al,t)$ and $\vu(s,t)$ to denote $\vu(\vz(\al,t),t)$ and $\vu(\vX(s,t),t)$ respectively. From \eqref{eq:tranf}, we have
\begin{align}
	  		\vX_s(s,t)&=(1+y_s(s,t))\vz_\al(s+y(s,t),t),\label{eq:Xs}\\
	\vX_{ss}(s,t)&=(1+y_s(s,t))^2\vz_{\al\al}(s+y(s,t),t)+y_{ss}(s,t)\vz_\al(s+y(s,t),t).\label{eq:Xss}
\end{align}
The elastic force has the following expression
\begin{align}\label{eq:force-or}
\vF(s,t)=(1+y_s)(\lambda \theta_\al\vn-\frac{\theta_{\al\al\al}\vn}{\fs^2}-\frac{1}{2}\frac{\theta_\al^3\vn}{\fs ^2})(s,t)+\fs \big(y_{ss}\vt+(1+y_s)^2 \theta_\al\vn\big)(s,t).
\end{align}
Then, we rewrite \eqref{eq:u} as follows:
\begin{align}
	\vu(\al,t)=&\int_\bbT G\big(\vz(\al,t)-\vz(\al',t)\big)\cdot\Big(\big(\lambda \theta_\al\vn- \frac{\theta_{\al\al\al}}{\mathfrak s ^2}\vn-\frac{\theta_\al^3}{\mathfrak s ^2}\vn\big)(\al',t)\label{eq:ual}\\
	&\qquad\qquad+\mathfrak s \big(\frac{y_{ss}(s(\al',t),t)\vt}{1+y_s(s(\al',t),t)}+(1+y_s(s(\al',t),t)) \theta_\al(\al',t)\vn\big)\Big)d\al';\nonumber\\
	\vu(s,t)=&\int_\bbT G\big(\vX(s,t)-\vX(s',t)\big)\cdot\Big(\mathfrak s \big(y_{ss}(s',t)\vt+(1+y_s)^2 \theta_\al(\al(s',t),t)\vn\big)\label{eq:us}\\
	&\qquad\qquad+(1+y_s)\big(\lambda \theta_\al\vn- \frac{\theta_{\al\al\al}}{\mathfrak s ^2}\vn-\frac{\theta_\al^3}{\mathfrak s ^2}\vn\big)(\al(s',t),t)\Big)ds'.\nonumber
\end{align}
An easy computation shows that
\begin{align}
	\fs\big(\lambda \frac{\theta_\al}{\mathfrak s }\vn-\frac{\theta_{\al\al\al}}{\mathfrak s ^3}\vn- \frac{1}{2}(\frac{\theta_\al}{\mathfrak s })^3\vn   \big)=&\pa_\al(\lambda \vt-\frac{\theta_{\al\al}\vn}{\mathfrak s ^2}-\frac{1}{2}\frac{\theta_\al^2\vt}{\mathfrak s ^2}),\label{eq:paF1}\\
	\big(y_{ss}\vt+(1+y_s)^2 \theta_\al\vn\big)=&\pa_s\big((1+y_s)\vt\big).\label{eq:paF2}
\end{align}
Therefore, it also holds that
\begin{align}\label{eq:un}
	\vu=&\mathrm{p.v.}\int_\bbT -\frac{\pa}{\pa_{s'}}G(\vX(s,t)-\vX(s',t))\cdot \mathfrak s \big((1+y_s)\vt\big) ds'\\
	&+\mathrm{p.v.}\int_\bbT -\frac{\pa}{\pa_{\al'}}G(\vz(\al,t)-\vz(\al',t))\cdot(\lambda \vt-\frac{\theta_{\al\al}\vn}{\fs^2}-\frac{1}{2}\frac{\theta_\al^2\vt}{\fs^2}) d\al'.\nonumber
\end{align}
These three formulations of $\vu$ will be used in different situations. 

From the definitions of $\vz$ and $\vX$, one can see that
\begin{align}
	\vz(\al,t)-\vz(\al',t)&=\mathfrak s(t)\int^\al_{\al'}\big(\cos(\theta(\al'',t),t),\sin(\theta(\al'',t),t)\big)d\al'',\label{eq:reconvz}\\
	\vX(s,t)-\vX(s',t)&=\mathfrak s(t)\int^{s+\int_{-\pi}^{s}y_s(s'',t)ds''}_{s'+\int_{-\pi}^{s'}y_s(s'',t)ds''}\big(\cos(\theta(\al'',t),t),\sin(\theta(\al'',t),t)\big)d\al''\label{eq:reconvx}.
\end{align}
This indicates that $\vu$ is determined by $(\theta,y_s,\mathfrak s)$ and doesn't depend on the exact position of $\vX$.
\subsection{Contour Dynamic System}
Now, we are in a position to introduce the following contour dynamic system.
\begin{proposition}
	Assume that $\vX(s,t)$ is a closed string which satisfies \eqref{con:betaal}-\eqref{con:betas} for some constants $\beta_1,\beta_2>0$, and $\mathfrak s>0,\ \theta\in H^3(\bbT),\ y_s\in h^1(\bbT)$, for $\forall t\in[0,T]$. The evolution equation of $\vX(s,t)$ in the 2-D Stokes immersed boundary problem \eqref{eq2} is equivalently given by
	\begin{align}
	\theta_t(\al,t)&=\mathcal L(\theta)(\al,t)+g_\theta(\al,t),\quad \theta(\al,0)=\theta_0(\al),\label{eq:th}\\
	y_{st}(s,t)&=\mathfrak L (y_{s})(s,t)+g_y(s,t),\quad y_s(s,0)=y_{0s}(s,0),\label{eq:ys}\\
	\mathfrak s_t(t)&=-\frac{1}{2\pi}\int^\pi_{-\pi}\theta_\alpha \vu\cdot\vn d\al,\quad \mathfrak s(0)=\mathfrak s_0.\label{eq:St}
\end{align}
Here 
\begin{align*}
	\mathcal L(\theta)(\al,t)=\frac{1}{4\mathfrak s^3(t)}\mathcal H(\theta_{\al\al\al})(\al,t),\ \mathfrak L (y_{s})(s,t)=-\frac{1}{4}\mathfrak h(y_{ss})(s,t),
\end{align*}
are two negative operators, $g_\theta$ and $g_y$ are the error terms with the following expressions:
\begin{align*}
	&g_\theta(\al,t)\\
	=&\frac{1}{4\mathfrak s^3}\vn\cdot[\mathcal H,\vn](\theta_{\al\al\al})(\al,t)-\frac{1}{4}\vn\cdot[\mathcal H,\vt](\frac{y_{ss}}{1+y_s})(\al,t)\\
	&-\frac{1}{4}\vn\cdot\mathcal H\big((1+y_s) \theta_\al\vn\big)(\al,t)-\frac{\lambda}{4\fs}\vn\cdot\mathcal H\big(\theta_\al\vn\big)(\al,t)+\frac{1}{8\fs^3}\vn\cdot\mathcal H\big((\theta_\al)^3\vn\big)(\al,t)\\
	&+\vn\cdot\int_\bbT\Big(\frac{\pa}{\pa_\al}G\big(\vz(\al,t)-\vz(\al',t)\big)+\frac{1}{8\pi\tan( \frac{\al-\al'}{2})}Id\Big)\\
	&\qquad\qquad\qquad\cdot\Big(\frac{y_{ss}}{1+y_s}\vt+(1+y_s) \theta_\al\vn+\frac{\lambda\theta_\al}{\mathfrak s }\vn-\frac{\theta_{\al\al\al}}{\mathfrak s ^3}\vn-\frac{1}{2}\frac{(\theta_\al)^3}{\mathfrak s ^3}\vn\Big) d\al'\\
	&+\big(\int_{-\pi}^\al\theta_\al \vu\cdot\vn d\al'-\frac{\al+\pi}{2\pi}\int_{-\pi}^\pi\theta_\al \vu\cdot\vn d\al'-\vu\cdot\vt(\al,t)+\overline \cT(t)\big)\frac{\theta_\al(\al,t)}{\mathfrak s},\\
	\\
	&g_y(s,t)\\
	=&-(1+y_s(s,t))\frac{\mathfrak s_t}{\mathfrak s}-\frac{1}{4}\vt\cdot[\mathfrak h,\vt](y_{ss})(s,t)-\frac{1}{4}\vt\cdot\mathfrak h\big((1+y_{s})^2\theta_\al\vn\big)(s,t)\\
	&+\frac{1}{4\fs^3}\vt\cdot[\mathfrak h,\vn]\big(\pa_s\theta_{\al\al}\big)(s,t)-\frac{\lambda}{4\fs}\vt\cdot\mathfrak h\big((1+y_s)\theta_\al\vn\big)(s,t)+\frac{1}{8\fs^3}\vt\cdot\mathfrak h\big((1+y_s)(\theta_\al)^3\vn\big)(s,t)\\
	&+\vt\cdot\int_\bbT\Big(\frac{\pa}{\pa_s}G\big(\vX(s,t)-\vX(s',t)\big)+\frac{1}{4}\frac{1}{2\pi\tan( \frac{s-s'}{2})}Id\Big)\\
	&\qquad\qquad\qquad\cdot\Big(y_{ss}\vt+(1+y_s)^2\theta_\al\vn+(1+y_s)\big(\frac{\lambda}{\fs}\theta_\al\vn-\frac{\theta_{\al\al\al}}{\mathfrak s ^3}\vn- \frac{1}{2}\frac{(\theta_\al)^3}{\mathfrak s ^3}\vn\big)\Big) ds',
\end{align*}
and $(\mathcal H,\mathfrak h)$ are the Hilbert transform operators on $\bbT$ in the arc-length coordinate and in the material coordinate respectively.
\end{proposition}
\begin{proof}
From \eqref{eq:thetaoriginal} we know that
\begin{align*}
	\theta_t(\al,t)=\frac{1}{\mathfrak s }\big(\vu_\al(\al,t)\cdot\vn+(\cT(\al,t)-\vu(\al,t)\cdot\vt)\theta_\al(\al,t)\big).
\end{align*}
Here $\vu_\al(\al,t)\cdot\vn$ is the most important term. By \eqref{eq:u}, it holds that
\begin{align*}
	&\vu_\al(\al,t)\cdot\vn\\
	=&\vn\cdot \mathrm{p.v.}\int_\bbT \Big(-\frac{\big(\vz(\al)-\vz(\al')\big)\cdot\vz_\al(\al)}{4\pi|\vz(\al)-\vz(\al')|^2}Id\\
	&\qquad-\frac{2\big(\vz(\al)-\vz(\al')\big)\cdot\vz_\al(\al)\big(\vz(\al)-\vz(\al')\big)\otimes\big(\vz(\al)-\vz(\al')\big)}{4\pi|\vz(\al)-\vz(\al')|^4}\\
	&\qquad+\frac{\vz_\al(\al)\otimes\big(\vz(\al)-\vz(\al')\big)+\big(\vz(\al)-\vz(\al')\big)\otimes\vz_\al(\al)}{4\pi|\vz(\al)-\vz(\al')|^2}\Big)\\
	&\qquad\qquad\cdot\Big(\mathfrak s \big(\frac{y_{ss}(s(\al',t),t)\vt}{1+y_s}+(1+y_s) \theta_\al(\al',t)\vn\big)+\big(\lambda \theta_\al\vn- \frac{\theta_{\al\al\al}}{\mathfrak s ^2}\vn-\frac{\theta_\al^3}{\mathfrak s ^2}\vn\big)\Big)d\al'.
\end{align*}
When $|\al'-\al|$ is small, we formally find
\begin{align*}
	\frac{\pa}{\pa_\al}G\big(\vz(\al,t)-\vz(\al',t)\big)\sim \frac{1}{4\pi}\frac{1}{\al'-\al}\sim- \frac{1}{4}\frac{1}{2\pi\tan( \frac{\al-\al'}{2})}.
\end{align*}
The Hilbert transform on $\bbT$ is defined as
\begin{align*}
	\mathcal H Y(\al)=p.v.\int_\bbT \frac{Y(\al')}{2\pi\tan( \frac{\al-\al'}{2})}d\al'.
\end{align*}
Therefore, it follows that
\begin{align}\label{eq:vual}
	&\vu_\al(\al,t)\cdot\vn\\
	=&\frac{1}{4\mathfrak s^2}\mathcal H(\theta_{\al\al\al})+\frac{1}{4\mathfrak s^2}\vn\cdot[\mathcal H,\vn](\theta_{\al\al\al})(\al,t)-\frac{\fs}{4}\vn\cdot[\mathcal H,\vt](\frac{y_{ss}}{1+y_s})(\al,t)\nonumber\\
	&-\frac{\fs}{4}\vn\cdot\mathcal H\big((1+y_s) \theta_\al\vn\big)(\al,t)-\frac{\lambda}{4}\vn\cdot\mathcal H\big(\theta_\al\vn\big)(\al,t)+\frac{1}{8\fs^2}\vn\cdot\mathcal H\big((\theta_\al)^3\vn\big)(\al,t)\nonumber\\
	&+\vn\cdot\int_\bbT\Big(\frac{\pa}{\pa_\al}G\big(\vz(\al,t)-\vz(\al',t)\big)+\frac{1}{8\pi\tan( \frac{\al-\al'}{2})}Id\Big)\nonumber\\
	&\qquad\qquad\qquad\cdot\Big(\frac{\fs y_{ss}}{1+y_s}\vt+\fs(1+y_s)\theta_\al\vn+\lambda\theta_\al\vn-\frac{\theta_{\al\al\al}}{\mathfrak s^2}\vn-\frac{1}{2}\frac{(\theta_\al)^3}{\mathfrak s^2}\vn\Big) d\al'.\nonumber
\end{align}
Here $[\mathcal H,\vn](\theta_{\al\al\al})=\mathcal H\big(\vn\theta_{\al\al\al}\big)-\vn\mathcal H\big(\theta_{\al\al\al}\big)$ is a commutator.

Differentiating \eqref{eq:yt} in $s$, from \eqref{eq:T} we have
\begin{align}\label{eq:yst}
	&y_{st}(s,t)\\
	=&\frac{1}{\mathfrak s }\pa_s\big(\vX_t(s,t)\cdot\vt-\vz_t(s+y(s,t),t)\cdot\vt\big)\nonumber\\
	=&\frac{1}{\mathfrak s }\Big(\vu_s(s,t)\cdot\vt+(1+y_s)\vu(s,t)\cdot\theta_\al\vn-(1+y_s)\theta_\al\vu(s,t)\cdot\vn-(1+y_s)\mathfrak s_t\Big)\nonumber\\
	=&\frac{1}{\mathfrak s }\vu_s(s,t)\cdot\vt-(1+y_s)\frac{\mathfrak s_t}{\mathfrak s}.\nonumber
\end{align}
Similar to \eqref{eq:vual}, it holds that
\begin{align*}
	&\vu_s(s,t)\cdot\vt\\
	=&-\frac{\mathfrak s }{4}\mathfrak h (y_{ss})-\frac{\mathfrak s }{4}\vt\cdot[\mathfrak h,\vt](y_{ss})-\frac{\mathfrak s }{4}\vt\cdot\mathfrak h\big((1+y_{s})^2\theta_\al\vn\big)+\frac{1}{4\fs^2}\vt\cdot[\mathfrak h,\vn]\big(\pa_s\theta_{\al\al}\big)\\
	&-\frac{\lambda}{4}\vt\cdot\mathfrak h\big((1+y_s)\theta_\al\vn\big)+\frac{1}{8\fs^2}\vt\cdot\mathfrak h\big((1+y_s)\theta_\al^3\vn\big)\\
	&+\vt\cdot\int_\bbT\Big(\frac{\pa}{\pa_s}G\big(\vX(s,t)-\vX(s',t)\big)+\frac{1}{4}\frac{1}{2\pi\tan( \frac{s-s'}{2})}Id\Big)\\
	&\qquad\qquad\qquad\cdot\Big(\mathfrak s y_{ss}\vt+\mathfrak s(1+y_s)^2\theta_\al\vn+(1+y_s)\big(\lambda\theta_\al\vn-\frac{\theta_{\al\al\al}}{\mathfrak s ^2}\vn -\frac{1}{2}\frac{(\theta_\al)^3}{\mathfrak s ^2}\vn\big)\Big) ds',
\end{align*}
where $\pa_s\theta_{\al\al}=(1+y_s(s,t))\theta_{\al\al\al}(\al(s,t),t)$ and
\begin{align*}
	\mathfrak h Y(s)=p.v.\int_\bbT \frac{Y(s')}{2\pi\tan( \frac{s-s'}{2})}ds'
\end{align*}
is the Hilbert transform on $\bbT$ in material coordinate. Then, one can deduce \eqref{eq:th}-\eqref{eq:St} immediately.

On the other hand, one can reconstruct $\vX$ from $(\theta,y_s,\mathfrak s)$. Indeed, recalling the definition of arc-length coordinate, we have
\begin{align}\label{eq:z}
	\vz(\al,t)=\vz(-\pi,t)+\mathfrak s\int^\al_{-\pi}\big(\cos(\theta(\al',t)),\sin(\theta(\al',t))\big)d\al'.
\end{align}
From \eqref{eq:xtzt} and \eqref{eq:ytpi}, it is clear that
\begin{align}\label{eq:zt}
	\vz(-\pi,t)=\vz(-\pi,0)+\int_0^t\vz_t(-\pi,t')dt'=\vX(-\pi,0)+\int_0^t\vu(\vX(-\pi,t'),t')dt'.
\end{align}
Consequently, it holds that
\begin{align}\label{eq:X}
	\vX(s,t)=&\vX(-\pi,0)+\int_0^t\vu(\vX(-\pi,t'),t')dt'\\
	&+\mathfrak s\int^{s+\int_{-\pi}^{s}y_s(s'',t)ds''}_{-\pi}\big(\cos(\theta(\al',t)),\sin(\theta(\al',t))\big)d\al'.\nonumber
\end{align}	
This completes the proof.
\end{proof}
In the derivation of \eqref{eq:th} and \eqref{eq:ys}, we extract the linear principal parts $\mathcal L(\theta) $ and $\mathfrak L(y_s)$. This approach is known as small-scale decomposition which is introduced by Beale, Hou, Lowengrub, and Shelley in \cite{BHL,HLS}. In what follows, we shall analyze the properties of the contour dynamic system.
\subsection{Energy Dissipation}
Solutions to the immersed boundary problem satisfy the following energy identity.
\begin{lemma}\label{lem:energy}
	Assume $\vX$ is a solution to \eqref{eq3} with 
	\begin{align*}
		\fs(t)>0,\quad \theta\in L^2([0,T];\dot H^3(\bbT)),\quad y_s\in L^2([0,T];\dot h^2(\bbT))
	\end{align*}
	satisfying \eqref{con:betaal}-\eqref{con:betas} for some constants $\beta_1,\beta_2>0$. It holds that
	\begin{align}\label{eq:energy}
		\frac{d}{dt}\Big(\frac{1}{2\mathfrak s(t)}\int_\bbT \theta_\al^2(\al,t) d\al+2\pi\lambda \mathfrak s(t)+\frac{\mathfrak s^2(t)}{2}\int_\bbT(1+y_s)^2ds\Big)=-\int_{\mathbb R^2}|\nabla \vu(x,t)|^2dx,
	\end{align}
	where $\vu$ is the velocity field.
\end{lemma}
\begin{proof}
From our assumption and the properties of the Green function $G$, $\vu$ is continuous on $\mathbb R^2$ and tends to $0$ as $|x|\to+\infty$.  Since the first equation of \eqref{eq2} holding in the sense of distribution, we choose the test function to be $\vu$, which implies
\begin{align*}
	\int_{\mathbb R^2}|\nabla \vu|^2dx=&\int_{\mathbb R^2}\vu(x,t)\cdot f(x,t)dx\\
	=&\int_{\mathbb R^2}\int_\bbT \vu(x,t)\cdot \vF(s,t)\delta(x-\vX(s,t))ds dx\\
	=&\int_\bbT \vu(\al,t)\cdot (\lambda \theta_\al\vn-\frac{\theta_{\al\al\al}\vn}{\mathfrak s ^2}-\frac{1}{2}\frac{\theta_\al^3\vn}{\mathfrak s ^2}) d\al+\int_\bbT \vu(s,t)\cdot \vX_{ss}(s,t) ds.
\end{align*}
By \eqref{eq:Xs}, it holds that
\begin{align*}
	-\frac{d}{dt}\Big(\frac{\mathfrak s^2}{2}\int_\bbT(1+y_s)^2ds\Big)=&-\frac{d}{dt}\frac{1}{2}\int_\bbT|(1+y_s)\vz_\al(\al(s,t),t)|^2ds\\
	=&-\frac{d}{dt}\frac{1}{2}\int_\bbT|\vX_s(s,t)|^2ds=-\int_\bbT\vX_s(s,t)\cdot\vX_{st}(s,t)ds\\
	=&\int_\bbT\vX_t(s,t)\cdot\vX_{ss}(s,t)ds=\int_\bbT\vu(\vX(s,t),t)\cdot\vX_{ss}(s,t)ds.
\end{align*}
From \eqref{eq:sa2}, \eqref{eq:thetaoriginal}, and \eqref{eq:St}, we have
\begin{align*}
	&-\frac{d}{dt}(\frac{1}{2\mathfrak s}\int_\bbT \theta_\al^2 d\al+2\pi\lambda \mathfrak s)\\
	=&-\frac{1}{\mathfrak s}\int_\bbT \theta_\al \theta_{\al t}d\al+\frac{ \mathfrak s_t}{2\mathfrak s^2}\int_\bbT \theta_\al^2 d\al-2\pi\lambda \mathfrak s_t\\
	=&\frac{1}{\mathfrak s}\int_\bbT \theta_{\al\al}(\frac{(\vu\cdot\vn)_\al}{\mathfrak s }+\frac{\mathcal T}{\mathfrak s }\theta_\al)d\al+\frac{\mathfrak s_t}{2\mathfrak s^2}\int_\bbT \theta_\al^2 d\al-2\pi\lambda \mathfrak s_t\\
	=&-\frac{1}{\mathfrak s ^2}\int_\bbT \theta_{\al\al\al}\vu\cdot\vn d\al- \frac{1}{2\mathfrak s^2}\int_\bbT \theta_{\al}^2\mathcal T_\al d\al	+\frac{\mathfrak s_t}{2\mathfrak s^2}\int_\bbT \theta_\al^2 d\al-2\pi\lambda \mathfrak s_t\\
	=&-\frac{1}{\mathfrak s ^2}\int_\bbT \theta_{\al\al\al}\vu\cdot\vn d\al- \frac{1}{2\mathfrak s^2}\int_\bbT \theta_{\al}^2\mathcal (\mathfrak s_t+\theta_\al \vu\cdot\vn) d\al	+\frac{\mathfrak s_t}{2\mathfrak s^2}\int_\bbT \theta_\al^2 d\al-2\pi\lambda \mathfrak s_t\\
	=&\int_\bbT \vu(\al,t)\cdot (\lambda \theta_\al\vn-\frac{\theta_{\al\al\al}\vn}{\mathfrak s ^2}-\frac{1}{2}\frac{\theta_\al^3\vn}{\mathfrak s ^2}) d\al.
\end{align*}
This proves the lemma.
\end{proof}
\begin{remark}
	From \eqref{eq:paF1} and \eqref{eq:paF2}, it is immediately that $\int_\bbT \vF(s,t)ds=0$. Thanks to this fact, we do not suffer from the Stokes paradox of logarithmic growth of the velocity field $\vu$ at infinity.
\end{remark}

\section{A Priori Estimates}
In this section, we derive the evolution equation to the oscillation part of the tangent angle function. Base on this equation, we introduce an modified contour dynamic system and give a priori estimates of this system. 
\subsection{The oscillation part of the tangent angle function}
The tangent angle function can be split into its mean and its oscillation, i.e.,
\begin{align*}
	\theta(\al,t)=\mathring \theta(\al,t)+\bar\theta(t),
\end{align*}
where $\bar\theta(t)=\frac{1}{2\pi}\int_\bbT \theta(\al',t)d\al'$, and $\mathring \theta(\al,t)=\theta(\al,t)-\bar\theta(t)$ is a zero-mean function. From \eqref{eq:thetaoriginal} and \eqref{eq:T}, one can see that the mean part satisfies
\begin{align}\label{eq:modify-btho}
	\bar\theta_t(t)=&\frac{1}{2\pi}\int_{-\pi}^{\pi}\theta_t(\al',t)d\al'=\frac{1}{2\pi\mathfrak s}\int_{-\pi}^{\pi}\mathcal T(\al',t)\theta_\al(\al',t)d\al',\ \bar\theta(0)=\bar\theta_0.
\end{align}
It follows that
\begin{align*}
	\mathring \theta_t(\al,t)=&\frac{1}{\mathfrak s}\big(\vu_\al\cdot\vn-\vu\cdot\vt\theta_\al+\mathcal T\theta_\al\big)(\al,t)-\frac{1}{2\pi\mathfrak s}\int_{-\pi}^{\pi}\mathcal T(\al',t)\theta_\al(\al',t)d\al'.
\end{align*}

Next, we will show that all the information of deformation is contained in $(\mathring\theta,y_s,\mathfrak s)$. That is to say, if we regard $\bar\theta$ as an independent variable, the following result holds.
\begin{lemma}\label{lem:barth}
	Given $\big(\theta(\bar\theta,\al),y_s(s),\mathfrak s\big)$, let $\vu$, $\vz$, $\vn$, and $\vt$ be the functions defined in \eqref{eq:ual}, \eqref{eq:z}, and \eqref{eq:direc}, it holds that
	\begin{align*}
		\frac{d}{d\bth}\Big((\vu\cdot\vn)(\bar\theta,\al)\Big)=\frac{d}{d\bth}\Big((\vu\cdot\vt)(\bar\theta,\al)\Big)=0.
	\end{align*}
\end{lemma}
The proof of this lemma can be found in Appendix C. As 
$$\pa_\al\Big(\vu(\vz(\bar\theta,\al))\Big)\cdot\vn(\bar\theta,\al)=\pa_\al\Big(\vu(\vz(\bar\theta,\al))\cdot\vn(\bar\theta,\al)\Big)+\vu(\vz(\bar\theta,\al))\cdot\theta_\al(\bar\theta,\al)\vt(\bar\theta,\al),$$ 
Lemma \ref{lem:barth} implies that $\frac{d}{d\bth}\Big(\pa_\al\big(\vu(\vz(\bth,\al))\big)\cdot\vn(\bth,\al)\Big)=0.$

We introduce
\begin{align*}
	\mathring \vn=(-\sin&(\rth),\cos(\rth)),\ \mathring\vt=(\cos(\rth),\sin(\rth)),\ \mathring\vz(\al,t)=\mathfrak s\int^\al_{-\pi}\big(\cos(\rth),\sin(\rth)\big)d\al',\\
	\mathring\vu(\al,t)=&\int_\bbT G\big(\mathring\vz(\al,t)-\mathring\vz(\al',t)\big)\cdot\Big(\big(\lambda \mathring\theta_\al\mathring\vn- \frac{\mathring\theta_{\al\al\al}}{\mathfrak s ^2}\mathring\vn-\frac{\mathring\theta_\al^3}{\mathfrak s ^2}\mathring\vn\big)(\al',t)\\
	&\qquad\qquad+\mathfrak s \big(\frac{y_{ss}(s(\al',t),t)\mathring\vt}{1+y_s(s(\al',t),t)}+(1+y_s(s(\al',t),t)) \mathring\theta_\al(\al',t)\mathring\vn\big)\Big)d\al'.
\end{align*}
From Lemma \ref{lem:barth}, we know that
\begin{align*}
	\mathring \vu(\al,t)\cdot\mathring\vn(\al,t)=\vu(\al,t)\cdot\vn(\al,t),\ \mathring \vu(\al,t)\cdot\mathring\vt(\al,t)=\vu(\al,t)\cdot\vt(\al,t),\\
	\pa_\al\big(\mathring \vu(\al,t)\big)\cdot\mathring\vn(\al,t)=\pa_\al\big(\vu(\al,t)\big)\cdot\vn(\al,t).\qquad\quad\qquad
\end{align*}
Therefore, the evolution equation of $\rth$ can be rewritten as
\begin{align*}
	\mathring \theta_t(\al,t)=&\frac{1}{\mathfrak s}\big(\mathring\vu_\al\cdot\mathring\vn-\mathring\vu\cdot\mathring\vt\mathring\theta_\al+\mathring\cT\mathring\theta_\al\big)(\al,t)-\frac{1}{2\pi\mathfrak s}\int_{-\pi}^{\pi}\mathring\cT(\al',t)\mathring\theta_\al(\al',t)d\al'.
\end{align*}
Here
\begin{align*}
	\mathring \cT(\al,t)=&\int^\al_{-\pi}\mathring\theta_\al(\al') \mathring \vu\cdot\mathring\vn d\al'-\frac{\al+\pi}{2\pi}\int^\pi_{-\pi}\mathring\theta_\al \mathring \vu\cdot\mathring\vn d\al+\mathring \vu(\mathring\vz(-\pi,t),t)\cdot\mathring\vt(-\pi,t).
\end{align*}
This shows that the evolution equation of $\rth(\al,t)$ is independent on $\bth(t)$, and $\bth(t)$ is determined by $(\rth,y_s,\mathfrak s)$.
\subsection{Modified contour dynamic system}
Based on the evolution equations of $(\rth,y_s,\mathfrak s)$, we introduce the modified contour dynamic system. Giving function $(\tilde\theta,\tilde y_s,\tilde {\mathfrak s})$ which satisfies 
\begin{align*}
	(\tilde\theta(\al,t)-\al)\in H^3(\bbT),\ \tilde y_s(s,t)\in h^1(\bbT),\ \int_{-\pi}^{\pi}\tilde\theta d\al=0,\ \int_{-\pi}^{\pi}\tilde y_s ds=0,\ \fs(t)>0,\forall t\in[0,T],
\end{align*}
we define the modified direction vectors by
\begin{align*}
	\tilde\vn(\al,t)=&(-\sin(\tilde\theta(\al,t)),\cos(\tilde\theta(\al,t))),\quad \tilde\vt(\al,t)=(\cos(\tilde\theta(\al,t)),\sin(\tilde\theta(\al,t))),
\end{align*}
and the modified velocity by
\begin{align}
	\tilde \vu(\al,t)=\int_\bbT &G(\tilde\vz(\al,t)-\tilde\vz(\al',t))\cdot\big(\lambda \tilde\theta_\al\tilde\vn- \frac{\tilde\theta_{\al\al\al}}{\tilde {\mathfrak s} ^2}\tilde\vn-\frac{\tilde\theta_\al^3}{\tilde {\mathfrak s} ^2}\tilde\vn\big)(\al',t)d\al'\label{eq:modify-u}\\
	&+\int_\bbT G(\tilde\vz(\al,t)-\widetilde\vX(s',t))\cdot\tilde \fs \big(\tilde y_{ss}(s',t)\tilde\vt+(1+\tilde y_s)^2 \tilde\theta_\al(\al(s',t),t)\tilde\vn\big)ds',\nonumber\\
	\widetilde\cT(\al,t)=\int^\al_{-\pi}&\tilde\theta_\al(\al',t) \tilde\vu(\al',t)\cdot\tilde\vn(\al',t)d\al'\label{eq:modify-T}\\
	&-\frac{\al+\pi}{2\pi}\int^\pi_{-\pi}\tilde\theta_\al(\al',t) \tilde\vu(\al',t)\cdot\tilde\vn(\al',t)d\al'+\overline {\widetilde\cT}(t).\nonumber
\end{align}
Here 
\begin{align}
	\tilde\vz(\al,t)=&\tilde {\mathfrak s}(t)\int^\al_{-\pi}\big(\cos(\tilde\theta(\al',t)),\sin(\tilde\theta(\al',t))\big)d\al'\label{eq:modify-z}\\
	&-\frac{\tilde {\mathfrak s} (t)\al}{2\pi}\int^\pi_{-\pi}\big(\cos(\tilde\theta(\al',t)),\sin(\tilde\theta(\al',t))\big)d\al',\nonumber\\
	\widetilde\vX(s,t)=&\tilde\vz(\al(s,t),t),\quad \al(s,t)=s+\int_{-\pi}^s\tilde y_s(s',t)ds',\label{eq:modify-X}\\
	\overline {\widetilde\cT}(t)=&\tilde \vu(\tilde\vz(-\pi,t),t)\cdot\tilde\vt(-\pi,t).\nonumber
\end{align}
Then we introduce the following modified contour dynamic system:
\begin{align}
	\tilde\theta_t(\al,t)&=\frac{1}{\tilde {\mathfrak s}(t)}\big(\tilde\vu_\al\cdot\tilde\vn+(\widetilde\cT-\tilde\vu\cdot\tilde\vt)\tilde\theta_\al\big)-\frac{1}{2\pi\tilde {\mathfrak s}(t)}\int_{-\pi}^{\pi}\int_{-\pi}^{\al'}\widetilde \cT(t)\tilde\theta_\al d\al''d\al'\label{eq:modify-tho}\\
	&=\frac{1}{4\tilde {\mathfrak s}^3(t)}\mathcal H(\tilde\theta_{\al\al\al})(\al,t)+\tilde g_\theta(\al,t),\nonumber\\
	\tilde y_{st}(s,t)&=\frac{1}{\tilde {\mathfrak s}(t)}\pa_s\Big(\tilde \vu\cdot\tilde \vt(\al(s,t),t)-\widetilde\cT(\al(s,t),t)\Big)=\frac{1}{\tilde {\mathfrak s}(t)}\tilde \vu_s(s,t)\cdot\tilde \vt-(1+\tilde y_s)\frac{\tilde {\mathfrak s}_t}{\tilde {\mathfrak s}}\label{eq:modify-ys}\\
	&=-\frac{1}{4}\mathfrak h(\tilde y_{ss})(s,t)+\tilde g_y(s,t),\nonumber\\
	\tilde {\mathfrak s}_t(t)&=-\frac{1}{2\pi}\int^\pi_{-\pi}\tilde\theta_\alpha(\al,t) \tilde \vu(\al,t)\cdot\tilde\vn(\al,t) d\al,\label{eq:modify-S},
\end{align}
with
\begin{align}\label{ini:modify}
	\tilde\theta(\al,0)=\tilde\theta_0(\al), \quad \tilde y_s(s,0)=y_{0s}(s,0),\quad \tilde {\mathfrak s}(0)={\mathfrak s}_0.
\end{align}
Here $\tilde g_\theta$ and $\tilde g_y$ are the modified error terms which have similar expressions to $g_\theta$ and $g_y$. From above definitions, it is easy to verify that $\int_\bbT\tilde g_\theta d\al=\int_{\bbT}\tilde g_y ds=0$. Therefore, 
\begin{align*}
	\int_\bbT\tilde \theta_t d\al=\int_\bbT\tilde \theta d\al=\int_\bbT\tilde y_{st} ds=\int_\bbT\tilde y_s ds\equiv0.
\end{align*}
The reason to introduce such modified system is that, not every $\tilde\theta$ satisfying $(\tilde\theta-\al)\in C(\bbT)$ can reconstruct a closed string, in other word, $\tilde\theta$ may not satisfy the following closed-string condition
\begin{align}\label{con:adm}
	\int^\pi_{-\pi}\big(\cos(\tilde\theta),\sin(\tilde\theta)\big)d\al=(0,0).
\end{align}
In the proof of Theorem \ref{thm:locex} we need to use the Schauder fixed point theorem, which is valid only in a convex space. However, the set of $\theta$ satisfying \eqref{con:adm} is not convex. To overcome such difficulty, we introduce $\tilde \vz$ which is continuous on $\bbT$ \cite{LA}, then \eqref{eq:modify-u} is well defined. We call $(\theta_0,y_{0s},\fs_0)$ the closed-string initial data if $\theta_0$ satisfies \eqref{con:adm}. To solve this new system, we introduce the modified non-self-intersecting assumption and the modified well-stretched assumption:
\begin{align}
 	\frac{1}{|\al_1-\al_2|}\big|\int^{\al_1}_{\al_2}\big(\cos(\tilde\theta),\sin(\tilde\theta)\big)d\al'\big|-\frac{1}{2\pi}\big|\int^\pi_{-\pi}(\cos(\tilde\theta),\sin(\tilde\theta))d\al\big|\ge\beta_1>0,\label{eq:modify-betaal}\\
 	1+\tilde y_s(s,t)\ge \beta_2>0.\label{eq:modify-betas}
\end{align}
From above inequalities, one can see that
\begin{align*}
 	\beta_1\le 1,\quad \beta_2\le\min_{s\in \bbT}(1+\tilde y_s)\le \frac{1}{2\pi}\int_\bbT(1+\tilde y_s)ds=1.
 \end{align*}
 Under these two assumptions, for $\forall s_1,s_2\in\bbT$, it holds that
\begin{align*}
	&|\tilde\vX(s_1,t)-\tilde\vX(s_2,t)|\\
	\ge&\tilde\fs\big|\int^{\al(s_1,t)}_{\al(s_2,t)}\big(\cos(\tilde\theta),\sin(\tilde\theta)\big)d\al'\big|-\tilde\fs\frac{|\al(s_1,t)-\al(s_2,t)|}{2\pi}\big|\int^\pi_{-\pi}(\cos(\tilde\theta),\sin(\tilde\theta))d\al\big|\\
	\ge&\beta_1\tilde\fs|\al(s_1,t)-\al(s_2,t)|=\beta_1\tilde\fs\big|\int^{s_1}_{s_2}1+\tilde y_s(s',t)ds'\big|\\
	\ge&\beta_1\beta_2\tilde\fs|s_1-s_2|,
\end{align*}
where $|s_1-s_2|$ is the distance between $s_1$ and $s_2$ on $\bbT$.

In the next section, we will show that this modified system is well-posed, and the solutions to \eqref{eq:modify-tho}-\eqref{eq:modify-S} with closed-string initial data always satisfy \eqref{con:adm}. Then all these modified functions we defined above are actually the same to the functions defined in Section 2. Therefore, one can solve $\bth$ from \eqref{eq:modify-btho}, and $(\tilde\theta+\bth,\tilde y_s,\tilde {\mathfrak s})$ is a solution to \eqref{eq:th}-\eqref{eq:St}. In the rest of this paper, we omit the tilde on $(\tilde\theta,\tilde y_s,\tilde {\mathfrak s},\tilde\vz,\tilde\vu,\tilde\vX,\widetilde\cT,\tilde\vn,\tilde\vt)$ for convenience.
\subsection{Preliminaries}
First, we introduce some fundamental lemmas.
\begin{lemma}\label{lem:theta-z}\cite{Am}
	Let $s\ge1$ and assume $(\theta(\al)-\al)\in H^{s-1}(\bbT)$, we have 
	\begin{align*}
		||\vz||_{H^{s}}\lesssim \mathfrak s(1+||\theta-\al||_{H^{s-1}}).
	\end{align*}
\end{lemma}
\begin{proof}
	Recalling \eqref{eq:modify-z}, we deduce that
	\begin{align*}
		\vz_\al(\al,t)=&\fs(t)\big(\cos(\theta(\al,t)),\sin(\theta(\al,t))\big)-\frac{ \fs(t)}{2\pi}\int^\pi_{-\pi}\big(\cos(\theta(\al',t)),\sin(\theta(\al',t))\big)d\al'.
	\end{align*}
	The conclusion follows immediately.
\end{proof}
\begin{lemma}\label{lem:H}\cite{Am}
	For $\forall\psi\in H^s(\bbT)$, the operator $[\mathcal H,\psi]$ is bounded from $H^0(\bbT)$ to $H^{s-1}(\bbT)$, it is also bounded from $H^{-1}(\bbT)$ to $H^{s-2}(\bbT)$. Thus, for $i=0,-1$, we have
	\begin{align*}
		||[\mathcal H,\psi]f||_{H^{s-1+i}}\lesssim||f||_{H^{i}}||\psi||_{H^{s}}.
	\end{align*}
\end{lemma}
\begin{proof}
	We write $[\mathcal H,\psi]$ as an integral operator:
	\begin{align*}
		[H,\psi]f(\al)=&\f{1}{2\pi}\int_{-\pi}^\pi f(\al')(\psi(\al')-\psi(\al))\f{1}{\tan{(\frac{(\al-\al')}{2})}}d\al'\\
		=&\f{1}{2\pi}\int_{-\pi}^\pi f(\al')\f{\psi(\al')-\psi(\al)}{\al'-\al}\f{\al'-\al}{\tan{(\frac{(\al-\al')}{2})}}d\al'.
	\end{align*}
Here $\f{\psi(\al')-\psi(\al)}{\al'-\al}$ is a divided difference, and $\f{\al'-\al}{\tan{(\frac{(\al-\al')}{2})}}$ is an analytic function, the conclusion follows immediately. See \cite{Am} for more details.
\end{proof}
\begin{lemma}\label{lem:1/2}
	Suppose that $f\in H^1(\bbT)$ and $g\in H^{1/2}(\bbT)$, there exists a constant $C$ such that
	\begin{align*}
		||fg||_{H^{1/2}}\le C||f||_{H^1}||g||_{H^{1/2}}.
	\end{align*}
\end{lemma}
\begin{proof}
	Consider the operator $\mathcal N:H^k(\bbT)\to H^k(\bbT)$ given by $\mathcal N(g)=fg$ for $k=0,1$. It is a bounded operator for $k=0,1$. Indeed
	\begin{align*}
		||fg||_{L^2}\le&||f||_{L^\infty}||g||_{L^2}\le||f||_{H^1}||g||_{L^2},\\
		||fg||_{\dot H^1}\le&||f_\al g||_{L^2}+||fg_\al||_{L^2}\le||f||_{H^1}||g||_{H^1}.
	\end{align*}
	Then the interpolation theory implies that $\mathcal N$ is bounded from $H^{1/2}$ to itself \cite{Ta}.
\end{proof}
The conclusions of the above two lemmas still hold in the material coordinate.

We introduce some notations that will be heavily used in the rest of this paper. For $\al,\al'\in\bbT$, let
\begin{align*}
	\tau(\al,\al')=\left\{
  	\begin{array}{ll}
  		\al'-\al+2\pi,&\al'-\al<-\pi,\\
  		\al'-\al,&-\pi\le\al'-\al<\pi,\\
  		\al'-\al-2\pi,&\pi\le\al'-\al,
  	\end{array}
  \right.
\end{align*}
which means that $\tau(\al,\al')\in[-\pi,\pi)$. We define
\begin{align*}
	L(\al,\al')=\frac{\vz(\al')-\vz(\al)}{\tau(\al,\al')},\ M(\al,\al')=\frac{\vz_\al(\al')-\vz_\al(\al)}{\tau(\al,\al')},N(\al,\al')=\frac{L(\al,\al')-\vz_\al(\al)}{\tau(\al,\al')},
\end{align*}
for $\al'\neq\al$, and
\begin{align*}
	L(\al,\al)=\vz_\al(\al),\ M(\al,\al)=\vz_{\al\al}(\al),\ N(\al,\al)=\frac{\vz_{\al\al}(\al)}{2}.
\end{align*}
Note that $\tau(\al,\al')$ is not continuous at $|\al'-\al|=\pi$. For the functions involving $\tau$, i.e. $L(\al,\al')$, we write $\pa_\al L(\al,\al')$ in the sense of left derivative, and $\pa_{\al'} L(\al,\al')$ in the sense of right derivative. Therefore, it is easy to verify that
\begin{align*}
	\pa_\al L=N,\ \pa_{\al'}L=\frac{\vz_\al(\al')-L}{\tau},\ \vz_\al(\al')=L+\tau(M-N).
\end{align*}
Similarly, for $s,s'\in\bbT$, let $\iota(s,s')=\tau(s,s')$. We define
\begin{align*}
	l(s,s')=\frac{\vX(s')-\vX(s)}{\iota(s,s')},\ m(s,s')=\frac{\vX_s(s')-\vX_s(s)}{\iota(s,s')},\ n(s,s')=\frac{l(s,s')-\vX_s(s)}{\iota(s,s')},
\end{align*}
for $s'\neq s$, and 
\begin{align*}
	l(s,s)=\fs(1+y_s)\vt,\ m(s,s)=\fs\big((1+y_s)^2\theta_\al\vn+y_{ss}\vt\big),\ n(s,s)=\frac{\fs}{2}\big((1+y_s)^2\theta_\al\vn+y_{ss}\vt\big).
\end{align*}
\begin{lemma}\label{lem:delta}
Suppose $f(\al)\in H^{2/5}(\bbT)$, $g(s)\in h^{2/5}(\bbT)$, let
\begin{align*}
	\mathfrak f(\al,\al')=\frac{f(\al')-f(\al)}{\tau(\al,\al')},\quad \mathfrak g(s,s')=\frac{g(s')-g(s)}{\iota(s,s')},
\end{align*}
we have the following estimates:

	(1)For $\forall 1\le p\le\infty$, it holds that
	\begin{align}\label{eq:lp}
		||\mathfrak f(\al,\cdot)||_{L^p}\le C ||\pa_\al f||_{L^p}, \quad||\mathfrak g(s,\cdot)||_{l^p}\le C ||\pa_s g||_{l^p}.
	\end{align}

	(2)There exists a universal constant $C$ such that
	\begin{align*}
		||\mathfrak f||_{L^2L^2}\le C ||f||_{\dot H^{1/2}},\ ||\pa_\al\mathfrak f||_{L^2L^2}\le C ||f||_{\dot H^{3/2}},\ ||\pa_\al^2\mathfrak f||_{L^2L^2}\le C ||f||_{\dot H^{5/2}},\\
		||\mathfrak g||_{l^2l^2}\le C ||g||_{\dot h^{1/2}},\ ||\pa_s\mathfrak g||_{l^2l^2}\le C ||g||_{\dot h^{3/2}},\ ||\pa_s^2\mathfrak g||_{l^2l^2}\le C ||g||_{\dot h^{5/2}}.
	\end{align*}
	Especially, it holds that
	\begin{align}\label{eq:vX5/2}
		||\pa_\al n||_{l^2l^2}\le C||\vX||_{\dot h^{5/2}}\le C\fs(1+||y_s||_{l^\infty})^{5/2}(1+||\theta||_{\dot H^2})(1+||y_{s}||_{\dot h^{3/2}}+||\theta||_{\dot H^2}).
	\end{align}

	(3)Let $\mathcal M$ be the Hardy-Littlewood maximal operator on $\bbT$. Then for $\forall\al,\al'\in\bbT$, $\forall s,s'\in\bbT$, we have
	\begin{align}\label{eq:HDM}
	|\mathfrak f(\al,\al')|\le2\mathcal M f_{\al}(\al),\quad|\mathfrak g(s,s')|\le 2\mathcal M g_{s}(s).
	\end{align}
\end{lemma}
\begin{proof}
	For the proofs of \eqref{eq:lp} and \eqref{eq:HDM}, we refer the readers to \cite{LT}. We only give the proof of \eqref{eq:vX5/2}. The idea is that homogeneous Sobolev norms can be described in terms of finite differences \cite{BCD}. 

	From the definitions, we have $\pa_sn=\pa_s^2l$, therefore 
	\begin{align*}
		||\pa_sn||^2_{l^2l^2}=&\int_\bbT\int_\bbT \left|\frac{2\big(\vX(s')-\vX(s)\big)-2(s'-s)\vX_s(s)-(s'-s)^2\vX_{ss}(s)}{(s'-s)^3}\right|^2ds'ds\\
		=&\int_\bbT\int_\bbT \left|\frac{2\big(\vX(s''+s)-\vX(s)\big)-2s''\vX_s(s)-s''^2\vX_{ss}(s)}{{s''}^3}\right|^2ds''ds\\
		=&\int_\bbT \frac{1}{|s''|^6}\sum_{k\in\mathbb Z}|2e^{is''k}-2-2is''k+s''^2k^2|^2|\widehat\vX(k)|^2ds''.
	\end{align*}
	We define
	\begin{align*}
		\mathcal F(k)=\int_\bbT \frac{|2e^{is''k}-2-2is''k+s''^2k^2|^2}{|s''|^6}ds''.
	\end{align*}
	By the Taylor expansion, one can see that $\mathcal F$ is well defined. It is easily checked that $\mathcal F$ is a radial and homogeneous function of degree 5. This implies that the function $\mathcal F(k)$ is proportional to $|k|^5$. As a result, we have
	\begin{align*}
		||\pa_sn||^2_{l^2l^2}=C||\vX||_{\dot h^{5/2}}^2.
	\end{align*}
	It follows from \eqref{eq:modify-X} that
	\begin{align*}
		||\vX||_{\dot h^{5/2}}\le||y_{ss}\vz_\al||_{\dot h^{1/2}}+\fs||(1+y_s)^2\theta_\al\vn||_{\dot h^{1/2}}.
	\end{align*}
	Applying Lemma \ref{lem:1/2}, we deduce that
	\begin{align*}
		||y_{ss}\vz_\al||_{\dot h^{1/2}}\lesssim& \fs||y_{s}||_{\dot h^{3/2}}(1+||y_s||_{l^\infty})||\theta||_{\dot H^1}\\
		||(1+y_s)^2\theta_\al\vn||_{\dot h^{1/2}}\le&||(1+y_s)^2\theta_\al\vn||_{\dot h^1}\\
		\lesssim&(1+||y_s||_{l^\infty})^{5/2}(1+||\theta||_{\dot H^2})(1+||y_{s}||_{\dot h^{1}}+||\theta||_{\dot H^2}).
	\end{align*}
	Then, we conclude that 
	\begin{align*}
		||\vX||_{\dot h^{5/2}}\lesssim\fs(1+||y_s||_{l^\infty})^{5/2}(1+||\theta||_{\dot H^2})(1+||y_{s}||_{\dot h^{3/2}}+||\theta||_{\dot H^2}).
	\end{align*}
\end{proof}
\subsection{Estimate of $\tilde g_\theta$ and $\tilde g_y$}
We start from the estimates of a special term.
\begin{lemma}\label{lem:ulinet}
 	Suppose that $(\theta,y_s,\fs)$ satisfies
  \begin{align*}
    \theta(\al,t)-\al\in H^4(\bbT),\ y_s(s,t)\in h^2(\bbT),\ \int_{-\pi}^{\pi}\theta d\al=\int_{-\pi}^{\pi}y_s ds=0,\ \fs(t)>0,\ \forall t\in[0,T],
  \end{align*}
 and \eqref{eq:modify-betaal}-\eqref{eq:modify-betas} for some constants $\beta_1,\beta_2>0$. Then we have
 	\begin{align*}
 		\left\|\int_\bbT G(\vz(\cdot,t)-\vz(\al',t))\cdot \theta_\al\vn(\al',t)d\al'\right\|_{L^2}\le C \frac{1}{\beta_1}||\theta-\al||_{\dot H^1},\\
 		\left\|\int_\bbT G(\vX(\cdot,t)-\vX(s',t))\cdot \big(1+y_s(s')\big)(\theta_\al\vn)(\al(s',t),t)ds'\right\|_{l^2}\le C \frac{1}{\beta_1\beta_2^{1/2}}||\theta-\al||_{\dot H^1},
 	\end{align*}
 	where $C>0$ is a constant.
 \end{lemma} 
 \begin{proof}
 	From the fact that
 	\begin{align*}
 		&\int_\bbT G(\vX(s,t)-\vX(s',t))\cdot \big(1+y_s(s')\big)(\theta_\al\vn)(\al(s',t),t)ds'\\
 		=&\int_\bbT G(\vz(\al(s,t),t)-\vz(\al',t))\cdot \theta_\al\vn(\al',t)d\al',
 	\end{align*}
 	we only give the proof of the first inequality.

 	As $G$ is the Green function of the incompressible Stokes equation, and $\frac{\vz_\al^\perp}{|\vz_\al|}$ is the unit normal vector of the modified string $\vz$, it holds that \cite{Po}
 	\begin{align*}
 		\int_\bbT G(\vz(\al,t)-\vz(\al',t))\cdot \vz_\al^\perp(\al',t)d\al'=0.
 	\end{align*}
 	Therefore
 	 \begin{align*}
 		&\int_\bbT G(\vz(\al,t)-\vz(\al',t))\cdot \theta_\al\vn(\al',t)d\al'\\
 		=&\int_\bbT G(\vz(\al,t)-\vz(\al',t))\cdot \big(\theta_\al\vn- \frac{\vz_\al^\perp}{\fs}\big)(\al',t)d\al'\\
 		=&\int_\bbT-\frac{\pa}{\pa_{\al'}}G(\vz(\al,t)-\vz(\al',t))\cdot \big(\vt(\al')-\vt(\al)-\underline\vt(\al')+\underline\vt(\al)\big)d\al'\\
 		=&\int_\bbT\big(\frac{L\cdot\vz_\al(\al')}{|L|^2}Id+\frac{2L\cdot \vz_\al(\al')L\otimes L}{|L|^4}-\frac{\vz_\al(\al')\otimes L+L\otimes \vz_\al(\al')}{|L|^2}\big)\\
		&\qquad\cdot \frac{1}{4\pi\tau}\big(\vt(\al')-\vt(\al)-\underline\vt(\al')+\underline\vt(\al)\big)d\al',
 	\end{align*}
 	where
 	\begin{align*}
 		\underline\vt(\al,t)\eqdefa\int^\al_{-\pi}\big(-\sin(\theta(\al',t)),\cos(\theta(\al',t))\big)d\al'-\frac{\al}{2\pi}\int^\pi_{-\pi}\big(-\sin(\theta(\al',t)),\cos(\theta(\al',t))\big)d\al'.
 	\end{align*}
 	Consequently, by using Lemma \ref{lem:delta}, we have
 	\begin{align*}
 		&\left\|\int_\bbT G(\vz(\al,t)-\vz(\al',t))\cdot \theta_\al\vn(\al',t)d\al'\right\|_{L^2}\le C\frac{1}{\beta_1}||\vt-\underline\vt||_{\dot H^1}.
 	\end{align*}
 	As $\theta-\al\in C(\bbT)$, it holds that $\theta(\pi)-\theta(-\pi)=2\pi$. Therefore, it is easy to see that
 	\begin{align*}
 		\int_{-\pi}^{\pi}\big(\sin(\theta(\al')),\cos(\theta(\al'))\big)\theta_\al(\al')d\al'=0.
 	\end{align*}
 	Finally we have
 	\begin{align*}
 		||\vt-\underline\vt||_{\dot H^1}\le&||\sin(\theta)\big(\theta_\al-1\big)||_{L^2}+||\sin(\theta)\big(\theta_\al-1\big)||_{L^2}\\
 		&+\frac{1}{2\pi}\left\|\int_{-\pi}^{\pi}\sin(\theta)(\theta_\al-1)d\al'\right\|_{L^2}+\frac{1}{2\pi}\left\|\int_{-\pi}^{\pi}\cos(\theta)(\theta_\al-1)d\al'\right\|_{L^2}\\
 		\le&C||\theta-\al||_{\dot H^1}.
 	\end{align*}
 \end{proof}
We emphasize that, by using the Poincar\'e inequality, we have $||\theta-\al||_{\dot H^1}\le||\theta||_{\dot H^\gamma}$ for $\gamma>1$. However, as $\frac{1}{2\pi}\int_\bbT{\theta_\al}d\al=1$, it only holds that $||\theta||_{\dot H^1}\le C(1+||\theta||_{\dot H^\gamma})$. This is the reason we introduce the above lemma. Next, we give the estimates of $\vu$.
\begin{lemma}\label{lem:u-inf}
		Under the assumptions of Lemma \ref{lem:ulinet}, it holds that
	\begin{align*}
		||\vu||_{L^\infty}\le& C(\lambda)\frac{1+\fs^4}{\beta_1\beta_2\fs^3}(1+||\theta||_{\dot H^2})^2(1+||y_s||_{l^\infty})^2\big(||\theta||^{1/2}_{\dot H^{2}}||\theta||^{1/2}_{\dot H^{3}}+||y_s||_{h^1}\big),\\
		||\vu||_{L^2}\le& C(\lambda)\frac{1+\fs^4}{\beta_1\beta_2\fs^3}(1+||\theta||_{\dot H^2})^2(1+||y_s||_{l^\infty})^2\big(||\theta||_{\dot H^{2}}+||y_s||_{h^1}\big),
	\end{align*}
	where $C(\lambda)>0$ is a constant depends only on $\lambda$.
\end{lemma}
\begin{proof}
	Recalling \eqref{eq:modify-u}, we deduce that 
	\begin{align*}
		\vu(\al,t)=&\int_\bbT G(\vz(\al,t)-\vz(\al',t))\cdot\big(\lambda \theta_\al\vn- \frac{\theta_{\al\al\al}}{\fs^2}\vn-\frac{\theta_\al^3}{\fs^2}\vn\big)(\al',t)d\al'\\
		&+\int_\bbT G(\vX(s,t)-\vX(s',t))\cdot\fs \big(y_{ss}(s',t)\vt+(1+y_s)^2 \theta_\al(\al(s',t),t)\vn\big)ds'\\
		=&\mathrm{p.v.}\int_\bbT -\frac{\pa}{\pa_{\al'}}G(\vz(\al,t)-\vz(\al',t))\cdot\big(\lambda \vt-\frac{\theta_{\al\al}\vn}{\fs^2}-\frac{\theta_\al^2\vt}{2\fs^2}\big) d\al'\\
		&+\mathrm{p.v.}\int_\bbT -\frac{\pa}{\pa_{s'}}G(\vz(\al,t)-\vX(s',t))\cdot\fs(1+y_s(s',t))\vt ds'\\
		=&\frac{1}{4\fs^2}\mathcal H(\theta_{\al\al})\vn+\frac{1}{4\fs^2}[\mathcal H,\vn](\theta_{\al\al})-\frac{\fs}{4}\mathfrak h(y_s)\vt-\frac{\fs}{4}[\mathfrak h,\vt](y_s)\\
		&-\int_\bbT\Big(\big(\frac{1}{\tau}+ \frac{1}{2\tan(\frac{\al-\al'}{2})}\big)Id+\frac{L\cdot(M-N)}{|L|^2}Id+\frac{2L\cdot (M-N)L\otimes L}{|L|^4}\\
		&\qquad-\frac{(M-N)\otimes L+L\otimes (M-N)}{|L|^2}\Big)\cdot\frac{\theta_{\al\al}(\al',t)\vn}{4\pi\fs^2}d\al'\\
		&+\int_\bbT\Big(\big(\frac{1}{\iota}+\frac{1}{2\tan(\frac{s-s'}{2})}\big)Id +\frac{l\cdot(m-n)}{|l|^2}Id+\frac{2l\cdot (m-n)l\otimes l}{|l|^4}\\
		&\qquad-\frac{(m-n)\otimes l+l\otimes (m-n)}{|l|^2}\Big)\cdot\frac{\fs y_s(s',t)\vt}{4\pi}ds'\\
		&+\int_\bbT\big(\frac{L\cdot\vz_\al(\al')}{|L|^2}Id+\frac{2L\cdot \vz_\al(\al')L\otimes L}{|L|^4}-\frac{\vz_\al(\al')\otimes L+L\otimes \vz_\al(\al')}{|L|^2}\big)\\
		&\qquad\cdot \frac{1}{4\pi\tau}\frac{(\theta_\al-1)^2\vt(\al')+2(\theta_\al-1)\vt(\al')-(\theta_\al-1)^2\vt(\al)-2(\theta_\al-1)\vt(\al)}{2\fs^2}d\al'\\
		&+\int_\bbT G(\vz(\al,t)-\vz(\al',t))\cdot (\lambda-\frac{1}{2\fs^2})\theta_\al\vn(\al',t)d\al'\\
		&+\int_\bbT G(\vz(\al,t)-\vX(s',t))\cdot \fs\big(1+y_s(s')\big)(\theta_\al\vn)(\al(s',t),t)ds'.
	\end{align*}	
	Here we use the fact that  $\vz_\al(\al')=L+\tau(M-N)$ and $\vX_s(s')=l+\iota(m-n)$. Note that $2\tan(\frac{\al-\al'}{2})\sim-\tau-\frac{\tau^3}{12}$, $2\tan(\frac{s-s'}{2})\sim-\iota-\frac{\iota^3}{12}$. Using Lemma \ref{lem:theta-z}, Lemma \ref{lem:delta}, Lemma \ref{lem:ulinet}, and the property of Hardy-Littlewood maximal operator, we have
	\begin{align*}
		|\vu|\lesssim& \frac{1}{\fs^2}\big(|\mathcal H(\theta_{\al\al})|+|[\mathcal H,\vn](\theta_{\al\al})|+\frac{\lambda\fs^2+\frac{1}{2}+\fs^3}{\beta_1}||\theta-\al||_{\dot H^1}+\frac{1}{\beta_1}(1+||\theta_{\al}||_{L^\infty})^2||\theta-\al||_{\dot H^2}\big)\\
		&+\fs\Big(|\mathfrak h(y_s)|+|[\mathfrak h,\vt](y_s)|+\frac{1}{\beta_1\beta_2}\big((1+||y_s||_{l^\infty})^2||\theta_\al||_{L^2}+||y_s||_{h^1}\big)||y_s||_{l^2}\Big).
	\end{align*}
	Therefore, one can achieve the results by using Lemma \ref{lem:H} and the the Gagliardo-Nirenberg interpolation inequality.
\end{proof}

Now, we give the estimates for $\tilde g_\theta$ and $\tilde g_y$.
\begin{lemma}\label{lem:h1}
	Suppose that $\big(\theta(\al,t),y_s(s,t),\fs(t)\big)$ satisfies all the assumptions in Lemma \ref{lem:ulinet}, then for $\forall\delta\in(0,\pi)$, it holds that
	\begin{align}\label{eq:gh1}
		||\tilde g_y||_{\dot h^1}\le& C(\lambda)\frac{1+\fs^3}{\beta_1^2\beta_2^2\fs^3}(1+||y_s||_{\dot h^{3/2}})^5(1+||\theta||_{\dot H^2})^4\big(\delta ||\theta||_{\dot H^4}+\frac{1}{\delta^{1/2}}||\theta||_{\dot H^{5/2}}+||y_s||_{\dot h^{3/2}}\big),
	\end{align}
	where $\beta_1,\beta_2$ are defined in \eqref{eq:modify-betaal}-\eqref{eq:modify-betas}, and $C(\lambda)>0$ is a constant that depends only on $\lambda$.
\end{lemma}
\begin{proof}
Using the same technique in Lemma \ref{lem:ulinet}, we deduce that
\begin{align*}
	&\tilde g_y(s,t)\\
	=&-(1+y_s)\frac{\mathfrak s_t}{\mathfrak s}-\frac{1}{4}\vt\cdot[\mathfrak h,\vt](y_{ss})-\frac{1}{4}\vt\cdot\mathfrak h\big((1+y_{s})^2\theta_\al\vn-(1+y_{s})\frac{\vz_\al^\perp}{\fs}\big)+\frac{1}{4\fs^3}\vt\cdot[\mathfrak h,\vn]\big(\pa_s\theta_{\al\al}\big)\\
	&\qquad\qquad\qquad-\frac{\lambda}{4\fs}\vt\cdot\mathfrak h\big((1+y_s)(\theta_\al\vn-\frac{\vz_\al^\perp}{\fs})\big)+\frac{1}{8\fs^3}\vt\cdot\mathfrak h\big((1+y_s)(\theta_\al^3\vn-\frac{\vz_\al^\perp}{\fs})\big)\\
	+&\vt\cdot\int_\bbT\Big(\frac{\pa}{\pa_s}G\big(\vX(s,t)-\vX(s',t)\big)+\frac{1}{4}\frac{1}{2\pi\tan( \frac{s-s'}{2})}Id\Big)\\
	&\quad\cdot\Big(y_{ss}\vt+(1+y_s)^2\theta_\al\vn+(1+y_s)\big(\frac{\lambda}{\fs}\theta_\al-\frac{\theta_{\al\al\al}}{\mathfrak s ^3}- \frac{1}{2}\frac{(\theta_\al)^3}{\mathfrak s ^3}\big)\vn-(1+y_s)\frac{\fs^3+\lambda\fs^2-1}{2\fs^4}\vz_\al^\perp\Big) ds'\\
	\eqdefa&I_1+I_2.
\end{align*}

\vspace{10pt}

\textit{ Estimate of $I_1$}.
Recalling \eqref{eq:modify-S}, we have
\begin{align*}
	||(1+y_s)\frac{\mathfrak s_t}{\mathfrak s}||_{\dot h^1}\le ||\theta||_{\dot H^1}||y_s||_{\dot h^1}||\vu||_{L^2}.
\end{align*}
With the help of Lemma \ref{lem:H}, it is easy to see that
\begin{align*}
	||\vt\cdot[\mathfrak h,\vt](y_{ss})||_{\dot h^1}\le&||\pa_s\vt||_{l^\infty}||[\mathfrak h,\vt](y_{ss})||_{l^2}+||[\mathfrak h,\vt](y_{ss})||_{\dot h^1}\\
	\le&||\pa_s\vt||_{l^\infty}||\vt||_{h^1}||y_{s}||_{\dot h^1}+||\pa_s\vt||_{h^1}||y_{s}||_{\dot h^1}\\
	\le&C(1+||y_s||_{\dot h^1})^2(1+||\theta||_{\dot H^2})^2||y_{s}||_{\dot h^1},\\
	||\vt\cdot\mathfrak h\big((1+y_{s})^2\theta_\al\vn-(1+y_{s})\frac{\vz_\al^\perp}{\fs}\big)||_{\dot h^1}\le&C(1+||y_s||_{\dot h^1})^3(1+||\theta||_{\dot H^2})^2\big(||y_{s}||_{\dot h^1}+||\theta||_{\dot H^2}\big),\\
	||\vt\cdot[\mathfrak h,\vn]\big(\pa_s\theta_{\al\al}\big)||_{\dot h^1}\le&C(1+||y_s||_{\dot h^1})^3(1+||\theta||_{\dot H^2})^2||\theta||_{\dot H^3}\\
	\le&C(1+||y_s||_{\dot h^1})^3(1+||\theta||_{\dot H^2})^2(\delta||\theta||_{\dot H^4}+\frac{1}{\delta^{1/2}}||\theta||_{\dot H^{5/2}}).
\end{align*}
In a similar way, we conclude that
\begin{align*}
	||\pa_s I_1||_{l^2}\le&C(\lambda)\frac{1+\fs^3}{\beta_1\beta_2\fs^3}(1+||y_s||_{\dot h^1})^3(1+||\theta||_{\dot H^2})^2(||y_s||_{\dot h^1}+\delta||\theta||_{\dot H^4}+\frac{1}{\delta^{1/2}}||\theta||_{\dot H^{5/2}}).
\end{align*}

\vspace{10pt}

\textit{ Estimate of $I_2$}. By direct computation, we deduce that
\begin{align*}
	\pa_s I_2=&\vt\cdot\pa_s\int_\bbT\Big(\frac{\pa}{\pa_s}G\big(\vX(s,t)-\vX(s',t)\big)+\frac{1}{4}\frac{1}{2\pi\tan( \frac{s-s'}{2})}Id\Big)\cdot\Big(y_{ss}\vt+(1+y_s)^2\theta_\al\vn\\
	&\quad\qquad\qquad+(1+y_s)\big(\frac{\lambda}{\fs}\theta_\al-\frac{\theta_{\al\al\al}}{\mathfrak s ^3}- \frac{1}{2}\frac{(\theta_\al)^3}{\mathfrak s ^3}\big)\vn-(1+y_s)\frac{\fs^3+\lambda\fs^2-1}{2\fs^4}\vz_\al^\perp\Big) ds'\\
	&+\pa_s\vt\cdot\int_\bbT\Big(\frac{\pa}{\pa_s}G\big(\vX(s,t)-\vX(s',t)\big)+\frac{1}{4}\frac{1}{2\pi\tan( \frac{s-s'}{2})}Id\Big)\cdot\Big(y_{ss}\vt+(1+y_s)^2\theta_\al\vn\\
	&\quad\qquad\qquad+(1+y_s)\big(\frac{\lambda}{\fs}\theta_\al-\frac{\theta_{\al\al\al}}{\mathfrak s ^3}- \frac{1}{2}\frac{(\theta_\al)^3}{\mathfrak s ^3}\big)\vn-(1+y_s)\frac{\fs^3+\lambda\fs^2-1}{2\fs^4}\vz_\al^\perp\Big) ds'\\
	\eqdefa&i_2+r_2,
\end{align*}
where $i_2$ is the most troublesome term. From the definitions, we have $\vX_s(s)=l-\iota\ n$. It follows that
\begin{align}\label{eq:G-s-s'}
	&4\pi\frac{\pa}{\pa_s}G\big(\vX(s,t)-\vX(s',t)\big)\\
	=&\frac{l\cdot\vX_s(s)}{|l|^2\iota}Id+\frac{2l\cdot \vX_s(s)l\otimes l}{|l|^4\iota}-\frac{\vX_s(s)\otimes l+l\otimes \vX_s(s)}{|l|^2\iota}\nonumber\\
	=&\frac{1}{\iota}Id-\frac{l\cdot n}{|l|^2}Id-\frac{2l\cdot nl\otimes l}{|l|^4}+\frac{n\otimes l+l\otimes n}{|l|^2}.\nonumber
\end{align}
Note that $\frac{1}{\iota(s,s')}$ is not continuous at $\left\{(s,s')\big|s=s'\text{ or } |s-s'|=\pi\right\}$. However, $\frac{1}{\iota(s,s')}-\frac{l(s,s')\cdot n(s,s')}{|l(s,s')|^2}$ and $\frac{\pa}{\pa_s}G\big(\vX(s,t)-\vX(s',t)\big)$ are continuous functions on $\bbT\times\bbT$. Therefore, it holds that
\begin{align*}
	&\pa_s\Big(\frac{\pa}{\pa_s}G\big(\vX(s,t)-\vX(s',t)\big)+\frac{1}{4}\frac{1}{2\pi\tan( \frac{s-s'}{2})}Id\Big)\\
	=&\frac{\sin^2(\frac{s-s'}{2})-\frac{1}{4}\iota^2}{4\pi\iota^2\sin^2(\frac{s-s'}{2})}Id-\frac{n\cdot n+l\cdot \pa_sn}{4\pi|l|^2}Id+\frac{2l\cdot nl\cdot n}{4\pi|l|^4}Id+\frac{\pa_sn\otimes l+l\otimes\pa_sn+2n\otimes n}{4\pi|l|^2}\\
	&+\frac{8l\cdot nl\cdot nl\otimes l}{4\pi|l|^6}-\frac{2n\cdot nl\otimes l+2l\cdot \pa_snl\otimes l+2l\cdot n(n\otimes l+l\otimes n)+2l\cdot n(n\otimes l+l\otimes n)}{4\pi|l|^4}.
\end{align*}
Then, we rewrite $i_2$ as follows:
\begin{align*}
	i_2=&\vt\cdot\int_\bbT\Big(\frac{\sin^2(\frac{s-s'}{2})-\frac{1}{4}\iota^2}{4\pi\iota^2\sin^2(\frac{s-s'}{2})}Id-\frac{n\cdot n+l\cdot \pa_sn}{4\pi|l|^2}Id+\frac{2l\cdot nl\cdot n}{4\pi|l|^4}Id+\frac{8l\cdot nl\cdot nl\otimes l}{4\pi|l|^6}\\
	&\quad-\frac{2n\cdot nl\otimes l+2l\cdot \pa_snl\otimes l+2l\cdot n(n\otimes l+l\otimes n)+2l\cdot n(n\otimes l+l\otimes n)}{4\pi|l|^4}\\
	&\qquad+\frac{\pa_sn\otimes l+l\otimes\pa_sn+2n\otimes n}{4\pi|l|^2}\Big)\cdot\Big(y_{ss}\vt+(1+y_s)^2\theta_\al\vn\\
	&\quad\qquad\qquad+(1+y_s)\big(\frac{\lambda}{\fs}\theta_\al-\frac{\theta_{\al\al\al}}{\mathfrak s ^3}- \frac{1}{2}\frac{(\theta_\al)^3}{\mathfrak s ^3}\big)\vn-(1+y_s)\frac{\fs^3+\lambda\fs^2-1}{2\fs^4}\vz_\al^\perp\Big) ds'.
\end{align*}
By using Lemma \ref{lem:delta}, we have
\begin{align*}
	&||\int_{\bbT}\frac{l\cdot \pa_snl\otimes l}{|l|^4}\cdot(1+y_s)\theta_{\al\al\al}\vn(s')ds'||_{l^2}\\
	\lesssim& \frac{1}{\beta_1\beta_2}(1+||y_s||_{l^\infty})^{1/2}||\pa_sn||_{l^2l^2}||\theta||_{\dot H^3}\\
	\lesssim& \frac{1}{\beta_1\beta_2}(1+||y_s||_{\dot h^{3/2}})^4(1+||\theta||_{\dot H^2})^2(\delta ||\theta||_{\dot H^4}+\frac{1}{\delta^{1/2}}||\theta||_{\dot H^{5/2}}).
\end{align*}
In the same way, we deduce that
\begin{align*}
	||i_2||_{l^2}\le& C(\lambda)\frac{1+\fs^3}{\beta_1^2\beta_2^2\fs^3}(1+||y_s||_{\dot h^{3/2}})^5(1+||\theta||_{\dot H^2})^4\big(\delta ||\theta||_{\dot H^4}+\frac{1}{\delta^{1/2}}||\theta||_{\dot H^{5/2}}+||y_s||_{\dot h^{3/2}}\big),\\
	||r_2||_{l^2}\le&C(\lambda)\frac{1+\fs^3}{\beta_1\beta_2\fs^3}(1+||y_s||_{\dot h^{1}})^2(1+||\theta||_{\dot H^2})^3\big(\delta||\theta||_{\dot H^4}+\frac{1}{\delta^{1/2}}||\theta||_{\dot H^{5/2}}+||y_s||_{\dot h^1}\big).
\end{align*}
Finally, we conclude that
\begin{align*}
	||\tilde g_y||_{\dot h^1}\le& C(\lambda)\frac{1+\fs^3}{\beta_1^2\beta_2^2\fs^3}(1+||y_s||_{\dot h^{3/2}})^5(1+||\theta||_{\dot H^2})^4\big(\delta ||\theta||_{\dot H^4}+\frac{1}{\delta^{1/2}}||\theta||_{\dot H^{5/2}}+||y_s||_{\dot h^{3/2}}\big).
\end{align*}
\end{proof}
\begin{lemma}\label{lem:H1}
	Suppose that $\big(\theta(\al,t),y_s(s,t),\fs(t)\big)$ satisfies all the assumptions in Lemma \ref{lem:ulinet}, then for $\forall\delta\in(0,\pi)$, it holds that
	\begin{align}\label{eq:gH1}
		||\tilde g_\theta||_{\dot H^1}\le &C(\lambda)\frac{1+\fs^3}{\fs^3\beta_1^2\beta_2^2}(1+||\theta||_{\dot H^2})^4(1+||y_s||_{\dot h^{3/2}})^3\Big(\delta||\theta||_{\dot H^4}+\frac{1}{\delta^{1/2}}||\theta||_{\dot H^{5/2}}+||y_s||_{\dot h^{3/2}}\Big).
	\end{align}
	where $\beta_1,\beta_2$ are defined in \eqref{eq:modify-betaal}-\eqref{eq:modify-betas}, and $C(\lambda)>0$ is a universal constant that only depends on $\lambda$.
\end{lemma}
The proof is similar to Lemma \ref{lem:h1}, and we omit the details.
\section{Existence and Uniqueness of the Local-in-Time Solution}
In this section, we will prove the local well-posedness of the modified contour dynamic system \eqref{eq:modify-tho}-\eqref{eq:modify-S}, and give the proof of Theorem \ref{thm:locex}.
\subsection{Existence of solutions of the contour dynamic system}
Let us introduce some notations before showing the existence of solutions of \eqref{eq:modify-tho}-\eqref{eq:modify-S}. For $T>0$, we define
\begin{align*}
	\Omega_T\eqdefa\{&\big(\theta(\al,t),y_s(s,t)\big):\theta(\al,t)-\al\in C(\bbT),\ \forall t\in[0,t],\\
	&\theta-\al\in L^\infty([0,T];H^{5/2}(\bbT))\cap L^2([0,T];H^{4}(\bbT)),\theta_t\in L^2([0,T];H^{1}(\bbT))\},\\
	&y_s\in L^\infty([0,T];h^{3/2}(\bbT))\cap L^2([0,T];h^{2}(\bbT)),y_{st}\in L^2([0,T];h^{1}(\bbT))\}.
\end{align*}
Given $(\theta_0,y_{0s},\fs_0)$ which satisfies \eqref{con:adm} and $\int_\bbT\theta_0d\al=\int_\bbT y_{0s}ds=0$, we also define
\begin{align*}
	&\Omega_{0,T}(\theta_0,y_{0s},\fs_0)\\ 
	\eqdefa&\Big\{\big(\theta,y_s\big)\in\Omega_T:\int_\bbT\theta(\al,t)d\al\equiv0,\int_\bbT y_s(s,t)ds\equiv0,\theta(\al,0)=\theta_0,y_s(s,0)=y_{0s},\\
	&\ ||\theta_t||_{L^2_T\dot H^1}\le 30(||\theta_0||_{\dot H^{5/2}}+||y_{0s}||_{\dot h^{3/2}}),||\theta-e^{t\mathcal L_0}\theta_0||_{L^\infty_T \dot H^{5/2}\cap L^2_T\dot H^4}\le ||\theta_0||_{\dot H^{5/2}}+||y_{0s}||_{\dot h^{3/2}},\\
	&\ ||y_{st}||_{L^2_T\dot h^1}\le 30(||\theta_0||_{\dot H^{5/2}}+||y_{0s}||_{\dot h^{3/2}}),||y_s-e^{t\mathfrak L}y_{0s}||_{L^\infty_T \dot h^{3/2}\cap L^2_T\dot h^2}\le ||\theta_0||_{\dot H^{5/2}}+||y_{0s}||_{\dot h^{3/2}}\Big\}.
\end{align*}
Here $\mathcal L_0(\theta)=\frac{1}{4\fs_0^3}\mathcal H(\theta_{\al\al\al})$. Be careful that not any $\theta\in\Omega_{0,T}$ can be a tangent angle function of a closed string. By Lemma \ref{lem:ex} and Lemma \ref{lem:exy}, $(e^{t\mathcal L_0}\theta_0,e^{t\mathfrak L}y_{0s})\in\Omega_{0,T}(\theta_0,y_{0s},\fs_0)$, thus $\Omega_{0,T}(\theta_0,y_{0s},\fs_0)$ is nonempty. Furthermore, $\Omega_{0,T}(\theta_0)$ is convex and closed in $\Omega_T$. Then we are going to state the result.
\begin{proposition}\label{pro:modify-ex}
	Assume $\fs_0>0$, $(\theta_0-\al)\in H^{5/2}(\bbT)$, $y_{0s}\in h^{3/2}(\bbT)$ satisfying \eqref{con:adm}, $\int_\bbT\theta_0 d\al=\int_\bbT y_{0s} ds=0$ and
	\begin{align*}
		\frac{1}{|\al_1-\al_2|}\big|\int^{\al_1}_{\al_2}\big(\cos(\theta_0),\sin(\theta_0)\big)d\al'\big|\ge\tilde \beta_1, \ \forall \al_1,\al_2\in\bbT,\\
		1+y_{0s}(s)\ge\tilde \beta_2,\ \forall s\in\bbT,
	\end{align*}
	for some constants $\tilde \beta_1,\tilde \beta_2>0$. Then there exists $T_0=T_0(\tilde \beta_1,\tilde \beta_2,\fs_0,||\theta_0||_{\dot H^{5/2}(\bbT)},||y_{0s}||_{\dot h^{3/2}})\in[0,\infty)$ and a solution with $\big(\theta(\al,t),y_s(s,t)\big)\in\Omega_{T_0}$ and $\fs(t)\in C^1[0,T]$ of the modified system \eqref{eq:modify-tho}-\eqref{eq:modify-S} satisfying
	\begin{align*}
	&||\theta||_{L^\infty_{T_0} \dot{H}^{5/2}\cap L^2_{T_0}\dot{H}^4}\le (2+4\sqrt2\fs_0^{3/2})||\theta_0||_{\dot{H}^{5/2}}+||y_{0s}||_{\dot h^{3/2}},\ ||\pa_t \theta||_{L^2_{T_0}\dot{H}^1}\le 30(||\theta_0||_{\dot H^{5/2}}+||y_{0s}||_{\dot h^{3/2}}),\\
	&||y_s||_{L^\infty_{T_0} \dot h^{3/2}\cap L^2_{T_0}\dot h^2}\le 4||y_{0s}||_{\dot h^{3/2}}+||\theta_0||_{\dot{H}^{5/2}},\ ||\pa_t y_s||_{L^2_{T_0}\dot{h}^1}\le 30(||\theta_0||_{\dot H^{5/2}}+||y_{0s}||_{\dot h^{3/2}}).
	\end{align*}
	Furthermore, for $t\in[0,T_0]$, it holds that
	\begin{align*}
		\frac{1}{|\al_1-\al_2|}\big|\int^{\al_1}_{\al_2}\big(\cos(\theta),\sin(\theta)\big)d\al'\big|-\big|\int^{\pi}_{-\pi}\big(\cos(\theta),\sin(\theta)\big)d\al'\big|\ge \frac{1}{3}\tilde \beta_1,\ \forall \al_1,\al_2\in\bbT\\
		1+y_s(s,t)\ge \frac{1}{3}\tilde\beta_2,\ \forall s\in\bbT.
	\end{align*}
\end{proposition}
To prove this proposition, we need the following lemma.
\begin{lemma}\label{lem:est-S}
	Assume $(\theta,y_s)\in\Omega_{T}$, $\theta(\al,0)=\theta_0$ which satisfies \eqref{con:adm}.
	Furthermore, we assume that $\fs_0>0$, $\int_\bbT\theta d\al=\int_\bbT y_s ds\equiv0$ and 
	\begin{align*}
		\frac{1}{|\al_1-\al_2|}\big|\int^{\al_1}_{\al_2}\big(\cos(\theta_0),\sin(\theta_0)\big)d\al'\big|\ge\tilde \beta_1,\\
		1+y_{0s}(s)\ge\tilde \beta_2.
	\end{align*}
	Then there exists a constant $T=T(||\theta||_{L^\infty_T \dot H^{5/2}},||y_s||_{L^\infty_T \dot h^{3/2}},\tilde \beta_1,\tilde \beta_2)$ such that for $\forall t\in[0,T]$, it holds that
	\begin{align}
		\big|\int^{\al_1}_{\al_2}\big(\cos(\theta(\cdot,t)),\sin(\theta(\cdot,t))\big)d\al'\big|\ge \frac{2}{3}\tilde \beta|\al_1-\al_2|,\label{eq:beta-f1}\\
		\frac{1}{3}\tilde \beta\ge \big|\int^{\pi}_{-\pi}\big(\cos(\theta(\cdot,t)),\sin(\theta(\cdot,t))\big)d\al'\big|,\label{eq:beta-f2}
	\end{align}
	and
	\begin{align}\label{eq:beta-f3}
		1+y_s(s,t)\ge \frac{1}{3}\tilde\beta_2.
	\end{align}	
	Moreover, \eqref{eq:modify-S} admits a unique solution $\fs_{\theta,y_s}$ on $[0,T]$ satisfying
	\begin{align}\label{con:S}
		\frac{1}{2}\fs_0\le \fs_{\theta,y_s}(t)\le \frac{3}{2}\fs_0,
	\end{align}
	and
	\begin{align}\label{eq:ltst}
		||\pa_t\fs_{\theta,y_s}||_{L^\infty_T}\le\frac{C}{\tilde \beta_1\tilde\beta_2}(1+||\theta||_{L^\infty_T\dot H^{2}})^3(1+||y_s||_{L^\infty_T\dot h^{1}})^3\big(||\theta||_{L^\infty_T\dot H^{2}}+||y_s||_{L^\infty_T\dot h^{1}}\big),
	\end{align}
	Where $C>0$ is a constant.
\end{lemma}
\begin{proof}
	From the assumption, we see that
	\begin{align*}
		&\Big|\big|\int^{\al_1}_{\al_2}\big(\cos(\theta(\cdot,t)),\sin(\theta(\cdot,t))\big)d\al'\big|-\big|\int^{\al_1}_{\al_2}\big(\cos(\theta_0(\cdot)),\sin(\theta_0(\cdot))\big)d\al'\big|\Big|\\
		\le&\int^{\al_1}_{\al_2}\big|\big(\cos(\theta(\cdot,t))-\cos(\theta_0(\cdot)),\sin(\theta(\cdot,t))-\sin(\theta_0(\cdot)\big)\big|d\al'\\
		\le&C|\al_1-\al_2|T^{1/2}||\theta_t||_{L^2_T\dot H^1}.
	\end{align*}
	For the same reason, we have
	\begin{align*}
		&\Big|\big|\int^{\pi}_{-\pi}\big(\cos(\theta(\cdot,t)),\sin(\theta(\cdot,t))\big)d\al'\big|-\big|\int^{\pi}_{-\pi}\big(\cos(\theta_0(\cdot)),\sin(\theta_0(\cdot))\big)d\al'\big|\Big|\\
		\le&CT^{1/2}||\theta_t||_{L^2_T\dot H^1}.
	\end{align*}
	and
	\begin{align*}
		&|1+y_s(s,t)-1-y_{0s}(s)|\\
		\le&CT^{1/2}||y_{st}||_{L^2_T\dot h^1}.
	\end{align*}
	Taking $T$ small enough, for all $t\in[0,T]$ we have
	\begin{align*}
		\big|\int^{\al_1}_{\al_2}\big(\cos(\theta(\cdot,t)),\sin(\theta(\cdot,t))\big)d\al'\big|\ge \frac{2}{3}\tilde \beta_1|\al_1-\al_2|,\\
		\frac{1}{3}\tilde \beta_1\ge \big|\int^{\pi}_{-\pi}\big(\cos(\theta(\cdot,t)),\sin(\theta(\cdot,t))\big)d\al'\big|,\\
		1+y_s(s,t)\ge \frac{1}{3}\tilde\beta_2.
	\end{align*}
	Recall that
	\begin{align*}
		\pa_t\fs_{\theta,y_s}=\frac{1}{2\pi}\int^\pi_{-\pi}&\theta_\alpha \vn\cdot \Big(\int_\bbT G(\vz(\al,t)-\vz(\al',t))\cdot\big(\lambda \theta_\al\vn- \frac{\theta_{\al\al\al}}{\fs_{\theta,y_s}^2}\vn-\frac{\theta_\al^3}{\fs_{\theta,y_s}^2}\vn\big)(\al',t)d\al'\\
		&+\int_\bbT G(\vz(\al,t)-\vX(s',t))\cdot\fs_{\theta,y_s} \big(y_{ss}\vt+(1+y_s)^2 \theta_\al\vn\big)(s',t)ds'\Big)d\al,
	\end{align*}
	where $\vz$ and $\vX$ are constructed from $(\theta,y_s,\fs_{\theta,y_s})$ in \eqref{eq:modify-z} and \eqref{eq:modify-X}. From \eqref{eq:beta-f1}, \eqref{eq:beta-f2}, and Lemma \ref{lem:u-inf}, we know that
	\begin{align}\label{eq:esst}
		\pa_t\fs_{\theta,y_s}\lesssim \frac{1+\fs_{\theta,y_s}^3}{\tilde \beta_1\tilde\beta_2\fs_{\theta,y_s}^2}(1+||\theta||_{L^\infty_T\dot H^{2}})^3(1+||y_s||_{L^\infty_T\dot h^{1}})^3\big(||\theta||_{L^\infty_T\dot H^{2}}+||y_s||_{L^\infty_T\dot h^{1}}\big).
	\end{align}
	One can verify that \eqref{eq:modify-S} satisfies Lipschitz condition. Therefore, by the  Cauchy-Lipschitz theorem, there exists a unique solution $\fs_{\theta,y_s}$ to \eqref{eq:modify-S} with $\fs_{\theta,y_s}(0)=\fs_0$.	With the help of Gronwall's inequality, we find a constant $T$ such that for $\forall t\in[0,T]$
	\begin{align*}
		\frac{1}{2}\fs_0\le \fs_{\theta,y_s}(t)\le \frac{3}{2}\fs_0.
	\end{align*}
\end{proof}
Now, we are able to prove Proposition \ref{pro:modify-ex}.
\begin{proof}[Proof of Proposition \ref{pro:modify-ex}]
	By Lemma \ref{lem:ex} and Lemma \ref{lem:exy}, for $\forall (\Theta,Y_s)\in\Omega_{0,T}(\theta_0,y_{0s},\fs_0)$, we have
	\begin{align*}
		||\Theta||_{L^\infty_T \dot H^{5/2}\cap L^2_T\dot H^4}\le (2+2\fs_0^{3/2})||\theta_0||_{\dot H^{5/2}}+||y_{0s}||_{\dot h^{3/2}},\quad ||Y_s||_{L^\infty_T \dot h^{3/2}\cap L^2_T\dot h^2}\le ||\theta_0||_{\dot H^{5/2}}+4||y_{0s}||_{\dot h^{3/2}}.
	\end{align*}
	From Lemma \ref{lem:est-S}, there exists constant $\bar T$ such that $(\Theta,Y_s,\fs_{\Theta,Y_s})$ satisfies \eqref{eq:beta-f1}-\eqref{con:S} on $[0,\bar T]$. Here $\fs_{\Theta,Y_s}$ is the solution to \eqref{eq:modify-S} with $\fs_{\Theta,Y_s}(0)=\fs_0$. Note that $\bar T$ only depends on $(\theta_0,y_{0s},\fs_0)$ and $\tilde \beta_1,\tilde \beta_2$. We define a map 
  \begin{align*}
    V:\Omega_{0,\bar T}(\theta_0,y_{0s},\fs_0)\to\Omega_{0,\bar T}(\theta_0,y_{0s},\fs_0)
  \end{align*}
  as follows. Given $(\Theta,Y_s)\in\Omega_{0,\bar T}(\theta_0,y_{0s},\fs_0)$, let $(\Phi,\mathcal Y_s)= V(\Theta,Y_s)$ be the solution to
	\begin{align}\label{eq:psi}
		\pa_t\Phi(\al,t)=\mathcal L_{\Theta,Y_s}(\Phi)(\al,t)+\tilde g_\Theta(\al,t),\ \ \Phi(\al,0)=\theta_0(\al),\\
		\pa_t\mathcal Y_s(s,t)=\mathfrak L \mathcal (Y_s)(s,t)+\tilde g_{Y}(s,t),\ \ \mathcal Y_s(s,0)=y_{0s}(s),\label{eq:Y}
	\end{align}
	where $\mathcal L_{\Theta,Y_s}\Phi=\frac{1}{4\fs^3_{\Theta,Y_s}}\mathcal H(\Phi_{\al\al\al})$.
	To show that $V$ is well-defined, we first claim that $(\Phi,\mathcal Y_s)\in\Omega_{\bar T}$. In fact, for $(\Theta,Y_s)\in\Omega_{0,\bar T}(\theta_0,y_{0s},\fs_0)$, by Lemma \ref{lem:h1} and Lemma \ref{lem:H1}, we have
	\begin{align*}
		||\tilde g_{\Theta}||_{L^2_{\bar T}\dot H^1}\le& C(1+||\Theta||_{L^\infty_{\bar T}\dot H^2})^4(1+||Y_s||_{L^\infty_{\bar T}\dot h^1})^3\Big(\delta||\Theta||_{L^2_{\bar T}\dot H^4}+\frac{\bar T^{1/2}}{\delta^{1/2}}||\Theta||_{L^\infty_{\bar T}\dot H^{5/2}}+\bar T^{1/2}||Y_s||_{L^\infty_{\bar T}\dot h^1}\Big)\\
		\le& C(1+||\theta_0||_{\dot H^{5/2}}+||y_{0s}||_{\dot h^{1}})^7(\delta+\frac{\bar T^{1/2}}{\delta^{1/2}})\big(||\theta_0||_{\dot H^{5/2}}+||y_{0s}||_{\dot h^{1}}\big),\\
		||\tilde g_{Y}||_{L^2_{\bar T}\dot h^1}\le& C(1+||\Theta||_{L^\infty_{\bar T}\dot H^2})^4(1+||Y_s||_{L^\infty_{\bar T}\dot h^1})^5\Big(\delta||\Theta||_{L^2_{\bar T}\dot H^4}+\frac{\bar T^{1/2}}{\delta^{1/2}}||\Theta||_{L^\infty_{\bar T}\dot H^{5/2}}+\bar T^{1/2}||Y_s||_{L^\infty_{\bar T}\dot h^1}\Big)\\
		\le& C(1+||\theta_0||_{\dot H^{5/2}}+||y_{0s}||_{\dot h^{1}})^9(\delta+\frac{\bar T^{1/2}}{\delta^{1/2}})\big(||\theta_0||_{\dot H^{5/2}}+||y_{0s}||_{\dot h^{1}}\big),
	\end{align*}
	where $C=(\lambda,\fs_{\Theta,Y_s},\tilde\beta_1,\tilde\beta_2)$ is a constant. Using Lemma \ref{lem:ex} and Lemma \ref{lem:exy}, we get the existence and uniqueness of the solution $(\Phi,\mathcal Y_s)\in\Omega_{\bar T}$ to \eqref{eq:psi}-\eqref{eq:Y}. It is obvious that $\Phi$ and $\mathcal Y_s$ have mean zero on $\bbT$ for $t\in[0,\bar T]$.

	Now, consider $\mathcal W=\Phi-e^{t\mathcal L_0}\theta_0$ and $\mathfrak w_s=\mathcal Y_s-e^{t\mathfrak L}y_{0s}$ which solve
	\begin{align*}
		\pa_t\mathcal  W(\al,t)=&\mathcal L_{\Theta,Y_s}\mathcal  W(\al,t)+\tilde g_\Theta(\al,t)+(\mathcal L_{\Theta,Y_s}-\mathcal L_0)e^{t\mathcal L_0}\theta_0(\al,t),\ \mathcal W(\al,0)=0\\
		\pa_t\mathfrak  w_s(s,t)=&\mathfrak L\mathfrak  w_s(s,t)+\tilde g_{Y_s}(s,t),\ \mathfrak w_s(s,0)=0.
	\end{align*}
	It follows that
	\begin{align*}
		&||\mathcal W||_{L^\infty_{\bar T} \dot H^{5/2}\cap L^2_{\bar T}\dot H^4}\\
		\le& C ||g_\Theta||_{L^2_{\bar T}\dot H^1}+C||(\mathcal L_{\Theta,Y_s}-\mathcal L_0)e^{\mathcal L_{0}}\theta_0||_{L^2_{\bar T}\dot H^1}\\
		\le& C(1+||\theta_0||_{\dot H^{5/2}}+||y_{0s}||_{\dot h^{1}})^7(\delta+\frac{\bar T^{1/2}}{\delta^{1/2}}+\bar T)\big(||\theta_0||_{\dot H^{5/2}}+||y_{0s}||_{\dot h^{1}}\big),
	\end{align*}
	and
	\begin{align*}
		||\mathfrak w_s||_{L^\infty_{\bar T} \dot h^{3/2}\cap L^2_{\bar T}\dot h^2}\le& C(1+||\theta_0||_{\dot H^{5/2}}+||y_{0s}||_{\dot h^{1}})^9(\delta+\frac{\bar T^{1/2}}{\delta^{1/2}})\big(||\theta_0||_{\dot H^{5/2}}+||y_{0s}||_{\dot h^{1}}\big).
	\end{align*}
	Here we use the fact from \eqref{eq:ltst} that
	\begin{align*}
		&||(\mathcal L_0-\mathcal L_{\Theta,Y_s})e^{t\mathcal L_0}\theta_0||_{L^2_{\bar T}\dot H^1}\\
		\le& C ||\fs_0-\fs_{\Theta,Y_s}||_{L^\infty_{\bar T}}||\theta_0||_{\dot H^{5/2}}\le C \int_0^{\bar T}|\pa_t\fs_{\Theta,Y_s}|dt||\theta_0||_{\dot H^{5/2}}\\
		\le& C \bar T(1+||\theta_0||_{\dot H^{5/2}}+||y_{0s}||_{\dot h^{1}})^7||\theta_0||_{\dot H^{5/2}}.
	\end{align*}

We take $\de$ small enough and then take $\bar T$ small enough to obtain
\begin{align*}
		&||\Phi-e^{t\mathcal L_0}\theta_0||_{L^\infty_T \dot{H}^{5/2}\cap L^2\dot{H}^4}\le || \theta_0||_{\dot H^{5/2}}+|| y_{0s}||_{\dot h^{3/2}},\ ||\pa_t\Phi||_{L^2_{\bar T}\dot H^1}\le  30(||\theta_0||_{\dot H^{5/2}}+||y_{0s}||_{\dot h^{3/2}}),\\
		&||\mathcal Y_s-e^{t\mathfrak L}y_{0s}||_{L^\infty_T \dot h^{3/2}\cap L^2_T\dot h^2}\le|| \theta_0||_{\dot H^{5/2}}+|| y_{0s}||_{\dot h^{3/2}},\ ||\pa_t\mathcal Y_s||_{L^2_{\bar T}\dot h^1}\le30(||\theta_0||_{\dot H^{5/2}}+||y_{0s}||_{\dot h^{3/2}}).
\end{align*}
By the Aubin-Lions lemma, $V\big(\Om_{0,\bar T}(\theta_0,y_{0s},\fs_0)\big)$ is compact in $C([0,\bar T];H^2(\bbT))\times C([0,\bar T];h^1(\bbT))$. As $\Om_{0,\bar T}(\theta_0,y_{0s},\fs_0)$ is convex, from the Schauder fixed point theorem, there exists a fixed point of map $V$ in $V\big(\Om_{0,\bar T}(\theta_0,y_{0s},\fs_0)\big)\subset \Om_{0,\bar T}(\theta_0,y_{0s},\fs_0)$. We denote this fixed point by $(\theta,y_s)\in \Om_{\bar T}$, which is a solution to \eqref{eq:modify-tho}-\eqref{eq:modify-S} satisfying
\begin{align*}
&||\theta||_{L^\infty_T \dot{H}^{5/2}\cap L^2_T\dot{H}^4}\le (2+4\sqrt2\fs_0^{3/2})||\theta_0||_{\dot{H}^{5/2}}+||y_{0s}||_{\dot h^{3/2}},\ ||\pa_t \theta||_{L^2_T\dot{H}^1}\le C(||\theta_0||_{\dot{H}^{5/2}}+||y_{0s}||_{\dot h^{3/2}}),\\
&||y_s||_{L^\infty_T \dot h^{3/2}\cap L^2_T\dot h^2}\le 4||y_{0s}||_{\dot h^{3/2}}+||\theta_0||_{\dot{H}^{5/2}},\ ||\pa_t y_s||_{L^2_T\dot{h}^1}\le C(||\theta_0||_{\dot{H}^{5/2}}+||y_{0s}||_{\dot h^{3/2}}).
\end{align*}
Consequently, using Lemma \ref{lem:est-S}, for $\forall t\in[0,\bar T]$, it holds that 
	\begin{align*}
		\frac{1}{|\al_1-\al_2|}\big|\int^{\al_1}_{\al_2}\big(\cos(\theta),\sin(\theta)\big)d\al'\big|-\big|\int^{\pi}_{-\pi}\big(\cos(\theta),\sin(\theta)\big)d\al'\big|\ge \frac{1}{3}\tilde \beta_1,\ \forall \al_1,\al_2\in\bbT\\
		1+y_s(s,t)\ge \frac{1}{3}\tilde\beta_2,\ \forall s\in\bbT.
	\end{align*}
This is our assertion. 
\end{proof}
In above proof we consider the modified error term $\tilde g_\theta$ and $\tilde g_y$. When $(\theta,y_s,\fs)$ is a solution to \eqref{eq:modify-tho}-\eqref{eq:modify-S} with closed-string initial data, $\theta$ always satisfies \eqref{con:adm}. Indeed, we have the following result. 
\begin{lemma}\label{rmk:adm}
	Assume $(\theta,y_s,\fs)$ is a solution to \eqref{eq:modify-tho}-\eqref{eq:modify-S} with initial data satisfying 
	\begin{align*}
	 	\int^\pi_{-\pi}\big(\cos(\theta_0(\al)),\sin(\theta_0(\al))\big)d\al=(0,0),
	 \end{align*}
	then we have
	\begin{align*}
	 	\int^\pi_{-\pi}\big(\cos(\theta(\al,t)),\sin(\theta(\al,t))\big)d\al\equiv(0,0).
	 \end{align*}
\end{lemma}
\begin{proof}
	Recalling \eqref{eq:modify-u}, \eqref{eq:modify-T}, \eqref{eq:modify-tho}, one has
	\begin{align*}
		&\frac{d}{dt}\int^\pi_{-\pi}\sin(\theta)d\al=\int^\pi_{-\pi}\cos(\theta)\theta_td\al\\
		=&\int^\pi_{-\pi}\cos(\theta) \frac{1}{\fs}\pa_\al(\vu\cdot\vn) d\al+\int^\pi_{-\pi}\cos(\theta)\frac{\cT}{\fs}\theta_\al d\al- \frac{1}{2\pi\fs}\int_{-\pi}^\pi\cT\theta_\al d\al\int^\pi_{-\pi}\cos(\theta)d\al\\
		=&\int^\pi_{-\pi}\sin(\theta)\theta_\al\frac{1}{\fs}\vu\cdot\vn d\al-\int^\pi_{-\pi}\sin(\theta)\frac{\cT_\al}{\fs}d\al-\frac{1}{2\pi\fs}\int_{-\pi}^\pi\cT\theta_\al d\al\int^\pi_{-\pi}\cos(\theta)d\al\\
		=&-\frac{\fs_t}{\fs}\int^\pi_{-\pi}\sin(\theta)d\al- \frac{1}{2\pi\fs}\int_{-\pi}^\pi\cT\theta_\al d\al\int^\pi_{-\pi}\cos(\theta)d\al,
	\end{align*}
	where we use the fact that
	\begin{align*}
		 \fs_t=\cT_\al-\theta_\alpha \vu\cdot\vn.
	\end{align*}	
	Similarly, it holds that
	\begin{align*}
		\frac{d}{dt}\int^\pi_{-\pi}\cos(\theta)d\al=-\frac{\fs_t}{\fs}\int^\pi_{-\pi}\cos(\theta)d\al+ \frac{1}{2\pi\fs}\int_{-\pi}^\pi\cT\theta_\al d\al\int^\pi_{-\pi}\sin(\theta)d\al.
	\end{align*}
	Therefore, we have
	\begin{align*}
		&\frac{d}{dt}\Big(\big(\fs\int_{-\pi}^\pi\sin(\theta)d\al\big)^2+\big(\fs\int_{-\pi}^\pi\cos(\theta)d\al\big)^2\Big)\\
		=&2\fs\fs_t\Big(\big(\int_{-\pi}^\pi\sin(\theta)d\al\big)^2+\big(\int_{-\pi}^\pi\cos(\theta)d\al\big)^2\Big)\\
		&+2\fs^2\int_{-\pi}^\pi\sin(\theta)d\al\Big( -\frac{\fs_t}{\fs}\int^\pi_{-\pi}\sin(\theta)d\al- \frac{1}{2\pi\fs}\int_{-\pi}^\pi\cT\theta_\al d\al\int^\pi_{-\pi}\cos(\theta)d\al\Big)\\
		&+2\fs^2\int_{-\pi}^\pi\cos(\theta)d\al \Big(-\frac{\fs_t}{\fs}\int^\pi_{-\pi}\cos(\theta)d\al+ \frac{1}{2\pi\fs}\int_{-\pi}^\pi\cT\theta_\al d\al\int^\pi_{-\pi}\sin(\theta)d\al\Big)\\
		=&0.
	\end{align*}
	This completes the proof.
\end{proof}
\subsection{Uniqueness of solutions of the contour dynamic system}
In this subsection we will prove the uniqueness of solutions to \eqref{eq:modify-tho}-\eqref{eq:modify-S}. 
\begin{proposition}\label{pro:locuni}
	Suppose $(\theta_0,y_{0s},\fs_0)$  satisfies the same assumption to Proposition \ref{pro:modify-ex},
	the dynamic system \eqref{eq:modify-tho}-\eqref{eq:modify-S} admits at most one solution $(\theta,y_s,\fs)\in (\Om_T,C^1[0,T])$ with initial data $(\theta_0,y_{0s},\fs_0)$.
\end{proposition}
\begin{proof}
	Suppose $(\theta_1,y_{1s},\fs_1),(\theta_2,y_{2s},\fs_2)\in (\Om_T,C^1[0,T])$ are two solutions to \eqref{eq:modify-tho}-\eqref{eq:modify-S} with initial data $(\theta_0,y_{0s},\fs_0)$. Let
	\begin{align*}
	R=&1+||\theta_1||_{L^\infty_T \dot{H}^{5/2}\cap L^2_T \dot{H}^4}+||\theta_2||_{L^\infty_T \dot{H}^{5/2}\cap L^2_T \dot{H}^4}+||y_{s1}||_{L^\infty_T {h}^{3/2}\cap L^2_T \dot{h}^2}+||y_{2s}||_{L^\infty_T {h}^{3/2}\cap L^2_T \dot{h}^2},\\
	&\Theta(\al,t)=\theta_1(\al,t)-\theta_2(\al,t),\quad Y_s(s,t)=y_{1s}(s,t)-y_{2s}(s,t).
	\end{align*}
	It holds that
	\begin{align*}
		\pa_t \Theta(\al,t)=&\frac{1}{4\fs_1^3}\mathcal H(\Theta_{\al\al\al})(\al,t)+\frac{\fs_1^3-\fs_2^3}{4\fs_1^3\fs_2^3}\mathcal H(\theta_{2\al\al\al})(\al,t)+\tilde g_{\theta_1}(\al,t)-\tilde g_{\theta_2}(\al,t), \quad \Theta(\al,0)=0,\\
		\pa_t Y_s(s,t)=&-\frac{1}{4}\mathfrak h(Y_s)(s,t)+\tilde g_{y_1}(s,t)-\tilde g_{y_2}(s,t),\quad Y_s(s,0)=0.
	\end{align*}
	Using Lemma \ref{lem:est-S}, one can see that there exists $\bar T$ such that
	\begin{align*}
		||\frac{\fs_1^3-\fs_2^3}{4\fs_1^3\fs_2^3}\mathcal H(\theta_{2\al\al\al})||_{L^2_{\bar T}L^2}\le& C(\fs_0)\Big(\int_0^{\bar T}|\fs_1-\fs_2|^2(t)||\theta_2||_{\dot H^3}^2(t)dt\Big)^{1/2}\\
		\le& C(\fs_0)\Big(\int_0^{\bar T}|\fs_1-\fs_2|^2(t)||\theta_2||_{\dot H^{5/2}}^{4/3}(t)||\theta_2||_{\dot H^{4}}^{2/3}(t)dt\Big)^{1/2}\\
		\le& C(\fs_0)||\theta_2||_{L^\infty_{\bar T}\dot H^{5/2}}^{2/3}||\theta_2||_{L^2_{\bar T}\dot H^{4}}^{1/3}\big(\int_0^{\bar T}|\fs_1-\fs_2|^3dt\big)^{1/3}\\
		\le&C(R,\fs_0)\bar T^{1/3}||\Theta||_{L^\infty_{\bar T}L^{2}}.
	\end{align*}
	The last inequality follows from Lemma \ref{rmk:adm} and Lemma \ref{lem:dif-s}. We claim that the following estimates hold for $\forall\delta\in(0,1)$:
	\begin{align}\label{eq:lowest}
	 	||\tilde g_{\theta_1}-\tilde g_{\theta_2}||_{L^2_{\bar T}L^2}+||\tilde g_{y_1}-\tilde g_{y_2}||_{L^2_{\bar T}l^2}\le C\big(\delta||\Theta||_{L^2_{\bar T}\dot H^3}+\frac{\bar T^{1/2}}{\delta}||\Theta||_{L^\infty_{\bar T}\dot H^{3/2}}+\frac{\bar T^{1/2}}{\delta}||Y_s||_{L^\infty_{\bar T}\dot h^{1/2}}\big),
	\end{align}
	where $C=C(R,\fs_0,\lambda,\tilde\beta_1,\tilde\beta_2)$. One need to be careful that the there are two transfer functions 
	\begin{align}\label{eq:tr-dif}
		\al_1(s)=s+y_1(s),\ y_1(s)=\int_0^sy_{1s}(s')ds';\quad \al_2(s)=s+y_1(s),\ y_2(s)=\int_0^sy_{2s}(s')ds',
	\end{align}
	and two inverse functions
	\begin{align*}
		s_1(\al)=\al-y_1(s_1(\al)),\quad s_2(\al)=\al-y_2(s_2(\al)).
	\end{align*}
	For the difference terms defined on different coordinates, we have
	\begin{align*}
		&\theta_{1\al\al}(\al_1(s))-\theta_{2\al\al}(\al_2(s))\\
		=&\int_{\al_2(s)}^{\al_1(s)}\theta_{1\al\al\al}d\al+\theta_{1\al\al}(\al_2(s))-\theta_{2\al\al}(\al_2(s))\\
		\le&||\theta_{1\al\al\al}||_{L^1}|y_1(s)-y_2(s)|+|\theta_{1\al\al}(\al_2(s))-\theta_{2\al\al}(\al_2(s))|
	\end{align*}
	and
	\begin{align*}
		&y_{1s}(s_1(\al))-y_{2s}(s_2(\al))\\
		=&\int_{s_2(\al)}^{s_1(\al)}y_{1ss}ds+y_{1s}(s_2(\al))-y_{2s}(s_2(\al))\\
		\le&||y_{1ss}||_{l^1}|s_1(\al)-s_2(\al)|+|y_{1s}(s_2(\al))-y_{2s}(s_2(\al))|.
	\end{align*}
	Next, we give the estimate for $|s_1(\al)-s_2(\al)|$. From \eqref{eq:tr-dif}, it holds that
	\begin{align*}
		s_1(\al)+y_1(s_1(\al))=s_2(\al)+y_2(s_2(\al)).
	\end{align*}
	Without loss of generality, for fixed $\al$, we assume $s_1>s_2$, accordingly
	\begin{align}\label{eq:dify}
		y_2(s_2(\al))-y_1(s_2(\al))=&\int_{s_2(\al)}^{s_1(\al)}y_{2s}(s')ds'+s_1(\al)-s_2(\al)\\
		\ge& \frac{1}{3}\tilde \beta_2\big(s_1(\al)-s_2(\al)\big),\nonumber
	\end{align}
	here we use the fact that $y_{2s}\ge-1+\frac{1}{3}\tilde \beta_2$. It follows that
	\begin{align*}
		|s_1(\al)-s_2(\al)|\le 3\tilde\beta_2|y_2(s_2(\al))-y_1(s_2(\al))|.
	\end{align*}
	One can complete the proof of \eqref{eq:lowest} in a similar way to Lemma \ref{lem:h1}, we omit the details.

	Choosing appropriate $\delta$ and $T$, and using Lemma \ref{lem:ex} and Lemma \ref{lem:exy}, we conclude that 
	\begin{align*}
	||\Theta||_{L^\infty_T \dot{H}^{3/2}\cap L^2_T \dot{H}^3}+||Y_s||_{L^\infty_T \dot{h}^{1/2}\cap L^2_T \dot{h}^1}\le \frac{1}{2}\big(||\Theta||_{L^\infty_T \dot{H}^{3/2}\cap L^2_T \dot{H}^3}+||Y_s||_{L^\infty_T \dot{h}^{1/2}\cap L^2_T \dot{h}^1}\big),
	\end{align*}
	which implies $\Theta\equiv0$ and $Y_s\equiv0$. Hence $\fs_1\equiv\fs_2$.
\end{proof}
\subsection{Local well-posedness of the immersed boundary problem}
Now, we are in the position to give the proof of Theorem \ref{thm:locex}.
\begin{proof}[Proof of Theorem \ref{thm:locex}]
	Let $(\theta_0,y_0,\fs_0)$ be tangent angle function, the stretching function, and the perimeter function of $\vX_0$. Using Proposition \ref{pro:modify-ex} and Proposition \ref{pro:locuni}, we can get $(\rth,y_s,\fs)$ which is the unique solution to \eqref{eq:modify-tho}-\eqref{eq:modify-S}. Then one can find $\bar\theta(t)$ and $\vu(\vX(-\pi,t),t)$ by \eqref{eq:modify-btho} and \eqref{eq:us}. Therefore, $\vX(s,t)$ constructed from \eqref{eq:reconvx} is a solution to \eqref{eq:th}-\eqref{eq:St} satisfying \eqref{eq:regthy}-\eqref{eq:ybeta}. The uniqueness of $(\rth,y_s,\fs)$ implies that of $\bar\theta(t)$, $\vu(\vX(-\pi,t),t)$, and $\vX(s,t)$. This proves the theorem.
\end{proof}
\section{Existence and Uniqueness of Global-in-Time Solutions
near Equilibrium Configurations}
In this section, we prove the existence of a global solutions to \eqref{eq3} provided that the initial string configuration is sufficiently close to equilibrium. We introduce the following equilibrium configuration
\begin{align*}
	\vX_*(s)=\big(\sin(s),\cos(s)\big),\ s\in\bbT,
\end{align*}
which is an evenly parametrized unit circle, and it has the same expression in the are-length coordinate. The corresponding equilibrium state of tangent angle function, stretching function, and perimeter function is
\begin{align}\label{eq:equilibrium}
\big(\theta_*(\al)=&\al,\ y_{*s}(s)=0,\ \fs_*=1\big).
\end{align}
As the area enclosed by the string is unchanging, we choose the initial data $\vX_0$ which satisfies 
\begin{align}\label{con:area1}
	-\frac{1}{2}\int_{\bbT}\vX \cdot\vX_{s}^{\perp} ds=\pi.
\end{align}
This means that the area enclosed by the string is $\pi$. Next, we introduce a lemma which establishes the equivalence of the Sobolev norm distance and the energy difference between $(\theta,y_s,\fs)$ and the equilibrium configuration $(\theta_*,y_{*s},\fs_*)$. Our motivation is to transform the global coercive bound on the difference of energy into a more convenient quantity.
\begin{lemma}\label{lem:eqenergy}
	Let $\vX(s)\in h^2(\bbT)$ be a closed convex string which satisfies \eqref{con:area1}, then it holds that
	\begin{align*}
		&\frac{1}{C}\big(\frac{1}{2\fs}\int_{\bbT}\theta_\al^2d\al+\frac{\fs^2}{2}\int_\bbT(1+y_s)^2ds+2\pi\lambda\fs-2\pi-2\pi\lambda\big) \\
		\le&\frac{1}{2\fs}\int_\bbT (\theta_\al-1)^2d\al+\frac{\fs^2}{2}\int_\bbT y_s^2ds+\pi(\fs-1)+2\pi\lambda(\fs-1)\\
		\le&C\big(\frac{1}{2\fs}\int_{\bbT}\theta_\al^2d\al+\frac{\fs^2}{2}\int_\bbT(1+y_s)^2ds+2\pi\lambda\fs-2\pi-2\pi\lambda\big), 
	\end{align*}
	where $C>0$ is a constant that only depends on $\fs$.
\end{lemma}
\begin{proof}
	Direct calculations show that
	\begin{align*}
		\frac{1}{2\fs}\int_\bbT (\theta_\al-1)^2 d\al=&\frac{1}{2\fs}\int_\bbT \theta_\al^2+1-2\theta_\al d\al= \frac{1}{2\fs}\int_{\bbT} \theta_\al^2d\al-\frac{\pi}{\fs}\\
		=&\frac{1}{2\fs}\int_{\bbT} \theta_\al^2d\al-\pi+\frac{\fs-1}{\fs}\pi\ge\frac{1}{2\fs}\int_{\bbT} \theta_\al^2d\al-\pi,
	\end{align*}
	where the last inequality follows from the classical isoperimetric inequality. As the string is convex, using Lemma \ref{lem:iso-gage} (Gage's isoperimetric inequality), we have
	\begin{align*}
		\frac{1}{2\fs}\int_{\bbT} \theta_\al^2d\al=\frac{\fs}{2}\int_{\bbT} \kappa^2d\al\ge \fs\pi.
	\end{align*}
	Therefore, it holds that
	\begin{align*}
		2\big(\frac{1}{2\fs}\int_{\bbT} \theta_\al^2d\al-\pi\big)=&\frac{1}{2\fs}\int_\bbT (\theta_\al-1)^2 d\al-\frac{\fs-1}{\fs}\pi+\frac{1}{2\fs}\int_{\bbT} \theta_\al^2d\al-\pi\\
		\ge&\frac{1}{2\fs}\int_\bbT (\theta_\al-1)^2 d\al+\frac{(\fs-1)^2}{\fs}\pi\\
		\ge&\frac{1}{2\fs}\int_\bbT (\theta_\al-1)^2 d\al.
	\end{align*}
	Similarly, one can deduce that
	\begin{align*}
		\frac{1}{\fs+1}\big(\frac{\fs^2}{2}\int_\bbT(1+y_s)^2ds-\pi\big)\le&\frac{\fs^2}{2}\int_\bbT y_s^2ds+\pi(\fs-1)\\
		\le&(\fs+1)\big(\frac{\fs^2}{2}\int_\bbT(1+y_s)^2ds-\pi\big).
	\end{align*}
	This completes the proof.
\end{proof}
\begin{remark}
  The assumption of the above lemma is reasonable. In the rest of this paper, we only consider the string closed to the equilibrium configuration. When $||\theta-\al||_{\dot H^2}$ is small enough, one has $\theta_\al(\al)>0,\ \forall \al\in\bbT$, which means that $\vX$ is a convex string. Furthermore, from the above proof, we also have
\begin{align}\label{eq:con-th-S}
  (\fs-1)\le \frac{1}{4\pi}||\theta_\al-1||_{L^2}^2.
\end{align}
Therefore, from Lemma \ref{lem:energy} and Lemma \ref{lem:eqenergy}, we get the following global estimate
\begin{align*}
  \big(||\theta-\al||_{\dot H^1}+||y_s||_{l^2}\big)(t)\le C \big(||\theta_0-\al||_{\dot H^1}+||y_{0s}||_{l^2}+\lambda(\fs_0-1)\big).
\end{align*}
\end{remark}

As $\big(\theta_*(\al)=\al,\ y_{*s}(s)=0,\ \fs_*=1\big)$, we actuality give the Sobolev norm distance estimates between $(\theta,y_s,\fs)$ and the equilibrium configuration in Proposition \ref{pro:modify-ex}. Next, we will show that $\big(||\theta-\al||_{\dot{H}^{5/2}}+||y_s||_{\dot h^{3/2}}\big)(t)$ cannot be big when $\big(||\theta-\al||_{\dot H^1}+||y_s||_{l^2}\big)(t)$ is small.
\begin{lemma}\label{lem:unith}
Giving $T>0$, $R\le \frac{1}{4}$, let $\vX(s,t)$ be a (local) solution to \eqref{eq3} satisfying \eqref{con:betaal} and \eqref{con:betas} for some $\beta_1,\beta_2>0$ on $[0,T]$ with initial data $\vX_0$ satisfying \eqref{con:area1}. We also assume  $(\theta,y_s)\in\Omega_T$ such that
\begin{align}
	||\theta||_{L^\infty_T \dot{H}^{5/2}\cap L^2_T \dot{H}^4}+||y_s||_{L^\infty_T \dot{h}^{3/2}\cap L^2_T \dot{h}^2}\le R.
\end{align}
If in addition
\begin{align}\label{eq:thc}
||\theta-\al||_{\dot{H}^{5/2}}^2+||y_s||_{\dot h^{3/2}}^{2}\ge c_*\big(||\theta-\al||_{\dot{H}^{1}}+||y_s||_{\dot h^{3/2}}^{2/3}||y_s||_{l^2}^{1/3}\big)^2,\ t\in[0,T],
\end{align}
for some constant $c_*=c_*(R,\lambda,\beta_1,\beta_2)>0$, one has
\begin{align}\label{eq:decay-th}
\big(||\theta-\al||_{\dot{H}^{5/2}}+||y_s||_{\dot{h}^{3/2}}\big)(t)\le e^{-t}\big(||\theta_0-\al||_{\dot{H}^{5/2}}+||y_{0s}||_{\dot{h}^{3/2}}\big).
\end{align}
\end{lemma}
\begin{proof}
Recalling \eqref{eq:th}-\eqref{eq:St}, we have
\begin{align*}
	\pa_t(\theta(\al,t)-\al)&=\mathcal L(\theta)(\al,t)+g_{\theta}(\al,t),\\
	\pa_ty_s(s,t)&=\mathfrak L(y_s)(s,t)+g_y(s,t).
\end{align*}
From \eqref{eq:con-th-S} we know that $1\le \fs\le \frac{5}{4}$. Using Lemma \ref{lem:H1} and Lemma \ref{lem:h1}, one can see that
\begin{align}\label{eq:parabolic-dispation}
	\frac{d}{dt}(||\theta-\al||_{\dot{H}^{5/2}}^2+||y_s||_{\dot h^{3/2}}^2)(t)\le -\frac{1}{8}\big(||\theta-\al||_{\dot{H}^{4}}^2+||y_s||_{\dot{h}^{2}}^2\big)(t)+C\big(||g_{\theta}||_{\dot H^1}^2+||g_y||_{\dot h^1}^2\big)(t).
\end{align}
Noting that $\pa_\al\bth=0$, from Lemma \ref{lem:H1} and Lemma \ref{lem:h1}, for $\forall\delta\in(0,1)$, we have
\begin{align*}
	||g_{\theta}||_{\dot H^1}+||g_y||_{\dot h^1}\le& C(R,\lambda,\beta_1,\beta_2)\big(\delta||\theta||_{\dot H^{4}}+\frac{1}{\delta^{1/2}}||\theta||_{\dot H^{5/2}}+||y_s||_{\dot h^{3/2}}\big).
\end{align*}
We choose $\delta$ small enough which is determined by $(R,\lambda,\beta_1,\beta_2)$ such that
\begin{align}\label{eq:th-est}
	&-\frac{1}{8}\big(||\theta-\al||_{\dot{H}^{4}}^2+||y_s||_{\dot{h}^{2}}^2\big)(t)+C\big(||g_{\theta}||_{\dot H^1}^2+||g_y||_{\dot h^1}^2\big)(t)\\
	\le&-\frac{1}{10}\big(||\theta-\al||_{\dot{H}^{4}}^2+||y_s||_{\dot{h}^{2}}^2\big)(t)+C_1\big(||\theta-\al||_{\dot{H}^{5/2}}^2+||y_s||_{\dot h^{3/2}}^2\big)(t).\nonumber
\end{align}
Here $C_1=C_1(R,\lambda,\beta_1.\beta_2)$ is a constant. 

It follows from the Gagliardo–Nirenberg inequality that
\begin{align*}
	||\theta-\al||_{\dot{H}^{5/2}}^2\le& C||\theta-\al||_{\dot{H}^{1}}||\theta-\al||_{\dot{H}^{4}},\\
	||y_s||_{\dot h^{3/2}}^2\le&C||y_s||_{\dot h^{3/2}}^{2/3}||y_s||_{l^2}^{1/3}||y_s||_{\dot h^{2}}.
\end{align*}
Therefore, for 
\begin{align*}
	c(t)=\frac{||\theta-\al||_{\dot{H}^{5/2}}^2+||y_s||_{\dot h^{3/2}}^{2}}{\big(||\theta-\al||_{\dot{H}^{1}}+||y_s||_{\dot h^{3/2}}^{2/3}||y_s||_{l^2}^{1/3}\big)^2},
\end{align*}
it holds that
\begin{align*}
	c(t)(||\theta-\al||_{\dot{H}^{5/2}}^2+||y_s||_{\dot h^{3/2}}^2)(t)\le C_2 (||\theta-\al||_{\dot{H}^{4}}^2+||y_s||_{\dot{h}^{2}}^2)(t),
\end{align*}
where $C_2>1$ is a universal constant. We see from \eqref{eq:th-est} that
\begin{align*}
	\frac{d}{dt}(||\theta-\al||_{\dot{H}^{5/2}}^2+||y_s||_{\dot h^{3/2}}^2)(t)\le (-\frac{c(t)}{10C_2}+C_1)(||\theta-\al||_{\dot{H}^{5/2}}^2+||y_s||_{\dot h^{3/2}}^2)(t).
\end{align*}
Choosing $c_*=10(1+C_1)C_2$, we arrive at \eqref{eq:decay-th} with the help of the Gronwall inequality.
\end{proof}
\begin{remark}
  As the equilibrium energy is not $0$, and thus the nonlinear terms $g_\theta$ and $g_y$ are not high order small terms, one cannot get exponential decay estimates of $(||\theta-\al||_{\dot{H}^{5/2}}^2+||y_s||_{\dot h^{3/2}}^2)(t)$ directly from \eqref{eq:parabolic-dispation}.
\end{remark}
Now, we give the proof of Theorem \ref{thm:global} in the case area $\mathfrak a=\pi$. The general case can be treated in the same way.
\begin{proof}[Proof of Theorem \ref{thm:global}]
	Define
	\begin{align}\label{def:Q}
		\mathcal Q_\e=\{\vX\big|||\theta-\al||_{\dot H^{5/2}}+||y_s||_{\dot h^{3/2}}\le\e\}.
	\end{align}
	We claim that a universal constant $\e_*$ exists which will be clear below. For $\forall \vX\in \mathcal Q_{\e_*}$, it holds that
	\begin{align}
    \theta_\al(\al,t)\ge \frac{1}{2},\forall \al\in\bbT,\label{eq:beta00}\\
		\frac{1}{|\al_1-\al_2|}\big|\int^{\al_1}_{\al_2}\big(\cos(\theta),\sin(\theta)\big)d\al'\big|\ge \frac{1}{\pi},\forall \al_1,\al_2\in\bbT,\label{eq:betaal2}\\
		1+y_s(s,t)\ge \frac{3}{4},\forall s\in\bbT.\label{eq:betas2}
	\end{align}
	In fact, we have
	\begin{align*}
		&\big|\int^{\al_1}_{\al_2}\big(\cos(\theta),\sin(\theta)\big)d\al'\big|\\
		\ge&\big|\int^{\al_1}_{\al_2}\big(\cos(\al'),\sin(\al')\big)d\al'\big|-\big|\int^{\al_1}_{\al_2}\big(\cos(\theta)-\cos(\al'),\sin(\theta)-\sin(\al')\big)d\al'\big|\\
		\ge&\big(\frac{2}{\pi}-C_3\e_*)\big|\al_1-\al_2|,
	\end{align*}
	and
	\begin{align*}
    \theta_\al(\al,t)\ge 1-C_3\e_*,\\
		1+y_s(s,t)\ge 1-C_3\e_*,
	\end{align*}
	where $C_3>0$ is a universal constant coming from the Sobolev inequality. Hence, it suffices to take $\e_*=\min\{(3\pi C_3)^{-1},1\}$.

	Combing above uniform estimates and Theorem \ref{thm:locex}, we know that there exists a universal constant $T_0\in(0,1)$ such that for $\forall \vX_0\in \mathcal Q_{\e_*}$, \eqref{eq3} admits a unique solution $\vX$ stating from $\vX_0$ and satisfying
	\begin{align}\label{eq:th2}
		||\theta||_{L^\infty_{T_0}\dot H^{5/2}\cap L^2_{T_0}\dot H^4}+||y_s||_{L^\infty_{T_0}\dot h^{3/2}\cap L^2_{T_0}\dot h^2}\le 16(||\theta_0||_{\dot H^{5/2}}+||y_{0s}||_{\dot H^{3/2}})\le16\e_*.
	\end{align}
	Moreover, for $\forall t\in[0,T_0]$, \eqref{eq:beta00}-\eqref{eq:betas2} holds. Therefore, taking $\e_*=\min\{(48\pi C_3)^{-1},\frac{1}{64}\}$, we know that $\vX$ satisfies the assumption of Lemma \ref{lem:unith} with $T=T_0$, $R=16\e_*$, $\beta_1=\frac{1}{2\pi}$, and $\beta_2=\frac{1}{2}$, which are all universal constants.

	From the energy dissipation equation \eqref{eq:energy}, we have
	\begin{align*}
		\frac{d}{dt}\Big(\frac{1}{2\mathfrak s(t)}\int_\bbT \theta_\al^2(\al,t) d\al+\frac{\mathfrak s^2(t)}{2}\int_\bbT(1+y_s)^2ds+2\pi\lambda \mathfrak s(t)\Big)=-\int_{\mathbb R^2}|\nabla \vu|^2dx,
	\end{align*}
	which implies that
	\begin{align}\label{eq:endecayori}
		&\big(\frac{1}{2\fs}\int_{\bbT}\theta_\al^2d\al+\frac{\fs^2}{2}\int_\bbT(1+y_s)^2ds+2\pi\lambda\fs-2\pi-2\pi\lambda\big)\\
		\le&\big(\frac{1}{2\fs_0}\int_{\bbT}\theta_{0\al}^2d\al+\frac{\fs_0^2}{2}\int_\bbT(1+y_{0s})^2ds+2\pi\lambda\fs_0-2\pi-2\pi\lambda\big).\nonumber
	\end{align}
	With the help of Lemma \ref{lem:eqenergy} and \eqref{eq:con-th-S}, one can see that
	\begin{align*}
		||\theta-\al||_{L^\infty_{T_0}\dot H^1}+||y_s||_{L^\infty_{T_0}l^2}\le C_4(||\theta_0-\al||_{\dot H^1}+||y_{0s}||_{l^2}),
	\end{align*}
	where $C_4$ is a constant only depends on $\lambda$. Then, we choose
	\begin{align*}
		\e=(8C_4c_*^{3/2})^{-1}\e_*,
	\end{align*}
	which is the desired constant. Here $c_*=c_*(16\e_*,\lambda,\beta_1=\frac{1}{2\pi},\beta_2=\frac{1}{2})$ is defined in Lemma \eqref{lem:unith}. We claim that \eqref{eq:thc} holds if $||\theta-\al||_{\dot H^{5/2}}^2+||y_s||_{\dot h^{3/2}}^2\ge \frac{\e_*^2}{4}$. Indeed, as $||\theta-\al||_{L^\infty_{T_0}\dot H^1}+||y_s||_{L^\infty_{T_0}l^2}\le C_4\e$, it holds that
	\begin{align*}
		\frac{||\theta-\al||_{\dot{H}^{5/2}}^2+||y_s||_{\dot h^{3/2}}^{2}}{\big(||\theta-\al||_{\dot{H}^{1}}+||y_s||_{\dot h^{3/2}}^{2/3}||y_s||_{l^2}^{1/3}\big)^2}\ge\frac{\e_*^{2/3}}{8C_4^2\e^2\e_*^{-4/3}+2^{5/3}C_4^{2/3}\e^{2/3}}\ge c_*.
	\end{align*}

	Next, we show
	\begin{align}\label{eq:con-bnd}
		\big(||\theta-\al||_{\dot H^{5/2}}+||y_s||_{\dot h^{3/2}}\big)(t)\le \e_*,\quad\text{for}\ \forall t\in[0,T_0].
	\end{align}
	As $(||\theta-\al||_{\dot H^{5/2}}+||y_s||_{\dot h^{3/2}})(t)$ is continuous on $[0,T_0]$, if $(||\theta-\al||_{\dot H^{5/2}}+||y_s||_{\dot h^{3/2}})(t_1)>\e_*$ for some $t_1\in[0,T_0]$, there must be $t_0<t_1$ such that 
	\begin{align*}
		(||\theta-\al||_{\dot H^{5/2}}+||y_s||_{\dot h^{3/2}})(t_0)=\e_*&,\\
		\e_*\le(||\theta-\al||_{\dot H^{5/2}}+||y_s||_{\dot h^{3/2}})(t)\le 16\e_*&,\quad \forall t\in[t_0,t_1].
	\end{align*}
	Using Lemma \ref{lem:unith}, we have $(||\theta-\al||_{\dot H^{5/2}}+||y_s||_{\dot h^{3/2}})(t_1)\le e^{-[t_1-t_0]}\e_*$, which contradicts the assumption. Therefore, \eqref{eq:con-bnd} holds for all $t\in[0,T_0]$ and $\vX(T_0)\in \mathcal Q_{\e_*}$. We repeat the same steps like above and extend $\vX$ to $[0,2T_0]$, then existence of a global solution is established, and a universal estimate follows that
	\begin{align*}
		||\theta-\al||_{L^\infty_{+\infty}\dot H^{5/2}}+||y_s||_{L^\infty_{+\infty}\dot h^{3/2}}\le 8C_4c_*^{3/2}\e.
	\end{align*}
	That is to say, if $\vX_0(s)\in\mathcal Q_{\e}$, it holds that $\vX(s,t)\in\mathcal Q_{\e_*}$ for $\forall t>0$. 

	Accordingly, we arrive at the conclusion of this theorem.
\end{proof}

\section{Exponential Convergence to Equilibrium Configurations}
In this section, we prove that the global-in-time solution obtained in Theorem \ref{thm:global} converges exponentially to an equilibrium configuration as $t\to+\infty$. In what follows, we assume the area enclosed by the string is $\pi$.
\subsection{A lower bound of the rate of energy dissipation}
A key step to prove the exponential convergence is to establish a lower bound of $\int_{\mathbb R^2}|\nabla \vu|^2dx$ in terms of $(||\theta-\al||_{\dot H^2}+||y_s||_{l^2}+\fs-1)$.

Let $\mathcal Q_{\e_*}$ be defined as in \eqref{def:Q}, $\e_*>0$ is to be determined. Let $\Omega_{\vX}\subset \mathbb R^2$ denote the bounded open domain enclosed by $\vX(\bbT)$ where $\vX\in \mathcal Q_{\e_*}$. Define the collection of all such domains to be
\begin{align*}
	\mathfrak M_{\e_*}=\big\{\Omega_{\vX}\Subset\mathbb R^2:\pa \Omega_{\vX}=\vX(\bbT),\vX\in \mathcal Q_{\e_*}\big\}.
\end{align*}
We assume that $\e_*$ is sufficiently small such that domains in $\mathfrak M_{\e_*}$ satisfy the uniform $C^1$ regularity condition with uniform constants. Indeed, as $\vz\in H^{7/2}(\bbT)$, this is achievable due to the implicit function theorem and the Sobolev embedding $H^{3}(\bbT)\hookrightarrow C^{2}(\bbT)$.

For the velocity field $\vu$ determined by $\vX$, we define
\begin{align*}
	(\vu)_{\Omega_{\vX}}=|\Omega_{\vX}|^{-1}\int_{\Omega_{\vX}}\vu dx,\quad (\vu)_{\pa\Omega_{\vX}}=\frac{1}{2\pi}\int_{\bbT}\vu d\al.
\end{align*}
Then by the boundary trace theorem, it follows that
\begin{align*}
	\int_{\mathbb R^2}|\nabla \vu|^2dx&\ge\int_{\Omega_{\vX}}|\nabla \vu|^2 dx\ge C\fs\int_{\bbT}|\vu-(\vu)_{\Omega_{\vX}}|^2 d\al\\
	&\ge C\fs \int_{\bbT}|\vu-(\vu)_{\pa\Omega_{\vX}}|^2 d\al.
\end{align*}
Before showing the above value is bounded from below by $(||\theta-\al||_{\dot H^2}+||y_s||_{l^2}+\fs-1)$, we do some preparations.

Recalling that $\vX_*(\al)=\big(\cos(\al),\sin(\al)\big)$, a direct computation shows that
\begin{equation*}
  -\frac{\pa}{\pa_{\al'}}G(\vX_*(\al,t)-\vX_*(\al',t))=\frac{1}{8\pi}\left(
  	\begin{array}{cc}
  		-\sin(\al+\al')-\frac{1}{\tan(\frac{\al-\al'}{2})}&\cos(\al+\al')\\
  		\cos(\al+\al')&\sin(\al+\al')-\frac{1}{\tan(\frac{\al-\al'}{2})}
  	\end{array}
  \right).
\end{equation*}
We use $\vu_*$ to denote the velocity filed generated from $\vX_*$, and it is obvious that $\vu_*=0$. Recall that $\underline{\vt}(\al,t)=\int_{-\pi}^{\al}\vn(\al',t)d\al'$. Using the same technical in Lemma \ref{lem:ulinet}, we rewrite \eqref{eq:un} to
\begin{align}\label{eq:utov}
	&\vu(\al,t)\\
	=&\mathrm{p.v.}\int_\bbT -\frac{\pa}{\pa_{\al'}}G\big(\vz(\al,t)-\vz(\al',t)\big)\nonumber\\
	&\qquad\quad\cdot\big(\lambda \vt-\lambda\underline{\vt}-\frac{\theta_{\al\al}\vn}{\fs^2}-\frac{\theta_\al^2\vt}{2\fs^2}+\frac{\underline{\vt}}{2\fs^2}+\fs(1+y_s)\vt-\fs\underline{\vt}\big)(\al',t)d\al'\nonumber\\
	=&\mathrm{p.v.}\int_\bbT \Big(-\frac{\pa}{\pa_{\al'}}G\big(\vz(\al,t)-\vz(\al',t)\big)+\frac{\pa}{\pa_{\al'}}G\big(\vX_*(\al,t)-\vX_*(\al',t)\big)\Big)\nonumber\\
	&\qquad\quad\cdot\big(\lambda \vt-\lambda\underline{\vt}-\frac{\theta_{\al\al}\vn}{\fs^2}-\frac{\theta_\al^2\vt}{2\fs^2}+\frac{\underline{\vt}}{2\fs^2}+\fs(1+y_s)\vt-\fs\underline{\vt}\big)(\al',t)d\al'\nonumber\\
	&+\mathrm{p.v.}\int_\bbT -\frac{\pa}{\pa_{\al'}}G\big(\vX_*(\al,t)-\vX_*(\al',t)\big)\cdot\fs\big(y_s-y_\al\big)\vt(\al',t)d\al'\nonumber\\
	&+\mathrm{p.v.}\int_\bbT -\frac{\pa}{\pa_{\al'}}G\big(\vX_*(\al,t)-\vX_*(\al',t)\big)\nonumber\\
	&\qquad\quad\cdot\big(\lambda \vt-\lambda\underline{\vt}-\frac{\theta_{\al\al}\vn}{\fs^2}-\frac{\theta_\al^2\vt}{2\fs^2}+\frac{\underline{\vt}}{2\fs^2}+\fs(1+y_\al)\vt-\fs\underline{\vt}\big)(\al',t)d\al'\nonumber\\
	\eqdefa&\mathcal R_1+\mathcal R_2+\vv.\nonumber
\end{align}
Here $y_\al(\al,t)\eqdefa y_s(s,t)|_{s=\al}$ is a function defined on arc-length coordinate. One can regard it as $y_s(s,t)$ which have changed the notation of variable from $s$ to $\al$. We must be careful that $y_\al(\al,t)$ is different to $y_s(s(\al,t),t)$. We introduce such special notation to emphasize this point. When $||y_s||_{h^{3/2}}$ is small, there is little difference between $\al$ and $s(\al,t)$, so do $y_\al(\al,t)$ and $y_s(s(\al,t),t)$. We also note that $||y_\al||_{L^2}=||y_s||_{l^2}$.

Next, we introduce some auxiliary functions
\begin{equation}\label{eq:def-fun}
  	\begin{array}{l}
  		D=\theta(\al,t)-\al,\quad\theta_\eta(\al,t)=\al+\eta D(\al,t),\quad y_{\eta \al}(\al,t)=\eta y_\al(\al,t),\ \eta\in[0,1];\\
  		\vn_\eta(\al,t)=(-\sin(\theta_\eta(\al,t)),\cos(\theta_\eta(\al,t))),\ \vt_\eta(\al,t)=(\cos(\theta_\eta(\al,t)),\sin(\theta_\eta(\al,t)));\\
  		\vn_*(\al)=(-\sin(\al),\cos(\al)),\ \vt_*(\al)=(\cos(\al),\sin(\al));\\
   		\underline{\vt}_\eta=\int_{-\pi}^\al(-\sin(\theta_\eta(\al'')),\cos(\theta_\eta(\al'')))d\al''-\frac{\al}{2\pi}\int_{-\pi}^\pi(-\sin(\theta_\eta(\al'')),\cos(\theta_\eta(\al'')))d\al'';\\
   		\underline{Dn_*}(\al)=\int_{-\pi}^\al D(\al'')(-\cos(\al''),-\sin(\al''))d\al''-\frac{\al}{2\pi}\int_{-\pi}^\pi D(\al'')(-\cos(\al''),-\sin(\al''))d\al''.
  	\end{array}
\end{equation}
and an auxiliary velocity,
\begin{align*}
	\vv_\eta(\al,t)=&\int_\bbT -\frac{\pa}{\pa_{\al'}}G\big(\vX_*(\al,t)-\vX_*(\al',t)\big)\\
	&\quad\cdot\Big(\big(\lambda \vt_\eta-\lambda\underline{\vt}_\eta-\frac{\theta_{\eta\al\al}\vn_\eta}{\fs^2}-\frac{\theta_{\eta\al}^2\vt_\eta}{2\fs^2}+\frac{\underline{\vt}_\eta}{2\fs^2}+\fs(1+y_{\eta \al})\vt_\eta-\fs\underline{\vt}_\eta\big)(\al')\\
	&\qquad-\big(\lambda \vt_\eta-\lambda\underline{\vt}_\eta-\frac{\theta_{\eta\al\al}\vn_\eta}{\fs^2}-\frac{\theta_{\eta\al}^2\vt_\eta}{2\fs^2}+\frac{\underline{\vt}_\eta}{2\fs^2}+\fs(1+y_{\eta \al})\vt_\eta-\fs\underline{\vt}_\eta\big)(\al)\Big)d\al'\\
	\eqdefa&\int_\bbT h_\eta(\al,\al')d\al'.
\end{align*}
In the following lemma, we linearize $\vu$ and extract the principal part. The main idea is to apply Taylor expansion to the trigonometric functions.
\begin{lemma}\label{lem:linearu}
	Assume $\vX\in \mathcal Q_{\e_*}$ for a sufficiently small constant $\e_*$, it holds that
	\begin{align*}
		\vu(\al,t)=\pa_\eta|_{\eta=0}\vv_\eta(\al)+\mathcal R(\al),
	\end{align*}
	where 
	\begin{align}\label{eq:estR}
		||\mathcal R||_{L^2}\le C\e_*(||\theta-\al||_{\dot H^2}+||y_s||_{l^2}),
	\end{align}
	with some universal constant $C$.
\end{lemma}
\begin{proof}
	From \eqref{eq:utov}, it is easy to see that $\vv_\eta|_{\eta=0}=0$, $\vv_\eta|_{\eta=1}=\vv$. By the mean value theorem respect to $\eta$, for each fixed $\al$ there exists an $\eta^*=\eta^*(\al)\in[0,1]$ such that
\begin{align*}
	\vv(\al)=&\vv_\eta|_{\eta=1}(\al)-\vv_\eta|_{\eta=0}(\al)=\frac{\pa}{\pa_\eta}|_{\eta=\eta^*(\al)}\vv_\eta(\al)\\
	=&\frac{\pa}{\pa_\eta}|_{\eta=0}\vv_\eta(\al)+\frac{\pa}{\pa_\eta}|_{\eta=\eta^*(\al)}\vv_\eta(\al)-\frac{\pa}{\pa_\eta}|_{\eta=0}\vv_\eta(\al)\eqdefa \frac{\pa}{\pa_\eta}|_{\eta=0}\vv_\eta(\al)+\mathcal R_3(\al).
\end{align*}
	It remains to verify \eqref{eq:estR} for 
	\begin{align*}
		\mathcal R=\mathcal R_1+\mathcal R_2+\mathcal R_3.
	\end{align*}
	We see at once that $||\mathcal R_1||_{L^2}\le\e_*(||\theta-\al||_{\dot H^2}+||y_s||_{l^2})$. In the same way to Lemma \ref{lem:u-inf}, we have
	\begin{align}
		||\mathcal R_2||_{L^2}\le& C||\mathcal H\Big(\big(y_s(s(\al,t),t)-y_s(\al,t)\big)\vt\Big)||_{L^2}+C||y_s(s(\cdot,t),t)-y_s(\cdot,t)||_{L^2}\label{eq:esR2}\\
		\le&C||y_{ss}(\cdot,t)||_{l^2}^2||y_{s}(\cdot,t)||_{l^2}\le C\e_*||y_s||_{l^2}.\nonumber
	\end{align}
	By using the dominated convergence theorem, we have
	\begin{align*}
		&\pa_\eta\vv_\eta(\al)=\int_\bbT\pa_\eta h_\eta(\al,\al')d\al'\\
		=&\mathrm{p.v.}\int_\bbT -\frac{\pa}{\pa_{\al'}}G\big(\vX_*(\al,t)-\vX_*(\al',t)\big)\\
	&\quad\cdot\Big(\big(\lambda D\vn_\eta-\frac{D_{\al\al}\vn_\eta-\theta_{\eta\al\al}D\vt_\eta}{\fs^2}-\frac{2D\theta_{\eta\al}\vt_\eta+\theta_{\eta\al}^2D\vn_\eta}{2\fs^2}+\fs y_\al\vt_\eta+\fs(1+y_{\eta \al})D\vn_\eta\big)(\al')\\
	&\qquad\qquad\qquad+(\lambda+\fs- \frac{1}{2\fs^2})\big(\int_{-\pi}^{\al'}D\vt_\eta(\al'')d\al''-\frac{\al'}{2\pi}\int_{-\pi}^{\pi}D\vt_\eta(\al'')d\al''\big)\Big)d\al'.
	\end{align*}
	It follows that,
	\begin{align*}
		&\mathcal R_3(\al)=\frac{\pa}{\pa_\eta}|_{\eta=\eta^*(\al)}\vv_\eta(\al)-\frac{\pa}{\pa_\eta}|_{\eta=0}\vv_\eta(\al)\\
		=&\mathrm{p.v.}\int_\bbT -\frac{\pa}{\pa_{\al'}}G\big(\vX_*(\al,t)-\vX_*(\al',t)\big)\cdot\Big(\frac{{\eta^*}\theta_{\al\al}D\vt_{\eta^*}-{\eta^*} D_\al D\vt_{\eta^*}-{\eta^*} D_\al D\vn_{\eta^*}}{\fs^2}-\frac{{\eta^*}^2D_\al^2 D\vn_{\eta^*}}{2\fs^2}\\
	&\qquad+\fs{\eta^*} y_\al D\vn_{\eta^*}+\big(\lambda D-\frac{2D_{\al\al}+D}{2\fs^2}+\fs D\big)(\vn_{\eta^*}-\vn_*)+\big(\fs y_\al- \frac{D}{\fs^2}\big)(\vt_{\eta^*}-\vt_*)\\
	&\qquad+(\lambda+\fs- \frac{1}{2\fs^2})\big(\int_{-\pi}^{\al'}D(\vt_{\eta^*}-\vt_*)(\al'')d\al''-\frac{\al'}{2\pi}\int_{-\pi}^{\pi}D(\vt_{\eta^*}-\vt_*)(\al'')d\al''\big)\Big)(\al')d\al'.
	\end{align*}
	It is easy to check that
	\begin{align*}
		||\vn_{\eta^*}-\vn_*||_{H^1}+||\vt_{\eta^*}-\vt_*||_{H^1}\le C||\theta-\al||_{\dot H^1}.
	\end{align*}
	Accordingly, similar to \eqref{eq:esR2}, we conclude that
	\begin{align*}
		||\mathcal R_3||_{L^2}\le&C\e_*(||\theta-\al||_{\dot H^2}+||y_s||_{l^2}).
	\end{align*}
	This completes the proof.
\end{proof}

Next, we state an important observation which reflect a special property of the tangent angle functions. 

For any function $f$ defined on $\bbT$, we use $a_k$ and $b_k$ to denote the Fourier coefficients such that
\begin{align*}
	a_k(f)=&\int^\pi_{-\pi}\cos(k\al)f(\al)d\al,\ b_k(f)=\int^\pi_{-\pi}\sin(k\al)f(\al)d\al.
\end{align*}
Recalling that $D=\theta-\al$, we have the following inequalities.
\begin{lemma}\label{lem:normD}Let $\gamma\ge0$, for $\vX\in \mathcal Q_{\e_*}$ with $\e_*$ small enough, there exists a universal constant $C$ such that 
	\begin{align*}
		a_1^2(D)+b_1^2(D)\le C\e_*^2||D-\bar D||_{L^2}^2,\quad||D||^2_{\dot H^\gamma}\le \frac{1}{\pi(1-C\e_*^2)}\sum^{+\infty}_{k=2}(a_k^2(D)k^{2\gamma}+b_k^2(D)k^{2\gamma}).
	\end{align*}
where $\bar D$ is the mean value of $D$.
\end{lemma}
\begin{proof}
	From the assumption, $\theta$ satisfies $\int^\pi_{-\pi}\cos\big(\theta(\al)\big)d\al=0$ and $\theta(\pi)-\theta(-\pi)=2\pi$. It also holds that
	\begin{align*}
		\int^\pi_{-\pi}\theta_\al(\al)\cos(\theta(\al)-\bar \theta) d\al=\int^{\theta(\pi)}_{\theta(-\pi)}\cos(\theta -\bar \theta) d\theta=0.
	\end{align*}
	Then we have
	\begin{align*}
		\int^\pi_{-\pi}\big(\theta_\al(\al)-1\big)\cos(\theta(\al)-\bar\theta) d\al=0.
	\end{align*}
	Therefore, one can deduce that
	\begin{align*}
		b_1(D)=&\int^\pi_{-\pi}(\theta-\al)\sin(\al) d\al=\int^\pi_{-\pi}(\theta_\al-1)\cos(\al) d\al\\
		=&\int^\pi_{-\pi}(\theta_\al-1)\big(\cos(\al)-\cos(\theta-\bar\theta)\big) d\al\\
		=&2\int^\pi_{-\pi}(\theta_\al-1)\sin(\frac{\al+\theta-\bar\theta}{2})\sin(\frac{\theta-\bar\theta-\al}{2}) d\al\\
		\le&C\big(\int^\pi_{-\pi}(\theta_\al-1)^2 d\al\big)^{1/2}\big(\int^\pi_{-\pi}\sin^2(\frac{\theta-\bar\theta-\al}{2}) d\al\big)^{1/2}\\
		\le&C||\theta_\al-1||_{L^2}\big(\int^\pi_{-\pi}|\theta-\bar\theta-\al|^2 d\al\big)^{1/2}\le C||D||_{\dot H^1}||D-\bar D||_{L^2}.
	\end{align*}
	Similarly, one has
	\begin{align*}
		a_1(D)\le C||D||_{\dot H^1}||D-\bar D||_{L^2}.
	\end{align*}
	For any integrable real-valued function $f$, it holds that
	\begin{align*}
		a_k(f)=a_{-k}(f),\quad b_k(f)=-b_{-k}(f).
	\end{align*}
	From the above results, we have
	\begin{align*}
		(1-C\e^2_*)||D||^2_{\dot H^\gamma}\le \frac{1}{\pi}\sum^{+\infty}_{k=2}\big(a_k^2(D)k^{2\gamma}+b_k^2(D)k^{2\gamma}\big),
	\end{align*}
	which is the desired conclusion.
\end{proof}
Then, we can give the lower bound of the rate of energy dissipation.
\begin{lemma}\label{lem:normu}
	There exists a universal $\e_*>0$ such that
	\begin{align}
		||\vu-\vu_{\pa\Omega_\vX}||_{L^2(\bbT)}\ge C \big(||\theta-\al||_{\dot H^2}+||y_s||_{l^2}+\lambda^2(\fs-1)\big),\ \forall \vX\in \mathcal Q_{\e_*},\label{eq:lowerbd}
	\end{align}
	where $\mathcal Q_{\e_*}$ is defined in \eqref{def:Q} and $C>0$ is a universal constant. Since $\fs\ge1$, \eqref{eq:lowerbd} implies that
	\begin{align*}
		\int_{\mathbb R^2}|\nabla \vu|^2dx\ge C\big((||\theta-\al||_{\dot H^2}+||y_s||_{l^2}+\lambda^2(\fs-1)\big),\ \forall\vX\in \mathcal Q_{\e_*}.
	\end{align*}
\end{lemma}
\begin{proof}
	We first focus on the leading term of $\vu$. From the definitions, we have
\begin{align*}
	\pa_\eta|_{\eta=0}\vv_\eta(\al)=&\mathrm{p.v.}\int_\bbT\frac{1}{8\pi}\left(
  	\begin{array}{cc}
  		-\sin(\al+\al')-\frac{1}{\tan(\frac{\al-\al'}{2})}&\cos(\al+\al')\\
  		\cos(\al+\al')&\sin(\al+\al')-\frac{1}{\tan(\frac{\al-\al'}{2})}
  	\end{array}
  \right)\\
  	&\qquad\cdot\Big((\lambda+\fs- \frac{1}{2\fs^2})(D\vn_*-\underline{Dn_*})-\frac{D_{\al\al}\vn_*+D_\al\vt_*}{\fs^2}+\fs y_\al\vt_*\Big)(\al')d\al'.
\end{align*}
Direct calculations show that
\begin{align*}
	&\int^\pi_{-\pi}
	  \left(
  	\begin{array}{cc}
  		-\sin(\al+\al')&\cos(\al+\al')\\
  		\cos(\al+\al')&\sin(\al+\al')
  	\end{array}
  \right)\cdot
	  \underline{Dn_*}(\al')d\al'\\
	=&\int^\pi_{-\pi}
	  \pa_{\al'}\left(
	  	\begin{array}{ll}
	  		\cos(\al'+\al)&\sin(\al'+\al)\\
	  		\sin(\al'+\al)&-\cos(\al'+\al)
	  	\end{array}
	  \right)\\
	  &\cdot\Big(\int_{-\pi}^{\al'} D(\al'')\left(
	  	\begin{array}{l}
	  		-\cos(\al'')\\
	  		-\sin(\al'')
	  	\end{array}
	  \right)d\al''-\frac{\al'}{2\pi}\int_{-\pi}^\pi D(\al'')\left(
	  	\begin{array}{l}
	  		-\cos(\al'')\\
	  		-\sin(\al'')
	  	\end{array}
	  \right)d\al''\Big)d\al'\\
	=&\int^\pi_{-\pi}
	  \left(
	  	\begin{array}{ll}
	  		\cos(\al'+\al)&\sin(\al'+\al)\\
	  		\sin(\al'+\al)&-\cos(\al'+\al)
	  	\end{array}
	  \right)\cdot
	  \left(
	  	\begin{array}{l}
	  		D(\al')\cos(\al')-\int^\pi_{-\pi}D(\al'')\cos(\al'')d\al''\\
	  		D(\al')\sin(\al')-\int^\pi_{-\pi}D(\al'')\sin(\al'')d\al''
	  	\end{array}
	  \right)d\al'\\
	  =&\int^\pi_{-\pi}
	  \left(
	  	\begin{array}{l}
	  		\cos(\al)D(\al')\\
	  		\sin(\al)D(\al')
	  	\end{array}
	  \right)d\al'\\
	  =&\left(
      \begin{array}{l}
        2\pi\cos(\al)\bar \theta\\
        2\pi\sin(\al)\bar \theta
      \end{array}
    \right),
\end{align*}
and
\begin{align*}
	&\int^\pi_{-\pi}
	  \left(
  	\begin{array}{cc}
  		-\sin(\al+\al')&\cos(\al+\al')\\
  		\cos(\al+\al')&\sin(\al+\al')
  	\end{array}
  \right)\cdot
	  D\vn_*(\al')d\al'=\left(
      \begin{array}{l}
        2\pi\cos(\al)\bar \theta\\
        2\pi\sin(\al)\bar \theta
      \end{array}
    \right).
\end{align*}
Here $\bar\theta$ is the mean value of $\theta$. For the same reason, one can deduce that
\begin{align*}
	&\int^\pi_{-\pi}
	  \left(
  	\begin{array}{cc}
  		-\sin(\al+\al')&\cos(\al+\al')\\
  		\cos(\al+\al')&\sin(\al+\al')
  	\end{array}
  \right)\\
  &\qquad\quad\cdot\Big((\lambda+\fs- \frac{1}{2\fs^2})(D\vn_*-\underline{Dn_*})-\frac{D_{\al\al}\vn_*+D_\al\vt_*}{\fs^2}+\fs y_\al\vt_*\Big)(\al')d\al'=0.
\end{align*}
Therefore,
\begin{align*}
	\pa_\eta|_{\eta=0}\vv_\eta(\al)=-\frac{1}{4}\mathcal H\Big((\lambda+\fs- \frac{1}{2\fs^2})(D\vn_*-\underline{Dn_*})-\frac{D_{\al\al}\vn_*+D_\al\vt_*}{\fs^2}+\fs y_\al\vt_*\Big)(\al).
\end{align*}

For $k\neq0$, we have
\begin{align*}
	a_k(D\vn_*)=&\int^\pi_{-\pi}\cos(k\al)\left(
	  	\begin{array}{l}
	  		-D(\al)\sin(\al)\\
	  		D(\al)\cos(\al)
	  	\end{array}
	  \right)d\al\\
	  =&\frac{1}{2}\int^\pi_{-\pi}\left(
	  	\begin{array}{l}
	  		D(\al)\big(\sin((k-1)\al)-\sin((k+1)\al)\big)\\
	  		D(\al)\big(\cos((k-1)\al)+\cos((k+1)\al)\big)
	  	\end{array}
	  \right)d\al\\
	  =&\frac{1}{2}\left(
	  	\begin{array}{l}
	  		b_{k-1}(D)-b_{k+1}(D)\\
	  		a_{k-1}(D)+a_{k+1}(D)
	  	\end{array}
	  \right),\\
	b_k(D\vn_*)=&\frac{1}{2}\left(
	  	\begin{array}{l}
	  		-a_{k-1}(D)+a_{k+1}(D)\\
	  		b_{k-1}(D)+b_{k+1}(D)
	  	\end{array}
	  \right),
\end{align*}
and
\begin{align*}
	\mathcal H(D\vn_*)(\al)=&\sum_{k\in\mathbb Z/0}\Big(-i\frac{\text{sgn}k}{2\pi}a_k(D\vn_*)-\frac{\text{sgn} k}{2\pi}b_k(D\vn_*)\Big)e^{ik\al}\\
	=&\sum_{k\in\mathbb Z/0}\frac{1}{4\pi}\left(
	  	\begin{array}{l}
	  		\frac{k}{|k|}a_{k-1}(D)-\frac{k}{|k|}a_{k+1}(D)-i\frac{k}{|k|}b_{k-1}(D)+i\frac{k}{|k|}b_{k+1}(D)\\
	  		-\frac{k}{|k|}b_{k-1}(D)-\frac{k}{|k|}b_{k+1}(D)-i\frac{k}{|k|}a_{k-1}(D)-i\frac{k}{|k|}a_{k+1}(D)
	  	\end{array}
	  \right).
\end{align*}
In this way, we derive that
\begin{align*}
	\pa_\eta|_{\eta=0}\vv_\eta(\al)=\frac{1}{16\pi}\sum_{k\in\mathbb Z/0}(N_k,M_k),
\end{align*}
where
\begin{align*}
	N_k=&\Big[\Big(-(\lambda+\fs- \frac{1}{2\fs^2})\frac{k-1}{|k|}-\frac{1}{\fs^{2}} \frac{k(k-1)^{2}}{|k|}+\frac{1}{\fs^{2}} \frac{k(k-1)}{|k|}\Big)a_{k-1}(D)+\fs\frac{k}{|k|} b_{k-1}(y_\al)\\
	&+\Big((\lambda+\fs- \frac{1}{2\fs^2})\frac{k+1}{|k|}+\frac{1}{\fs^{2}} \frac{k(k+1)^{2}}{|k|}+\frac{1}{\fs^{2}} \frac{k(k+1)}{|k|}\Big)a_{k+1}(D)+\fs\frac{k}{|k|} b_{k+1}(y_\al)\\
	&+i\Big((\lambda+\fs- \frac{1}{2\fs^2})\frac{k-1}{|k|}+\frac{1}{\fs^{2}} \frac{k(k-1)^{2}}{|k|}-\frac{1}{\fs^{2}} \frac{k(k-1)}{|k|}\Big)b_{k-1}(D)+i\fs\frac{k}{|k|} a_{k-1}(y_\al)\\
	&+i\Big(-(\lambda+\fs- \frac{1}{2\fs^2})\frac{k+1}{|k|}-\frac{1}{\fs^{2}} \frac{k(k+1)^{2}}{|k|}-\frac{1}{\fs^{2}} \frac{k(k+1)}{|k|}\Big)b_{k+1}(D)+i\fs\frac{k}{|k|} a_{k+1}(y_\al)
	\Big]e^{ik\al},\\
	M_k=&\Big[\Big((\lambda+\fs- \frac{1}{2\fs^2})\frac{k-1}{|k|}+\frac{1}{\fs^{2}} \frac{k(k-1)^{2}}{|k|}-\frac{1}{\fs^{2}} \frac{k(k-1)}{|k|}\Big)b_{k-1}(D)+\fs\frac{k}{|k|} a_{k-1}(y_\al)\\
	&+\Big((\lambda+\fs- \frac{1}{2\fs^2})\frac{k+1}{|k|}+\frac{1}{\fs^{2}} \frac{k(k+1)^{2}}{|k|}+\frac{1}{\fs^{2}} \frac{k(k+1)}{|k|}\Big)b_{k+1}(D)-\fs\frac{k}{|k|} a_{k+1}(y_\al)\\
	&+i\Big((\lambda+\fs- \frac{1}{2\fs^2})\frac{k-1}{|k|}+\frac{1}{\fs^{2}} \frac{k(k-1)^{2}}{|k|}-\frac{1}{\fs^{2}} \frac{k(k-1)}{|k|}\Big)a_{k-1}(D)-i\fs\frac{k}{|k|} b_{k-1}(y_\al)\\
	&+i\Big((\lambda+\fs- \frac{1}{2\fs^2})\frac{k+1}{|k|}+\frac{1}{\fs^{2}} \frac{k(k+1)^{2}}{|k|}+\frac{1}{\fs^{2}} \frac{k(k+1)}{|k|}\Big)a_{k+1}(D)+i\fs\frac{k}{|k|} b_{k+1}(y_\al)
	\Big]e^{ik\al}.
\end{align*}
From the above expression, it is obvious that
\begin{align*}
	\int_\bbT \pa_\eta|_{\eta=0}\vv_\eta(\al)d\al=0.
\end{align*}

Accordingly, it holds that
\begin{align*}
	&\int^\pi_{-\pi}\big|\pa_\eta|_{\eta=0}\vv_\eta(\al)\big|^2d\al\\
	=&\sum^{+\infty}_{k=1}\frac{1}{32\pi}\Big[\Big(\big(-(\lambda+\fs- \frac{1}{2\fs^2})\frac{k-1}{|k|}-\frac{1}{\fs^{2}} \frac{k(k-1)^{2}}{|k|}+\frac{1}{\fs^{2}} \frac{k(k-1)}{|k|}\big)a_{k-1}(D)+\fs\frac{k}{|k|} b_{k-1}(y_\al)\Big)^2\\
	&\qquad+\Big(\big((\lambda+\fs- \frac{1}{2\fs^2})\frac{k+1}{|k|}+\frac{1}{\fs^{2}} \frac{k(k+1)^{2}}{|k|}+\frac{1}{\fs^{2}} \frac{k(k+1)}{|k|}\big)a_{k+1}(D)+\fs\frac{k}{|k|} b_{k+1}(y_\al)\Big)^2\\
	&\qquad+\Big(\big((\lambda+\fs- \frac{1}{2\fs^2})\frac{k-1}{|k|}+\frac{1}{\fs^{2}} \frac{k(k-1)^{2}}{|k|}-\frac{1}{\fs^{2}} \frac{k(k-1)}{|k|}\big)b_{k-1}(D)+\fs\frac{k}{|k|} a_{k-1}(y_\al)\Big)^2\\
	&\qquad+\Big(\big((\lambda+\fs- \frac{1}{2\fs^2})\frac{k+1}{|k|}+\frac{1}{\fs^{2}} \frac{k(k+1)^{2}}{|k|}+\frac{1}{\fs^{2}} \frac{k(k+1)}{|k|}\big)b_{k+1}(D)-\fs\frac{k}{|k|} a_{k+1}(y_\al)\Big)^2\Big]\\
	=&\frac{1}{32\pi}\Big(\frac{3}{2}(\lambda+\fs-\frac{1}{2\fs^2})a_{1}(D)-\fs b_{1}(y_\al)\Big)^2+\frac{1}{32\pi}\Big(\frac{3}{2}(\lambda+\fs-\frac{1}{2\fs^2})b_{1}(D)+\fs a_{1}(y_\al)\Big)^2\\
	&+\frac{10}{32\pi}\big(\frac{2}{3}(\lambda+\fs)+\frac{5}{3\fs^2}\big)^2\big(a_{2}^2(D)+b_2^2(D)\big)+\frac{2}{32\pi}\fs^2\big(a_2^2(y_\al)+b_2^2(y_\al)\big)\\
	&\qquad\qquad\qquad\quad+\frac{4}{32\pi}\fs\big(\frac{2}{3}(\lambda+\fs)+\frac{5}{3\fs^2}\big)\big(a_{2}(D) b_{2}(y_\al)-b_{2}(D) a_{2}(y_\al)\big)\\
	&+\sum^{+\infty}_{k=3}\frac{1}{32\pi}\Big[\Big(\big(\frac{k(\lambda+\fs)}{k-1}+\frac{1}{2\fs^2}\frac{2k^3-3k}{k-1}\big)^2+\big(\frac{k(\lambda+\fs)}{k+1}+\frac{1}{2\fs^2}\frac{2k^3-3k}{k+1}\big)^2\Big)\big(a_k^2(D)+b_k^2(D)\big)\\
	&\qquad+2\fs^2\big(a_k^2(y_\al)+b_k^2(y_\al)\big)+2\fs\big(\frac{2k(\lambda+\fs)}{k^2-1}+\frac{1}{\fs^2}\frac{2k^3-3k}{k^2-1}\big)\big(a_k(D)b_k(y_\al)-b_k(D)a_k(y_\al)\big)\Big]\\
	\ge&\frac{\fs^2}{64\pi}\big(a_1^2(y_\al)+b_1^2(y_\al)\big)-\frac{1}{8\pi}(\lambda+\fs-\frac{1}{2\fs^2})^2\big(a_1^2(D)+b_1^2(D)\big)\\
	&+\frac{\fs^2}{32\pi}\big(a_2^2(y_\al)+b_2^2(y_\al)\big)+\big(\frac{(\lambda+\fs)^2}{12\pi}+\frac{1}{2\pi\fs^4}\big)\big(a_2^2(D)+b_2^2(D)\big)\\
	&+\sum^{+\infty}_{k=3}\frac{1}{32\pi}\Big[\big(\frac{k(\lambda+\fs)}{k-1}+\frac{1}{2\fs^2}\frac{2k^2(k-1)+2k^2-3k}{k-1}\big)^2\big(a_k^2(D)+b_k^2(D)\big)\\
	&\quad+\fs^2\big(a_k^2(y_\al)+b_k^2(y_\al)\big)+\fs\Big((\lambda+\fs)\frac{k(k-3)}{k^2-1}+\frac{1}{2\fs^2}\frac{(2k^3-3k)(k-3)}{k^2-1}\Big)\big(a_k^2(D)+b_k^2(D)\big)\Big]\\
	\ge&\frac{\fs^2}{64\pi}\sum^{+\infty}_{k=1}\big(a_k^2(y_\al)+b_k^2(y_\al)\big)-\frac{1}{8\pi}(\lambda+\fs-\frac{1}{2\fs^2})^2\big(a_1^2(D)+b_1^2(D)\big)\\
	&+\frac{1}{32\pi}\sum^{+\infty}_{k=2}\big((\lambda+\fs)^2+\frac{k^4}{\fs^4}\big)\big(a_k^2(D)+b_k^2(D)\big)\\
	\ge&C\big(\fs^2||y_\al||_{L^2}^2+\frac{1}{\fs^4}||\theta-\al||_{\dot H^2}^2+(\lambda+\fs)^2||\theta-\bar\theta-\al||_{L^2}^2\big),
\end{align*}
the last inequality follows from Lemma \ref{lem:normD}. By using Lemma \ref{lem:linearu} and Lemma \ref{lem:con-S-th}, we have
\begin{align*}
	||\vu-\vu_{\pa\Omega_\vX}||_{L^2}^2\ge&C||\pa_\eta|_{\eta=0}\vv_\eta||_{L^2}^2-C||\mathcal R-\overline{\mathcal R}||_{L^2}^2\\
	\ge&C \big(||y_s||_{l^2}^2+||\theta-\al||_{\dot H^2}^2+\lambda^2(\fs-1)\big),
\end{align*}
which actually ensures the conclusion.
\end{proof}
Now, we are in a position to Theorem \ref{thm:conver}.
\begin{proof}[Proof of Theorem \ref{thm:conver}]
	From \eqref{eq:energy} and Lemma \ref{lem:normu}, taking $\e$ sufficiently small, we deduce that
	\begin{align*}
		&\frac{d}{dt}\big(\frac{1}{2\fs}\int_{\bbT}\theta_\al^2d\al+\frac{\fs^2}{2}\int_\bbT(1+y_s)^2ds+2\pi\lambda\fs-2\pi-2\pi\lambda\big)\\
		=&-\int_{\mathbb R^2}|\nabla \vu|^2dx\\
		\le&-C\Big(\int_\bbT(\theta_\al-1)^2d\al+\int_\bbT y_s^2ds+\lambda(\fs-1)\Big),
	\end{align*}
	where $C>0$ is a universal constant. From Lemma \ref{lem:eqenergy} one can see that there exists a constant $\gamma_*$ such that
	\begin{align*}
		||\theta-\al||_{\dot H^1}(t)+||y_s||_{l^2}(t)+\lambda|\fs(t)-1|\le C e^{-3\gamma_*t}\big(||\theta_0-\al||_{\dot H^1}+||y_{0s}||_{l^2}+\lambda|\fs_0-1|\big).
	\end{align*}
	In the proof of Theorem \ref{thm:global}, we know that 
	\begin{align*}
		||\theta-\al||_{\dot{H}^{5/2}}^2+||y_s||_{\dot h^{3/2}}^{2}\le c_*\big(||\theta-\al||_{\dot{H}^{1}}+||y_s||_{\dot h^{3/2}}^{2/3}||y_s||_{l^2}^{1/3}\big)^2
	\end{align*}
 	for some constant $c_*$, therefore 
 	\begin{align*}
 		||\theta-\al||_{\dot{H}^{5/2}}(t)+||y_s||_{\dot h^{3/2}}(t)\le Ce^{-\gamma_*t}.
 	\end{align*}
 	Recalling \eqref{eq:modify-btho} and \eqref{eq:us}, one can see that there exists
 	\begin{align*}
 		\lim_{t\to+\infty}\bar\theta(t)\to \theta_{\infty},\ \lim_{t\to+\infty}\int_0^t\vu(\vX(-\pi,t'),t')dt'\to x_{\infty}.
 	\end{align*}
 	This completes the proof.
\end{proof}
\section{Remarks on some general cases}
In this paper, we study the Stokes immersed boundary problem with bending and stretching energy and  establish the global well-posedness in Sobolev space. We give a proof for the case with elastic coefficients $c_1=c_3=1$, surface tension $\lambda\ge0$, the spontaneous curvature $B=0$, and the optimal perimeter $\fs_{op}=0$. The method developed in this paper still works for some general cases.
\subsection{Case without stretching energy or without bending energy}
For the case without bending energy and surface tension ($c_1=0,\lambda=0$), we go back to the model studied by Lin-Tong in \cite{LT}. In this case, the contour dynamic system is 
\begin{equation*}
  \left\{
    \begin{array}{ll}
   \theta_t(\al,t)=-\frac{1+y_s}{4}\mathcal H(\theta_\al)(\al,t)+g_\theta(\al,t),& \theta(\al,0)=\theta_0(\al),\\
  y_{st}(s,t)=-\frac{1}{4}\mathfrak h (y_{ss})(s,t)+g_y(s,t),& y_s(s,0)=y_{0s}(s,0),\\
  \mathfrak s_t(t)=-\frac{1}{2\pi}\int^\pi_{-\pi}\theta_\alpha \vu\cdot\vn d\al,& \mathfrak s(0)=\mathfrak s_0.     
    \end{array}
  \right.
\end{equation*}
It is also a parabolic system. The energy dissipation identity is
\begin{align*}
    \frac{d}{dt}\Big(\frac{\mathfrak s^2}{2}\int_\bbT y_s^2ds+\pi(\fs-1)\Big)=-\int_{\mathbb R^2}|\nabla \vu(x,t)|^2dx.
\end{align*} 
We can see that there is no $\theta$ term. However, Lemma \ref{lem:con-S-th} (Fuglede’s isoperimetric inequality) shows that $||\theta-\bar\theta-\al||_{L^2}^2\le C(\fs-1)$ when $||\theta-\al||_{\dot H^1}$ is small enough. Then we can use a similar method to get the global well-posedness of this problem. Furthermore, our method also works for the case ($c_1=0,c_3>0,\lambda>0$).

For the case without stretching energy ($c_3=0$), the force applied on the membranes only depend on the shape of the string. We need only to use $\vz(\al)$ to parametrize the curve. However, $\vz_t(\al,t)\cdot\vn=\vu(\vz(\al,t),t)\cdot\vn$, the velocity field of the fluid on the string determines the evolution of the membrane's shape, so the free boundary problem \eqref{eq1} is equivalent to the following equation:
\begin{align*}
 	\vz_t(\al,t)\cdot\vn=\vn\cdot\int_{\bbT}G\big(\vz(\al,t)-\vz(\al',t)\big)\cdot F(\al',t)|\vz_\al(\al',t)|d\al',
\end{align*}
where
\begin{align*}
	F=\lambda \kappa\vn-(\frac{1}{\fs^2}\pa_\al^2 \kappa+\frac{1}{2}\kappa^3)\vn.
\end{align*}
Next, one can derive the contour dynamic system base on the evolution equations of $\theta$ and $\fs$:
\begin{equation*}
  \left\{
    \begin{array}{ll}
   \theta_t(\al,t)=\frac{1}{4\fs^3}\mathcal H(\theta_{\al\al\al})(\al,t)+g_\theta(\al,t),& \theta(\al,0)=\theta_0(\al),\\
 \fs_t(t)=-\int^\pi_{-\pi}\theta_\alpha \vu\cdot\vn d\al,& \fs(0)=\fs_0,   
    \end{array}
  \right.
\end{equation*}
where $g_\theta$ is the error term. Instead of \eqref{eq:barct}, one can choose
\begin{align*}
	\overline \cT(t)=\frac{1}{4\pi^2}\int^\pi_{-\pi} \vu\cdot\vn\theta_\al d\al\int^\pi_{-\pi}\al\theta_\al d\al- \frac{1}{2\pi}\int^\pi_{-\pi}\theta_\al\int^\al_{-\pi} \vu(\al')\cdot\vn(\al')\theta_\al(\al')d\al' d\al.
\end{align*}
Thus we always have $\int_\bbT g_\theta(\al,t)d\al=0$.

One can study the case with none but surface tension ($c_1=0,c_3=0,\lambda>0$) in the same way to the case ($c_3=0$), and Lemma \ref{lem:con-S-th} will be used as well. We refer the readers to \cite{PS1,PS2,MP} for more results about this kind of problem. 
\subsection{Case with $B>0$ and $\fs_{op}>0$}
When the optimal perimeter $\fs_{op}=0$, the string has zero resting length in the stretching force-free state, which is mentioned in the last section of \cite{LT}. For the case with $B>0$ and $\fs_{op}>0$, the free energy is 
\begin{align*}
	E=\frac{\mathfrak c_1}{2}\int_\bbT (\kappa-B)^2|\vX_s|ds+\frac{\mathfrak c_3}{2}\int_\bbT(|\vX_s|-\fs_{op})^2ds+\lambda \int_\bbT|\vX_s|ds,
\end{align*}
and the elastic force applied on the string has the following formulation:
\begin{align}\label{eq:force-ph}	
\vF=&\underbrace{(1+y_s)\lambda \theta_\al\vn}_{\text{surface tension}}+\underbrace{\mathfrak c_1(1+y_s)\big(B\theta_\al-\frac{1}{2}B^2\theta_\al- \frac{1}{\fs^2}\theta_{\al\al\al}-\frac{1}{2\fs^2}\theta_\al^3\big)\vn}_{\text{bending force}}\\
&+\underbrace{\mathfrak c_3\big(\fs y_{ss}\vt+\frac{\fs(1+y_s)- \fs_{op}}{\fs(1+y_s)}\fs(1+y_s)^2\theta_{\al}\vn\big)}_{\text{stretching force}}.\nonumber
\end{align}
One can see that the main terms of \eqref{eq:force-or} and \eqref{eq:force-ph} are the same in both normal and tangential directions, and the difference between \eqref{eq:force-or} and \eqref{eq:force-ph} are all lower order terms. As a result, the contour dynamic system for the case with $B>0$ and $\fs_{op}>0$ is still a parabolic system, and the local well-posedness can be establish in the same way. 

Let
\begin{align*}
	\fs_*=\Big(-\frac{1}{2\pi}\int_\bbT\vX_0\cdot\vX_{0s}^\perp ds\Big)^{1/2}.
\end{align*}
When $B> \frac{1}{\fs_*}$ or $\fs_{op}>\fs_*$, the equilibrium state may no longer be an evenly parametrized circle, and the global results may be different to the case with $(B=0,\fs_{op}=0)$. It is interesting to study the equilibrium states of strings for arbitrarily chosen $B$ and $\fs_{op}$, and there are some related works about this kind of problems \cite{HO,NOO}.

\begin{appendix}
\section{A Priori Estimates Involving $\mathcal L$ and $\mathfrak L$}
Recalling that
\begin{align*}
	\mathcal L(\theta)(\al,t)=\frac{1}{4\fs^3(t)}\mathcal H(\theta_{\al\al\al})(\al,t),\ \mathfrak L(y_s)(s,t)=-\frac{1}{4}\mathfrak h(y_{ss})(s,t),
\end{align*}
we have following estimates under the assumption that $c\ge \fs(t)\ge c_0>0$ for $t\in[0,T]$. 
\begin{lemma}
	For $\forall \phi\in H^\gamma(\bbT), w\in h^\gamma(\bbT)$ with arbitrary $\gamma\in\mathbb R_+$, it holds that
	\begin{equation*}
	  	\begin{array}{ll}
	  		(1)||e^{t \mathcal{L}} \phi||_{ \dot{H}^{\gamma}} \le e^{-t/(4c^3)}||\phi||_{\dot{H}^{\gamma}},& ||e^{t \mathfrak{L}} w||_{ \dot{h}^{\gamma}} \le e^{-t/4}||w||_{\dot{h}^{\gamma}};\\
	  		(2)e^{t \mathcal{L}} \phi \in C\left([0,+\infty) ; H^{\gamma}(\mathbb{T})\right),& e^{t \mathfrak{L}} w \in C\left([0,+\infty) ; h^{\gamma}(\mathbb{T})\right);\\
	  		(3)e^{t \mathcal{L}} \phi \rightarrow \phi \text { in } H^{\gamma}(\mathbb{T}) \text { as } t \rightarrow 0^{+},& e^{t \mathfrak{L}} w \rightarrow w \text { in } h^{\gamma}(\mathbb{T}) \text { as } t \rightarrow 0^{+}.
	  	\end{array}
	\end{equation*}
\end{lemma}
\begin{lemma}\label{lem:ex}
	Assume $T>0$, let $g\in L^2_TH^\gamma(\bbT)$. The model equation
	\begin{align*}
		(\theta(\al,t)-\al)_t=\mathcal{L}(\theta)(\al,t)+g(\al,t),\  \theta(\al,0)=\theta_0(\al),\ \al\in\bbT,t\ge0,
	\end{align*}
	admits a unique solution $\theta-\al\in L^\infty_{T}H^{\gamma+\frac{3}{2}}(\bbT)\cap L^2_{T} H^{\gamma+3}(\bbT)$ with $\theta_t\in L^2_{T} H^\gamma(\bbT)$. Furthermore, this solution satisfies
\begin{align*}
	||\theta-\al||_{\dot{H}^{\gamma+3/2}}^2(t) \le e^{-t / (4c^3)}||\theta_{0}-\al||_{\dot{H}^{\gamma+3 / 2}}^2+4c^{3}||g||_{L_t^{2} \dot H^{\gamma}}^2,\ \forall t\in[0,T],\\
	\frac{d}{dt}||\theta-\al||_{\dot{H}^{\gamma+3/2}}^2(t)\le-\frac{1}{4c^3}||\theta-\al||_{\dot{H}^{\gamma+3}}^2(t)+4c^3||g||_{\dot H^\gamma}^{2}(t),\ \forall t\in[0,T].
\end{align*}
Hence,
\begin{align*}
	||\theta-\al||_{L_t^{\infty} \dot{H}^{\gamma+3 / 2} \cap L_t^{2} \dot{H}^{\gamma+3}} \le (1+2c^{3/2})||\theta_{0}-\al||_{\dot{H}^{\gamma+3 / 2}}+C||g||_{L_t^{2} \dot{H}^{\gamma}(\mathbb{T})}.
\end{align*}
It also holds that
\begin{align*}
	||(e^{t \mathcal{L}} \theta_{0})-\al||&_{L_t^{\infty} \dot{H}^{\gamma+3 / 2} \cap L_t^{2} \dot{H}^{\gamma+3}} \le (1+2c^{3/2})||\theta_{0}-\al||_{\dot{H}^{\gamma+3 / 2}}, \\ 
	||\partial_{t} \theta||_{L_t^{2} \dot{H}^{\gamma}} \le& \frac{c^{3/2}}{2c_0^3}||\theta_{0}-\al||_{\dot{H}^{\gamma+3 / 2}}+C||g||_{L_t^{2} \dot{H}^{\gamma}}.
\end{align*}
\end{lemma}
\begin{lemma}\label{lem:exy}
	Assume $T>0$, let $g\in L^2_Th^\gamma(\bbT)$. The model equation
	\begin{align*}
		y_{st}=\mathfrak{L}(y_s)(s,t)+g(s,t),\  y_s(s,0)=y_{0s}(s),\ s\in\bbT,t\ge0,
	\end{align*}
	admits a unique solution $y_s\in L^\infty_{T}h^{\gamma+\frac{1}{2}}(\bbT)\cap L^2_{T} h^{\gamma+1}(\bbT)$ with $y_{st}\in L^2_{T} h^\gamma(\bbT)$. Furthermore, this solution satisfies
\begin{align*}
	||y_s||_{\dot{h}^{\gamma+1/2}}^2(t) \le e^{-t/4}||y_{0s}||_{\dot{h}^{\gamma+1/2}}^2+4||g||_{L_{[0,t]}^{2}{h}^{\gamma}}^2,\ \forall t\in[0,T],\\
	\frac{d}{dt}||y_s||_{\dot{h}^{\gamma+1/2}}^{2}(t)\le-\frac{1}{4}||y_s||_{\dot{h}^{\gamma+1}}^2(t)+4||g||_{\dot h^\gamma}^{2}(t),\ \forall t\in[0,T].
\end{align*}
It also holds that,
\begin{align*}
	||y_s||_{L_t^{\infty} \dot{h}^{\gamma+1/2} \cap L_t^{2} \dot{h}^{\gamma+1}} \le 3||y_{0s}||_{\dot{H}^{\gamma+3 / 2}}+6||g||_{L_t^{2} \dot{h}^{\gamma}},\\
	\begin{array}{l}
	||e^{t \mathfrak{L}} y_{0s}||_{L_t^{\infty} \dot{h}^{\gamma+1/2} \cap L_t^{2} \dot{h}^{\gamma+1}} \le 3||y_{0s}||_{\dot{h}^{\gamma+1/2}}, \quad||y_{st}||_{L_t^{2} \dot{h}^{\gamma}} \le \frac{1}{2}||y_{0s}||_{\dot{h}^{\gamma+1/2}}+||g||_{L_t^{2} \dot{h}^{\gamma}}.
	\end{array}
\end{align*}
\end{lemma}
\section{Quantitative Isoperimetric Inequalities}
\begin{lemma}\label{lem:iso-gage}(Gage’s isoperimetric inequality \cite{Ga})
	If $\vz$ is a closed, convex, $C^2$ string which satisfies \eqref{con:area1}, it holds that
	\begin{align*}
		\frac{1}{2}\int_\bbT\kappa^2d\al\ge\pi,
	\end{align*}
	where $\kappa$ is the curvature of the string.
\end{lemma}
\begin{lemma}\label{lem:dif-s}
	Given two closed string
	\begin{align*}
		\vz_1(\al)=\fs_1\int_{-\pi}^\al\big(\cos(\theta_1(\al')),\sin(\theta_1(\al'))\big)d\al',\ \vz_2(\al)=\fs_2\int_{-\pi}^\al\big(\cos(\theta_2(\al')),\sin(\theta_2(\al'))\big)d\al',
	\end{align*}
	if the areas enclosed by $\vz_1$ and $\vz_2$ are both $\mathfrak a$, it holds that
	\begin{align*}
		\frac{\fs_1^2-\fs_2^2}{\fs_1^2\fs_2^2}\le C||\theta_1-\theta_2||_{L^2},
	\end{align*}
	where $C$ is a constant depends only on $\mathfrak a$.
\end{lemma}
\begin{proof}
	From the assumption we have
	\begin{align*}
		\mathfrak a=&-\frac{1}{2}\int_\bbT\vz_1\cdot\vz_{1\al}^\perp d\al=-\frac{1}{2}\int_\bbT\vz_2\cdot\vz_{2\al}^\perp d\al\\
		=&-\frac{\fs_1^2}{2}\int^{\pi}_{-\pi}\Big(-\int^{\al}_{-\pi}\cos(\theta_1(\al'))d\al'\sin(\theta_1(\al))+\int^{\al}_{-\pi}\sin(\theta_1(\al'))d\al'\cos(\theta_1(\al)) \Big)d\al\\
		=&-\frac{\fs_2^2}{2}\int^{\pi}_{-\pi}\Big(-\int^{\al}_{-\pi}\cos(\theta_2(\al'))d\al'\sin(\theta_2(\al))+\int^{\al}_{-\pi}\sin(\theta_2(\al'))d\al'\cos(\theta_2(\al)) \Big)d\al\\
		=&\fs_1^2\int^{\pi}_{-\pi}\int^{\al}_{-\pi}\cos(\theta_1(\al'))d\al'\sin(\theta_1(\al))d\al=\fs_2^2\int^{\pi}_{-\pi}\int^{\al}_{-\pi}\cos(\theta_2(\al'))d\al'\sin(\theta_2(\al))d\al.
	\end{align*}
	It follows that
	\begin{align*}
		\frac{\fs_1^2-\fs_2^2}{\fs_1^2\fs_2^2}\mathfrak a=&\int^{\pi}_{-\pi}\int^{\al}_{-\pi}\big(\cos(\theta_2(\al'))-\cos(\theta_1(\al'))\big)d\al'\sin(\theta_2(\al))d\al\\
		&+\int^{\pi}_{-\pi}\int^{\al}_{-\pi}\cos(\theta_1(\al'))d\al'\big(\sin(\theta_2(\al))-\sin(\theta_1(\al))\big) d\al.
	\end{align*}
	Then, it is easy to show the result of this lemma.
\end{proof}
\begin{lemma}\label{lem:con-S-th}(Fuglede’s isoperimetric inequality)
	Let $(\theta,\fs)$ be the tangent angle function and perimeter of a closed string which satisfies \eqref{con:area1}, then there exists $\e>0$ such that if in addition $||\theta-\al||_{\dot H^1}\le \e$, it holds that
	\begin{align*}
		\frac{1}{C}(\fs-1)\le||\theta-\bar\theta-\al||_{L^2}^2\le C(\fs-1),
	\end{align*}
	where $C$ is a constant.
\end{lemma}
\begin{proof}
	Without loss of generality, we only consider the case $\bar\theta=0$. From Lemma \ref{lem:dif-s}, we know that
	\begin{align*}
		\frac{\fs^2-1}{\fs^2}\pi=&\int^{\pi}_{-\pi}\int^{\al}_{-\pi}\big(\cos(\al')-\cos(\theta(\al'))\big)d\al'\sin(\al)d\al\\
		&+\int^{\pi}_{-\pi}\int^{\al}_{-\pi}\cos(\theta(\al'))d\al'\big(\sin(\al)-\sin(\theta(\al))\big) d\al.
	\end{align*}
	Using Taylor expansion, we have
	\begin{align*}
		\sin(\theta)=\sin(\al)+D\cos(\al)-\frac{D^2}{2}\sin(\al)+O(D^3),\\
		\cos(\theta)=\cos(\al)-D\sin(\al)-\frac{D^2}{2}\cos(\al)+O(D^3),
	\end{align*}
	where $D$ stands for $\big(\theta(\al)-\al\big)$. Therefore, one can deduce that
	\begin{align*}
		\frac{\fs^2-1}{\fs^2}\pi=&\frac{1}{2}\int^{\pi}_{-\pi}D^2(\al)d\al+\int^{\pi}_{-\pi}\int^{\al}_{-\pi}D(\al')\sin(\al')d\al'D(\al)\cos(\al)d\al+R.
	\end{align*}
	Here $R$ is the higher order error term and satisfies $R\le C\e ||D||_{L^2}^2$.

	Using Fourier expansion, we have
	\begin{align*}
		&\int^{\al}_{-\pi}D(\al')\sin(\al')d\al'-\frac{\al}{2\pi}\int^{\pi}_{-\pi}D(\al')\sin(\al')d\al'\\
		=&\sum_{k\in \mathbb Z/0}\frac{e^{ik\al}}{4k\pi}\big(a_{k+1}(D)-a_{k-1}(D)-ib_{k+1}(D)+ib_{k-1}(D)\big)+\frac{1}{2\pi}\int_{-\pi}^{\pi}D\sin(\al)-\al D\sin(\al)d\al,\\
		D&(\al)\cos(\al)=\sum_{k\in \mathbb Z/0}\frac{e^{ik\al}}{4\pi}\big(a_{k+1}(D)+a_{k-1}(D)-ib_{k+1}(D)-ib_{k-1}(D)\big)+\frac{1}{2\pi}\int_{-\pi}^{\pi}D\cos(\al)d\al.
	\end{align*}
	 Therefore, with the help of Lemma \ref{lem:normD}, it holds that
	\begin{align*}
		&\int^{\pi}_{-\pi}\int^{\al}_{-\pi}D(\al')\sin(\al')d\al'D(\al)\cos(\al)d\al\\
		=&\int^{\pi}_{-\pi}\Big(\int^{\al}_{-\pi}D(\al')\sin(\al')d\al'-\frac{\al}{2\pi}\int^{\pi}_{-\pi}D(\al')\sin(\al')d\al'\Big)D(\al)\cos(\al)d\al\\
		&+\frac{1}{2\pi}\int_{-\pi}^{\pi}D(\al')\sin(\al')d\al'\int_{-\pi}^{\pi}\al D(\al)\cos(\al)d\al\\
		=&\frac{1}{2\pi}\int_{-\pi}^{\pi}D\sin(\al)-\al D\sin(\al)d\al\int_{-\pi}^{\pi}D\cos(\al)d\al+\frac{1}{2\pi}\int_{-\pi}^{\pi}D\sin(\al)d\al\int_{-\pi}^{\pi}\al D\cos(\al)d\al\\
		\le&C||D||_{L^2}\big(|a_1(D)|+|b_1(D)|\big)\le C\e||D||_{L^2}^2.
	\end{align*}
	This actually finishes the proof.
\end{proof}
This Lemma is a weak version of 2-dimensional Fuglede's isoperimetric inequality. For general results, we refer the readers to \cite{Fug,Fu}.
\section{Auxiliary Calculations}
\begin{lemma}
	Given $\big(\theta(\bar\theta,\al),y_s(s),\mathfrak s\big)$, let $\vu$, $\vz$, $\vn$, and $\vt$ be the functions defined in \eqref{eq:ual}, \eqref{eq:z}, and \eqref{eq:direc}, it holds that
	\begin{align*}
		\frac{d}{d\bth}\Big((\vu\cdot\vn)(\bar\theta,\al)\Big)=\frac{d}{d\bth}\Big((\vu\cdot\vt)(\bar\theta,\al)\Big)=0.
	\end{align*}
\end{lemma}
\begin{proof}
	By the definition, one has
	\begin{align*}
		\cos(\theta)=\cos(\rth+\bth),\ \sin(\theta)=\sin(\rth+\bth),\ \frac{d}{d\bth}\cos(\theta)=-\sin(\theta),\ \frac{d}{d\bth}\sin(\theta)=\cos(\theta).
	\end{align*}
	It follows that
	\begin{align*}
		&\frac{d}{d\bth}|\vz(\bth,\al)-\vz(\bth,\al')|^2\\
		=&\frac{2}{\mathfrak s^2}\int^\al_{\al'}\big(\cos(\theta(\al'')),\sin(\theta(\al''))\big)d\al''\cdot \int^\al_{\al'}\big(-\sin(\theta(\al'')),\cos(\theta(\al''))\big)d\al''\\
		=&0.
	\end{align*}
	Therefore, we deduce that
	\begin{align*}
		&4\pi\frac{d}{d\bth}G\big(\vz(\bth,\al)-\vz(\bth,\al')\big)\\
		=&\frac{\int^\al_{\al'}\big(-\sin(\theta(\al'')),\cos(\theta(\al''))\big)d\al''\otimes \int^\al_{\al'}\big(\cos(\theta(\al'')),\sin(\theta(\al''))\big)d\al''}{|\int^\al_{\al'}(\cos(\theta),\sin(\theta))d\al''|^2}\\
		&+\frac{\int^\al_{\al'}\big(\cos(\theta(\al'')),\sin(\theta(\al''))\big)d\al''\otimes \int^\al_{\al'}\big(-\sin(\theta(\al'')),\cos(\theta(\al''))\big)d\al''}{|\int^\al_{\al'}(\cos(\theta),\sin(\theta))d\al''|^2}.
	\end{align*}
	For any scalar function $f(\al)\in C(\bbT)$ which is independent of $\bth$, one can see
	\begin{align*}
		&4\pi\frac{d}{d\bth}\Big(\vn(\al)\cdot\int_\bbT G(\vz(\bth,\al)-\vz(\bth,\al'))f\cdot\vn(\bth,\al')d\al'\Big)\\
		=&-4\pi\int_\bbT f(\al')\Big(\vt(\bth,\al)\cdot G(\vz(\bth,\al)-\vz(\bth,\al'))\cdot\vn(\bth,\al')\\
		&\qquad\qquad\qquad\qquad+\vn(\bth,\al)\cdot G(\vz(\bth,\al)-\vz(\bth,\al'))\cdot\vt(\bth,\al')\Big)d\al'\\
		&+\vn(\bth,\al)\cdot\int_\bbT\frac{\int^\al_{\al'}\big(-\sin(\theta),\cos(\theta)\big)d\al''\otimes \int^\al_{\al'}\big(\cos(\theta),\sin(\theta)\big)d\al''}{|\int^\al_{\al'}(\cos(\theta),\sin(\theta))d\al''|^2}f(\al')\cdot\vn(\bth,\al')d\al'\\
		&+\vn(\bth,\al)\cdot\int_\bbT\frac{\int^\al_{\al'}\big(\cos(\theta),\sin(\theta)\big)d\al''\otimes \int^\al_{\al'}\big(-\sin(\theta),\cos(\theta)\big)d\al''}{|\int^\al_{\al'}(\cos(\theta),\sin(\theta))d\al''|^2}f(\al')\cdot\vn(\bth,\al')d\al'\\
		=&\int_\bbT f(\al')\ln|\vz(\bth,\al)-\vz(\bth,\al')|\big(-\cos(\theta(\al))\sin(\theta(\al'))+\sin(\theta(\al))\cos(\theta(\al'))\\
		&\qquad\qquad\qquad\qquad\qquad\qquad-\sin(\theta(\al))\cos(\theta(\al'))+\cos(\theta(\al))\sin(\theta(\al'))\big)d\al'\\
		&+\int_\bbT f(\al')\frac{\Big(\big(\int^\al_{\al'}\cos(\theta)d\al''\big)^2-\big(\int^\al_{\al'}\sin(\theta)d\al''\big)^2\Big)\sin\big(\theta(\al)+\theta(\al')\big)}{|\int^\al_{\al'}(\cos(\theta),\sin(\theta))d\al''|^2}d\al'\\
		&-\int_\bbT f(\al')\frac{2\big(\int^\al_{\al'}\cos(\theta)d\al''\big)\big(\int^\al_{\al'}\sin(\theta)d\al''\big)\Big)\cos\big(\theta(\al)+\theta(\al')\big)}{|\int^\al_{\al'}(\cos(\theta),\sin(\theta))d\al''|^2}d\al'\\
		&-\int_\bbT f(\al')\frac{\Big(\big(\int^\al_{\al'}\cos(\theta)d\al''\big)^2-\big(\int^\al_{\al'}\sin(\theta)d\al''\big)^2\Big)\sin\big(\theta(\al)+\theta(\al')\big)}{|\int^\al_{\al'}(\cos(\theta),\sin(\theta))d\al''|^2}d\al'\\
		&+\int_\bbT f(\al')\frac{2\big(\int^\al_{\al'}\cos(\theta)d\al''\big)\big(\int^\al_{\al'}\sin(\theta)d\al''\big)\Big)\cos\big(\theta(\al)+\theta(\al')\big)}{|\int^\al_{\al'}(\cos(\theta),\sin(\theta))d\al''|^2}d\al'\\	
		=&0.
	\end{align*}
	For the same reason, it holds that
	\begin{align*}
		&\frac{d}{d\bth}\Big(\vt(\bth,\al)\cdot\int_\bbT G(\vz(\bth,\al)-\vz(\bth,\al'))f\cdot\vn(\bth,\al')d\al'\Big)=0,\\
		&\frac{d}{d\bth}\Big(\vn(\bth,\al)\cdot\int_\bbT G(\vz(\bth,\al)-\vz(\bth,\al'))f\cdot\vt(\bth,\al')d\al'\Big)=0,\\
		&\frac{d}{d\bth}\Big(\vt(\bth,\al)\cdot\int_\bbT G(\vz(\bth,\al)-\vz(\bth,\al'))f\cdot\vt(\bth,\al')d\al'\Big)=0.
	\end{align*}
	Thus, by \eqref{eq:ual} we arrive at
	\begin{align*}
		\frac{d}{d\bth}\Big(\vu(\vz(\bth,\al))\cdot\vn(\bth,\al)\Big)=0,\ \frac{d}{d\bth}\Big(\vu(\vz(\bth,\al))\cdot\vt(\bth,\al)\Big)=0.
	\end{align*}
\end{proof}
\end{appendix}
\section*{Acknowledgment}
The author wishes to express his thanks to Prof. Zhifei Zhang and Prof. Wei Wang for suggesting the problem and for many helpful discussions.


\begin{thebibliography}{99}
\bibitem{Am} D. M. Ambrose, {\it Well-posedness of vortex sheets with surface tension}. Doctoral dissertation, Duke University, 2002. 
\bibitem{AS} D. M. Ambrose, M. Siegel, {\it Well-posedness of two-dimensional hydroelastic waves}. Proc. Roy. Soc. Edinburgh Sect. A. 147 (2015), 529–570.
\bibitem{BCD} H. Bahouri, J. Chemin, R. Danchin, {\it Fourier Analysis and Nonlinear Partial Differential Equations}. Springer-Verlag, Berlin Heidelberg, 2011.
\bibitem{BHL} J. T. Beale, T. Y. Hou, J. S. Lowengrub, {\it Growth rates for the linearized motion of fluid interfaces away from equilibrium}. Comm. Pure Appl. Math. 46 (1993), 1269–1301.
\bibitem{CL} W. Cai, T. Lubensky, {\it Hydrodynamics and dynamic fluctuations of fluid membranes}. Phys. Rev. E. 52 (1995), 4251–4266. 
\bibitem{CGM} S. Canic, M. Galic, B. Muha, {\it Analysis of a 3D Nonlinear, Moving Boundary Problem describing Fluid-Mesh-Shell Interaction}. arXiv:1911.09927.
\bibitem{CG} R. Capovilla, J. Guven, {\it Stresses in lipid membranes}.J. Phys. A: Math. Gen. 35 (2002), 6233–6247. 
\bibitem{CCS} C. H. Cheng, D. Coutand, S. Shkoller, {\it Navier-Stokes equations interacting with a nonlinear elastic biofluid shell}. SIAM J. Math. Anal. 39 (2007), 742–800.
\bibitem{CS2} D. Coutand, S. Shkoller, {\it The interaction of the 3D Navier-Stokes equations with a moving nonlinear Koiter elastic shell}. SIAM J. Math. Anal. 42 (2010), 1094–1155.
\bibitem{DH} H. J. Deuling, W. Helfrich, {\it The curvature elasticity of fluid membranes : A catalogue of vesicle shapes}. J. Phys. 37(1976), 1335-1345.
\bibitem{Fug} B. Fuglede, {\it Stability in the isoperimetric problem for convex or nearly spherical domains in Rn}. T. Am. Math. Soc. 315(1989), 619–638.
\bibitem{Fu} N. Fusco, {\it The quantitative isoperimetric inequality and related topics}. Bull. Math. Sci. 5(2015), 517–607.
\bibitem{Ga} M. E. Gage, {\it An isoperimetric inequality with applications to curve shortening}. Duke Math. J. 50 (1983), 1225–1229.
\bibitem{GH} C. Grandmont, M. Hillairet, {\it Existence of global strong solutions to a beam-fluid interaction system}. Arch. Ration. Mech. Anal. 220 (2016), 1283–1333.
\bibitem{He} W. Helfrich, {\it Elastic properties of lipid bilayers: theory and possible experiments}. Zeitschrift f$\mathrm{\ddot{u}}$r Naturforschung C, 28 (1973), 693-703.
\bibitem{HLS} T. Y. Hou, J. Lowengrub, M. Shelley, {\it Removing the stiffness from interfacial flows with surface tension}. J. Comput. Phys., 114 (1994), 312–338.
\bibitem{HO} J. Hu, Z. Ouyang, {\it Shape equations of the axisymmetric vesicles
}. Phys. Rev. E, 47(1993), 461-467. 
\bibitem{HSZ} D. Hu, P. Song, P. Zhang, {\it Local existence and uniqueness of the dynamical equations of an incompressible membrane in two-dimensional space}. Commun. Math. Sci. 8 (2010), 783–796. 
\bibitem{HZE} D. Hu, P. Zhang, W. E, {\it Continuum theory of a moving membrane}. Phys. Rev. E, 75(2007), 041605, 11 pp.
\bibitem{Ko} W. T. Koiter, {\it On the nonlinear theory of thin elastic shells}. Proc. K. Ned. Akad. Wet. B, 69(1966), 1–54. 
\bibitem{LL} M. Lai, Z. Li, {\it The immersed interface method for the Navier-Stokes equations with singular forces}. J. Comput. Phys. 171 (2001), 822–842.
\bibitem{LR} D. Lengeler, M. Ruzicka, {\it Weak solutions for an incompressible Newtonian fluid interacting with a Koiter type shell}. Arch. Ration. Mech. Anal. 211 (2014), 205–255.
\bibitem{Le} G. Leoni, {\it A first course in Sobolev spaces, 2nd}. Graduate studies in Mathematics. American Mathematical Society, Rhode Island, 2017.
\bibitem{LT} F. Lin, J. Tong, {\it Solvability of the Stokes Immersed Boundary Problem in Two Dimensions}. Comm. Pure Appl. Math., 72 (2018), 159-226.
\bibitem{Lip} R. Lipowsky, {\it Coupling of bending and stretching deformations in vesicle membranes}. Adv. Colloid Interface Sci. , 208 (2014), 14-24.
\bibitem{LA} S. Liu, D. M. Ambrose, {\it Well-posedness of two-dimensional hydroelastic waves with mass}. J. Differential Equations, 262 (2017), 4656-4699. 
\bibitem{MP} B. V. Matioc, G. Prokert, {\it Two-phase Stokes flow by capillarity in full 2D space: an approach via hydrodynamic potentials}, arXiv:2003.14010.
\bibitem{MRS} Y. Mori, A. Rodenberg, D. Spirn, {\it Well-posedness and global behavior of the Peskin problem of an immersed elastic filament in Stokes flow}.Comm. Pure Appl. Math. 72 (2019), 887–980.
\bibitem{MC} B. Muha, S. Canic, {\it Existence of a weak solution to a nonlinear fluid-structure interaction problem modeling the flow of an incompressible, viscous fluid in a cylinder with deformable walls}. Arch. Ration. Mech. Anal. 207 (2013), 919–968.
\bibitem{NOO} H. Naito, M. Okuda, Z. Ouyang, {\it Counterexample to some shape equations for axisymmetric vesicles}. Phys. Rev. E, 48(1993), 2304-2307. 
\bibitem{Ou} Z. Ouyang, {\it Anchor ring-vesicle membranes
}.Phys. Rev. A 41 (1990), 4517–4520. 
\bibitem{OH} Z. Ouyang, W. Helfrich, {\it Instability and Deformation of a Spherical Vesicle by Pressure
}. Phys. Rev. Lett. 59 (1987) 2486-2488.
\bibitem{Pe} C. S. Peskin, {\it Flow patterns around heart valves: a numerical method}. J. Comput. Phys. 10 (1972), 252–271.
\bibitem{Po} C. Pozrikidis, {\it Boundary integral and singularity methods for linearized viscous flow}. Cambridge Texts in Applied Mathematics. Cambridge University Press, Cambridge, 1992.
\bibitem{Po2} C. Pozrikidis, {\it Modeling and Simulation of Capsules and Biological Cells}. CRC Press, London, 2003.
\bibitem{PS1} J. Pr\"uss, G. Simonett, {\it Analysis of the boundary symbol for the two-phase Navier-Stokes equations with surface tension. Nonlocal and abstract parabolic equations and their applications}. Banach Center Publications, 86. Polish Acad. Sci. Inst. Math., Warsaw, 2009, 265–285.
\bibitem{PS2} J. Pr\"uss, G. Simonett, {\it Moving interfaces and quasilinear parabolic evolution equations}. Birkh\"auser/Springer, 2016.
\bibitem{St} D. J. Steigmann, {\it Fluid films with curvature elasticity}. Arch. Ration. Mech. Anal. 150 (1999), 127–152.
\bibitem{Ta} M. E. Taylor, {\it Partial Differential Equations I, 2nd}. Springer-Verlag, New York, 2011.
\bibitem{To} J. Tong, {\it Regularized stokes immersed boundary problems in two dimensions: well-posedness, singular limit, and error estimates}. arXiv:1904.09528.
\bibitem{WZZ} W. Wang, P. Zhang, Z. Zhang, {\it Well-posedness of hydrodynamics on the moving elastic surface}. Arch. Ration. Mech. Anal. 206 (2012), 953–995.
\bibitem{Wa} A. M. Waxman, {\it Dynamics of a couple-stress fluid membrane}. Stud. Appl. Math. 70 (1984),
63–86.
\end{thebibliography}
\end{document}